\documentclass[UTF-8,reqno]{amsart}
\usepackage{enumerate}
\usepackage{mhequ}
\usepackage{amssymb,url,color, booktabs,nccmath}
\usepackage[left=3.2cm,right=3.2cm,top=4cm,bottom=4cm]{geometry}
\usepackage{mathrsfs}
\usepackage{enumitem,dsfont}

\usepackage{color}
\usepackage[colorlinks=true]{hyperref}
\hypersetup{
    linkcolor=blue,          
    citecolor=red,        
    filecolor=blue,      
    urlcolor=cyan
}

\definecolor{darkergreen}{rgb}{0.0, 0.5, 0.0}

\numberwithin{equation}{section}
\def\theequation{\arabic{section}.\arabic{equation}}
\newcommand{\be}{\begin{eqnarray}}
\newcommand{\ee}{\end{eqnarray}}
\newcommand{\ce}{\begin{eqnarray*}}
\newcommand{\de}{\end{eqnarray*}}
\newtheorem{theorem}{Theorem}[section]
\newtheorem{lemma}[theorem]{Lemma}
\newtheorem{proposition}[theorem]{Proposition}
\newtheorem{Examples}[theorem]{Example}
\newtheorem{corollary}[theorem]{Corollary}

\newtheorem{definition}[theorem]{Definition}
\theoremstyle{definition}
\newtheorem{remark}[theorem]{Remark}

\newcommand{\assign}{:=}
\newcommand{\cdummy}{\cdot}
\newcommand{\mathd}{\mathrm{d}}

\newcommand{\tmmathbf}[1]{\ensuremath{\mathbf{#1}}}
\newcommand{\tmop}[1]{\ensuremath{\operatorname{#1}}}
\newcommand{\tmtextit}[1]{{\itshape{#1}}}

\DeclareMathOperator{\supp}{supp}

\def\T{\mathbb{T}^{3}}

\def\u{\mathbf{u}}

\def\[{{\Big[}}
\def\]{{\Big]}}
\def\<{{\langle}}
\def\>{{\rangle}}
\def\({{\Big(}}
\def\){{\Big)}}

\def\bx{{\mathbf{x}}}
\def\tr{\mathrm {tr}}

\def\dif{{\mathord{{\rm d}}}}

\def\no{\nonumber}
\def\={&\!\!=\!\!&}
\DeclareMathOperator*{\esssup}{esssup}

\def\mR{{\mathbb R}}

\def\mT{{\mathbb T}}

\def\mX{{\mathbb X}}

\def\1{{\mathbf{1}}}

\def\sC{{\mathscr C}}

\def\sP{{\mathscr P}}

\def\E{\mathbf E}

\def\geq{\geqslant}
\def\leq{\leqslant}

\def\div{\mathord{{\rm div}}}

\def\u{\mathbf{u}}

\def\[{{\Big[}}
\def\]{{\Big]}}
\def\<{{\langle}}
\def\>{{\rangle}}
\def\({{\Big(}}
\def\){{\Big)}}

\def\bx{{\mathbf{x}}}
\def\tr{\mathrm {tr}}

\def\dif{{\mathord{{\rm d}}}}

\def\no{\nonumber}
\def\={&\!\!=\!\!&}
\def\bt{\begin{theorem}}
\def\et{\end{theorem}}
\def\bl{\begin{lemma}}
\def\el{\end{lemma}}
\def\br{\begin{remark}}
\def\er{\end{remark}}
\def\bx{\begin{Examples}}
\def\ex{\end{Examples}}
\def\bd{\begin{definition}}
\def\ed{\end{definition}}
\def\bp{\begin{proposition}}
\def\ep{\end{proposition}}
\def\bc{\begin{corollary}}
\def\ec{\end{corollary}}

\def\geq{\geqslant}
\def\leq{\leqslant}

\def\div{\mathord{{\rm div}}}

 \def\R{\mathbb R}
 \def\R{\mathbb R}    
\def\N{\mathbb N}  
   
\def\<{\langle} \def\>{\rangle}

\allowdisplaybreaks

\begin{document}

\title[On ill- and well-posedness  to stochastic 3D Euler equations]{On ill- and well-posedness of dissipative martingale solutions to stochastic 3D Euler equations}

\author{Martina Hofmanov\'a}
\address[M. Hofmanov\'a]{Fakult\"at f\"ur Mathematik, Universit\"at Bielefeld, D-33501 Bielefeld, Germany}
\email{hofmanova@math.uni-bielefeld.de}

\author{Rongchan Zhu}
\address[R. Zhu]{Department of Mathematics, Beijing Institute of Technology, Beijing 100081, China; Fakult\"at f\"ur Mathematik, Universit\"at Bielefeld, D-33501 Bielefeld, Germany}
\email{zhurongchan@126.com}

\author{Xiangchan Zhu}
\address[X. Zhu]{ Academy of Mathematics and Systems Science,
Chinese Academy of Sciences, Beijing 100190, China; Fakult\"at f\"ur Mathematik, Universit\"at Bielefeld, D-33501 Bielefeld, Germany}
\email{zhuxiangchan@126.com}
\thanks{Supported in part by NSFC (No. 11671035, No. 11771037, No. 11922103). Financial support by the DFG through the CRC
1283 ``Taming uncertainty and profiting from randomness and low regularity in analysis, stochastics and their
applications'' and support by key Lab of Random Complex Structures and Data Science, Chinese Academy of Science are gratefully acknowledged.}

\begin{abstract}
We are concerned with the question of  well-posedness of stochastic three dimensional incompressible Euler equations. In particular, we introduce a novel class of dissipative solutions and show that (i) existence; (ii) weak--strong uniqueness; (iii) non-uniqueness in law; (iv) existence of a strong Markov solution; (v) non-uniqueness of strong Markov solutions; all hold true within this class. Moreover, as a byproduct of (iii) we obtain existence and non-uniqueness of probabilistically strong and analytically weak solutions defined up to a stopping time and satisfying an energy inequality.
\end{abstract}


\date{\today}

\maketitle

\tableofcontents

\section{Introduction}

The   mathematics community is kept intrigued by the  questions of ill- and/or well-posedness of equations in fluid dynamics. A substantial progress has been experienced in the past ten years or so, starting from the groundbreaking results by De Lellis and Sz{\'e}kelyhidi~Jr. \cite{DelSze2,DelSze3,DelSze13}. In these works, the method of convex integration was used in order to prove non-uniqueness of weak solutions to  Euler equations. Furthermore, non-uniqueness was established among weak solutions dissipating energy which is one of the well-accepted criteria for the selection of physically relevant solutions. The method was then further developed and successfully applied to a number of other fluid dynamics models. In particular also to the Navier--Stokes system in three dimensions by Buckmaster and Vicol \cite{BV19a} and Buckmaster, Colombo and Vicol \cite{BCV18}, where it permitted to prove non-uniqueness for weak solutions not satisfying the energy inequality, i.e. not Leray solutions. We refer to the excellent review articles by Buckmaster and Vicol \cite{BV19,BV20} for a gentle introduction and further references.

In view of these issues, there has been a hope that a certain stochastic perturbation can provide a regularizing effect of the underlying PDE dynamics. And indeed, some positive results have been achieved. Flandoli and Luo \cite{FL19} showed that a noise of transport type prevents a vorticity blow-up in the Navier--Stokes equations. Flandoli, Hofmanov\'a, Luo and Nilssen \cite{FHLN20} then showed that the regularization is even provided by deterministic vector fields. A linear multiplicative noise prevents the blow up of the velocity with high probability for  the three dimensional Euler and Navier--Stokes system as well, as shown by Glatt-Holtz and Vicol \cite{GHV14} and R\"ockner, Zhu and Zhu \cite{RZZ14}, respectively. Noise also has a beneficial impact when it comes to long time behavior and ergodicity. Da Prato and Debussche \cite{DD03} obtained a unique ergodicity for three dimensional stochastic Navier--Stokes equations with non-degenerate additive noise. The theory of Markov selections by Flandoli and Romito \cite{FR08} provides an alternative approach which also allowed to prove ergodicity for every Markov solution, see Romito \cite{Ro08}.

Quite the contrary, in our previous work  \cite{HZZ19} we  proved a negative result: non-uniqueness in law holds for the stochastic three dimensional Navier--Stokes equations with additive or linear multiplicative noise in a class of analytically weak solutions. This in particular shows that the noise of \cite{ DD03,FR08, Ro08, RZZ14} mentioned above is of no help for the initial value problem in this solution framework. The proof relies on a stochastic variant of the convex integration method together with a general probabilistic construction developed in order to extend solutions defined up to a stopping time to the whole time interval $[0,\infty)$. This way, the convex integration permits to construct solutions which fail the corresponding energy inequality. The approach of \cite{HZZ19} was  applied by Yamazaki~\cite{Ya20a,Ya20b} to obtain  non-uniqueness in law for stochastic Navier--Stokes equations in a three dimensional  hyperviscous and a two dimensional fractional dissipative setting.

For completeness, let us also mention that prior to \cite{HZZ19} the convex integration has already been applied in a stochastic setting, namely, to the isentropic Euler system by Breit, Feireisl and Hofmanov\'a \cite{BFH20} and to the full Euler system by Chiodaroli, Feireisl and Flandoli \cite{CFF19}. Early works on stochastic Euler equations treat only the two dimensional case (see e.g. \cite{Bes99, BF99, BP01, GHV14, BFM16}). The three-dimensional case has been treated in \cite{MV00, Kim09, GHV14, CFH19}. In particular, Glatt-Holtz and Vicol \cite{GHV14} obtained local well-posedness of strong solutions to stochastic Euler equations in two and three dimensions, global well-posedness in two dimensions for additive and linear multiplicative noise, and  the above mentioned regularization by linear multiplicative noise in the three dimensional setting. Local well-posedness for  three dimensional stochastic compressible Euler equations was proved  by Breit and Mensah \cite{BM19}.

\medskip

In the present paper, we are concerned with stochastic Euler equations governing the time evolution of the velocity $u$ of an inviscid fluid  on the three dimensional torus $\mathbb{T}^3$. The system reads as
\begin{equation}\label{el}
\aligned
 \dif u+\div(u\otimes u) \dif t+\nabla P \dif t&=G(u)\dif B,
\\
\div u&=0,
\endaligned
\end{equation}
where $P$ stands for the corresponding pressure and the right hand side represents a random external force acting on the fluid. It is driven by a cylindrical Wiener process $B$ defined on some probability space and the diffusion coefficient $G$ satisfies suitable assumptions, see Section~\ref{s:not} for more details.

Our goal is to investigate the ill/well-posedness of this system from various perspectives. More precisely, we aim at finding one stable solution concept which provides a  suitable framework to study  the questions of existence, (non-)uniqueness as well as Markov selections. To this end, we introduce the notion of \emph{dissipative martingale solution}. Roughly speaking, it corresponds to measure--valued solutions weak in the probabilistic sense and satisfying a version of an energy inequality. Although measure--valued solutions have been extensively studied  in the deterministic literature (see e.g. \cite{DPM87,Li96}),  their formulation becomes rather challenging in the stochastic setting. A first attempt in the context of the Euler equations was done by Breit and Moyo \cite{BM20}, who also proved existence and weak--strong uniqueness. Nevertheless, it turns out that the question of existence of a Markov selection poses severe restrictions on the definition of solution. Consequently, the framework from \cite{BM20} cannot be used and new ideas are required.

The key insight which we put forward is twofold. On the one hand, we present  a novel formulation of the energy inequality and, on the other hand, we include an energy variable as a part of the solution. This is necessary due to
the presence of the so-called energy sinks which is an intrinsic property of the Euler equations. More precisely, the $L^{2}$-norm of the initial value  itself does not contain all the necessary information on the actual energy in order to restart the system. In this sense, including an additional variable is a legitimate step which reflects the nature of the  equations. The principle ideas behind our definition of solution can be found in Section~\ref{s:p1}, where we refer the reader for more details and further notations.

With this in hand, let us  summarize our main results. The precise formulations can be found in the respective sections and in particular the precise assumptions on the noise coefficient $G$ are also stated there. At this stage, let us only mention that all our results apply to the additive noise case with  a sufficient regularity of $G$, while some results also allow for certain  possibly nonlinear  coefficient $G(u)$.

\begin{enumerate}
\item[(i)] {\bf Existence:} We prove that dissipative martingale solutions exist globally in time for diver\-gence-free initial conditions in $L^{2}$. The proof relies on a compactness argument combined with Jakubowski--Skorokhod's representation theorem. Due to the limited compactness of the Euler system it is necessary to work with dissipative rather than analytically weak solutions. Indeed, the latter one can only be shown to exists for certain initial conditions up to a certain stopping time. See Section~\ref{s:stab} and Section~\ref{s:ex}.
\item[(ii)] {\bf Weak--strong uniqueness:} We show that dissipative martingale solutions satisfy a weak--strong uniqueness principle. More precisely, if for some initial value there is an analytically strong solution defined on the  canonical path space up to a stopping time, then it coincides with all dissipative martingale solutions having the  same initial value. See Section~\ref{s:un}.
\item[(iii)] {\bf Non-uniqueness in law:} We apply the method of convex integration in order to construct infinitely many solutions to the stochastic Euler system which live up to a certain stopping time, are analytically weak, probabilistically strong and satisfy an energy inequality. Using the general probabilistic extension of solutions which we developed in \cite{HZZ19}, we  extend these solutions beyond the stopping time to the whole time interval $[0,\infty)$. Even though analytically weak before the stopping time, these solutions become only dissipative, i.e., measure--valued,  after the stopping time  due to the limitations of the general existence result. See Section~\ref{s:law}.
\item[(iv)] {\bf Existence of a strong Markov solution:} Our notion of dissipative martingale solution permits to select a system of solutions to the stochastic Euler system satisfying the strong Markov property. In addition, the solutions fulfil a version of the principle of maximal energy dissipation proposed  by Dafermos \cite{Daf4} in order to select the physically relevant solutions. Our proof makes use of the abstract Markov selection procedure by Krylov \cite{K73}.  The main idea is to include an additional datum into the selection procedure. See Section~\ref{s:6.2}.
\item[(v)] {\bf Non-uniqueness of strong Markov solutions:} Finally, we combine the result of non-uniqueness in law with the existence of a strong Markov solution and deduce non-uniqueness of  strong Markov selections. See Section~\ref{s:nonM}.
\end{enumerate}

The main contribution of our paper lies in the points (iii), (iv) and (v) and some novelties are also present in the point (ii). In particular, compared to the weak--strong uniqueness in the deterministic setting, the additional difficulties in the stochastic setting originate in the fact that the times of the energy sinks are generally random. Consequently, evaluation of the energy inequality in its usual form at the stopping time, up to which the analytically strong solution lives, becomes delicate. We overcome this issue by using a different form of the relative energy, and accordingly also a different form of the energy inequality in the definition of a solution, with the help of a continuous stochastic process $z$ rather than the kinetic energy itself, which is  defined only a.e. in time.

Regarding the point (iv), we recall that applications of Krylov's Markov selection to SPDEs can be found in \cite{FR08, GRZ09, BFH18markov}. In particular, Flandoli and Romito \cite{FR08} introduced a weaker notion of Markov property, the so-called almost sure Markov property. It means that the Markov property holds up to an exceptional set of deterministic times in $(0,\infty)$ having zero Lebesgue measure. They were able to show that the stochastic Navier--Stokes equations admit an almost sure Markov solution. The same issue appears for the stochastic compressible Navier--Stokes system in  \cite{BFH18markov} as well as for the models  in \cite{GRZ09}.
However, in these works it was not possible to obtain the usual Markov property, let alone the strong Markov property.

Needless to say that the stochastic Euler equations \eqref{el} represent a number of additional  difficulties when it comes to the construction of Markov solutions. More precisely, due to the presence of random energy sinks together with the limited compactness the method of \cite{FR08} does not apply. Our approach not only works for the Euler equations, it even permits
to obtain the strong Markov property.
This is precisely where  the process $z$ as part of the solution plays an essential role. It gives the necessary control of the kinetic energy of the system after solution trajectories are shifted in time, required by the so-called disintegration property. We note that the same ideas can also be applied in the (simpler) settings of \cite{FR08, GRZ09, BFH18markov} in order to construct strong Markov selections.

The core of the non-uniqueness in law in (iii) is the method of convex integration. Unlike in our previous work \cite{HZZ19} where the convex integration relied on an iteration procedure, we rely here on a Baire category argument by De Lellis and Sz{\'e}kelyhidi~Jr.~\cite{DelSze3}. The first application of this method in the stochastic setting was done in \cite{BFH20}. In the present paper, we extend the stochastic approach further in several aspects. In particular, we present two new versions of oscillatory lemmas, see Lemma~\ref{lem:osc} and Lemma~\ref{lem:oscs}. The former one permits to construct a subsolution with a prescribed energy at time $t=0$. That is, we are able to eliminate the initial jump of the energy present in \cite{BFH20}, which is essential in order to obtain solutions satisfying the energy inequality and weak--strong uniqueness principle.
The latter oscillatory lemma is applied in order to enforce the energy inequality at a given stopping time. In order to combine these two, we need to construct suitable energy function by applying the theory of Young's integration and a control of the iterated stochastic integral of the Wiener process before a stopping time, see Lemma~\ref{lem:osc2}.
This way we are able to construct infinitely many analytically weak solutions before the stopping time satisfying  a.e. in time the usual energy inequality  with a prescribed energy defect. Moreover, the prescribed energy depends on the solution itself.

Furthermore, our convex integration solutions are probabilistically strong, i.e. adapted to the~given Wiener process. This is the key property needed  to extend these solutions as dissipative martingale solutions beyond the stopping time and to deduce the non-uniqueness in law. Compared to the Navier--Stokes setting from \cite{HZZ19} where the energy inequality could not be fulfilled  by the convex integration solutions, this point requires a careful treatment in the extension of solutions in Section~\ref{s:appl}. Since we have to control the iterated stochastic integral of the Wiener process before the stopping time as mentioned above,
we also need to define the corresponding stopping time on the canonical path space. But now the  difficulty lies in how to define the iterated stochastic integral on the  path space without the use of any probability measure. Indeed, due to the low time regularity of the Wiener process,  the stochastic integral cannot be defined by purely analytical tools and probability theory is required  in a nontrivial way. We overcome this issue by using the energy equality obtained in the convex integration step to identify the corresponding stochastic integral. Applying the theory of Young's integration we are able to define the necessary stopping time on the~path space and to transfer the convex integration solutions to the path space where they can be extended to dissipative martingale solutions on $[0,\infty)$.

To conclude this introductory part, we note that the Markov property corresponds to the semiflow property in the deterministic setting. In other words, as a simple observation, our results translated to the deterministic setting imply in particular non-uniqueness of the associated semiflow, which to the best of our knowledge has not been known before, see Remark~\ref{semiflow} for more details.

\section{Notations}

\subsection{Function spaces}

  Throughout the paper, we use the notation $a\lesssim b$ if there exists a constant $c>0$ such that $a\leq cb$, and we write $a\simeq b$ if $a\lesssim b$ and $b\lesssim a$. Given a Banach space $E$ with a norm $\|\cdot\|_E$ and $T>0$, we write $C_TE=C([0,T];E)$ for the space of continuous functions from $[0,T]$ to $E$, equipped with the supremum norm $\|f\|_{C_TE}=\sup_{t\in[0,T]}\|f(t)\|_{E}$. We also use $CE$ to denote the space of $C([0,\infty);E)$.  For $\alpha\in(0,1)$ we  define $C^\alpha_TE$ as the space of $\alpha$-H\"{o}lder continuous functions from $[0,T]$ to $E$, endowed with the seminorm $\|f\|_{C^\alpha_TE}=\sup_{s,t\in[0,T],s\neq t}\frac{\|f(s)-f(t)\|_E}{|t-s|^\alpha}.$ We denote by $L^p_{T}E$ the set of  $L^p$-integrable  functions from $[0,T]$ to $E$.  For $\alpha\in (0,1)$, $ p\in [1,\infty)$, we  define $W^{\alpha,p}_TE$ as the Sobolev space of all $f\in L^p_TE$ such that $\int_0^T\int_0^T\frac{\|f(t)-f(s)\|_E^p}{|t-s|^{1+\alpha p}}\dif t\dif s<\infty$ endowed with the norm
  $\|f\|_{W^{\alpha,p}_TE}^p:=\int_0^T\|f(t)\|_E^p\dif t+\int_0^T\int_0^T\frac{\|f(t)-f(s)\|_E^p}{|t-s|^{1+\alpha p}}\dif t\dif s$.
  We also use $C^\alpha_{\textrm{loc}}([0,\infty);E)$  and $W^{\alpha,p}_{\textrm{loc}}([0,\infty);E)$, respectively, to denote the space of  functions $f$ satisfying for every $T>0$ $f|_{[0,T]}\in C^\alpha_TE$  and $f|_{[0,T]}\in W^{\alpha,p}_TE$, respectively.

 We use $L^p$ to denote the set of the standard $L^p$-integrable functions from $\mathbb{T}^3$ to $\mathbb{R}^3$.  Set $L^{2}_{\sigma}=\{u\in L^2, \div u=0\}$. For $s>0$, $p>1$ we set $W^{s,p}:=\{f\in L^p; \|(I-\Delta)^{\frac{s}{2}}f\|_{L^p}<\infty\}$ with the norm $\|f\|_{W^{s,p}}=\|(I-\Delta)^{\frac{s}{2}}f\|_{L^p}$. For $s>0$, $H^s:=W^{s,2}\cap L^2_\sigma$. For $s<0$ define $W^{s,q}$ to be the dual space of $W^{-s,p}$ with $\frac{1}{p}+\frac{1}{q}=1$. Let $\{e_{i}\}_{i\in\mathbb{N}}$ be a complete orthonormal system in $L^{2}_{\sigma}$. For a domain $D$ we use $\mathcal{D}'(D)$ to denote the dual space of $C_c^\infty(D)$.

\subsection{Noise}\label{s:not}

For a Hilbert space $U$ let $L_2(U,L^2_\sigma)$ be the space all Hilbert--Schmidt operators from $U$ to $L^2_\sigma$ with the norm $\|\cdot\|_{L_2(U,L^2_\sigma)}$. Let $G: L^2_\sigma\rightarrow L_2(U,L^2_\sigma)$ be $\mathcal{B}(L^2_\sigma)/\mathcal{B}(L_2(U,L^2_\sigma))$ measurable. In the sequel, we assume the following.

(G1) There exists $C>0$ such that for all $x\in L^2_\sigma$
$$\|G(x)\|_{L_2(U,L^2_\sigma)}\leq C(1+\|x\|_{L^2}).$$

(G2) If $x_n\rightarrow x$ weakly in $L^{2}_{\sigma}$ then
$$\lim_{n\rightarrow \infty}\|G(x_n)-G(x)\|_{L_2(U,L^{2}_{\sigma})}=0.$$

\begin{remark}
We note that if the noise is additive, i.e., $G\in L_2(U,L^{2}_{\sigma})$ does not depend on the solution, then the conditions {(G1)} and {(G2)} hold. Moreover, for the multiplicative  case, we have for instance the following example: for $x\in L^2_\sigma$, $u\in U$ and $\{l_k\}_{k\in\N}$ being the orthonormal basis in $U$, let
$$G(x)u=\sum_{k=1}^{\infty}\langle u,l_k\rangle_U \Pi_k x f_k,$$
where $\Pi_ky=\sum_{j=1}^k\langle y,e_k\rangle e_k$ and $f_k\in L^\infty(\T)$ satisfying $\sum_{k=1}^{\infty}|k|^{s}\|f_k\|^2_{L^\infty}<\infty$ for some $s>0$. Indeed, it then holds
$$\|G(x_n)-G(x)\|_{L_2(U,L^{2}_{\sigma})}^2=\sum_{k=1}^{\infty}\| \Pi_k(x_n-x )f_k\|_{L^{2}_{\sigma}}^2\leq \|x_n-x\|_{H^{-s/2}}^2\sum_{k=1}^{\infty}|k|^s\|f_k\|^2_{L^\infty},$$
where  $\|x_n-x\|_{H^{-s/2}}\rightarrow0$ by compact embedding. One could also consider  nonlinear operators for instance by replacing $\Pi_{k}x$ by $g(\Pi_{\ell}x)$ for some fixed $\ell$ and  a Nemytskii operator  $g$  given by a Lipchitz function.
\end{remark}

\begin{remark}
At first sight, it may seem unsatisfactory that the simple linear multiplicative noise of the form
$G(u)\dif W=u\dif \beta$ with  a real-valued Brownian motion  $\beta$ is not covered by our theory. Indeed, it does not satisfy the condition (G2). However, let us point out that for this kind of noise one can perform a suitable transformation and a random rescaling of time in the spirit of \cite{CFF19} in order to rewrite the stochastic Euler system as a deterministic one. In other words, the results of the deterministic theory can be translated into this stochastic setting and actually much more can be proved in this case.
\end{remark}

For the weak--strong uniqueness principle we assume the following Lipschitz condition.

(Glip)  There exists a constant $L$ such that $$\|G(x)-G(y)\|^2_{L_2(U,L^2_\sigma)}\leq L\|x-y\|^2_{L^2}.$$

And additional assumptions are required for the convex integration method and the resulting non-uniqueness in law. Suppose there is another  Hilbert space $U_1$ such that the embedding  $U\subset  U_1$ is Hilbert--Schmidt. In particular, we  assume that the noise is additive and the following holds.

(G3) $G\in L_{2}(U,H^{(3+\sigma)/2})$ for some $\sigma>0$.

(G4) $G:U_1\rightarrow H^{5/2+\sigma}$ is a bounded operator for some $\sigma>0$.

\subsection{Path spaces}\label{s:path}

In Section~\ref{s:p1}, we introduce four notions of solutions: martingale and probabilistically weak solutions as well as simplified martingale and simplified probabilistically weak solutions. Thus, as a first step, we define four corresponding path spaces, which we denote by $\Omega_{M}$, $\Omega_{PW}$, $\Omega_{SM}$ and $\Omega_{SPW}$, respectively.

\subsubsection{Martingale solutions}

 Let $\alpha\in (2/3,1)$, $k\in(3/2,\infty)$, $q\in (1,\infty)$ be such that $\alpha q>2$ and $3\alpha/2-2/q>1$. Define
$$
\Omega_{SM}:=\left\{(x,y)\in C([0,\infty),H^{-3}\times (\mathcal{M}^+(\mathbb{T}^3,\R^{3\times 3}_{\rm{sym}}),w)); y\in W^{\alpha,q}_{\mathrm{loc}}([0,\infty);W^{-k,2}(\mathbb{T}^3,\R^{3\times 3}_{\rm{sym}}))\right\}.
$$
Here $\mathcal{M}^{+}(\mathbb{T}^3,\R^{3\times 3}_{\rm{sym}})$ denotes the space of (symmetric) positive semidefinite matrix valued finite measures $m$ so that for every $\xi\in \mathbb{R}^{3}$
$$
m:\xi\otimes\xi\ \mbox{is a  non-negative finite measure on}\  \mathbb{T}^{3},
$$
where $A:B=\sum_{i,j=1}^{3}a_{ij}b_{ij}$ denotes the matrix inner product of  symmetric  matrices $A=(a_{ij})_{i,j=1}^{3}$, $B=(b_{ij})_{i,j=1}^{3}$.
By $(\mathcal{M}^{+}(\mathbb{T}^3,\R^{3\times 3}_{\rm{sym}}),w)$ we denote this space equipped with the weak topology, which is completely separably metrizable by the Prokhorov metric, see \cite[Lemma 4.3, Lemma 4.5]{K17}. As a consequence, the path space $\Omega_{SM}$ is a Polish space. The parameters $\alpha,q$ are chosen in  a way to permit the definition of a certain Young integral, see the discussion after \eqref{eq:SSS} below for more details. The parameter $k$ is chosen so that the embedding $\mathcal{M}^{+}\subset W^{-k,2}$ holds in three space dimensions.

 Let $\mathscr{P}(\Omega_{SM})$ denote the set of all probability measures on $(\Omega_{SM},\mathcal{B}_{SM})$ with $\mathcal{B}_{SM}$ being the Borel $\sigma$-algebra coming from the topology of locally uniform convergence on $\Omega_{SM}$. Let  $(x,y):\Omega_{SM}\rightarrow H^{-3}\times \mathcal{M}^+(\mathbb{T}^3,\R^{3\times 3}_{\rm{sym}})$ denote the canonical process on $\Omega_{SM}$ given by
$$(x_t(\omega),y_t(\omega))=\omega(t).$$
Similarly, for $t\geq 0$ we define  Borel $\sigma$-algebra $\mathcal{B}_{SM}^{t}=\sigma\{ (x(s),y(s)),s\geq t\}$.
Finally,  we define the canonical filtration  $\mathcal{B}_{SM,t}^0:=\sigma\{ (x(s),y(s)),s\leq t\}$, $t\geq0$, as well as its right continuous version $\mathcal{B}_{SM,t}:=\cap_{s>t}\mathcal{B}^0_{{SM},s}$, $t\geq 0$.

Set $\mathbb{X}:=\{(x_0,y_0,z_0)\in L^2_\sigma\times \mathcal{M}^+(\mT^3,\mR^{3\times 3}_{\rm{sym}})\times \mR;\,\|x_0\|^2_{L^2}\leq z_0\}$.
\begin{align*}
\Omega_{M}:=\big\{(x,y,z)|&(x,y,z)\in C([0,\infty),H^{-3}\times (\mathcal{M}^+(\mathbb{T}^3,\R^{3\times 3}_{\rm{sym}}),w)\times \R);
 \\&y\in W^{\alpha,p}_{\mathrm{loc}}([0,\infty);W^{-k,2}(\mathbb{T}^3,\R^{3\times 3}_{\rm{sym}}))\big\},
 \end{align*}
and
\begin{align*}\Omega_{M}^t:=\big\{(x,y,z)|&(x,y,z)\in C([t,\infty),H^{-3}\times (\mathcal{M}^+(\mathbb{T}^3,\R^{3\times 3}_{\rm{sym}}),w)\times \R);
 \\&y\in W^{\alpha,p}_{\mathrm{loc}}([t,\infty);W^{-k,2}(\mathbb{T}^3,\R^{3\times 3}_{\rm{sym}}))\big\}.
 \end{align*}
  Let $\mathscr{P}({\Omega_M})$ denote the set of all probability measures on $({\Omega_M},{\mathcal{B}_M})$ with ${\mathcal{B}_M}$ being the Borel $\sigma$-algebra coming from the topology of locally uniform convergence on ${\Omega}_M$. Let  $(x,y,z):{\Omega}_M\rightarrow H^{-3}\times \mathcal{M}^+(\mathbb{T}^3,\R^{3\times 3}_{\rm{sym}})\times \R$ denote the canonical process on $\Omega_{M}$ given by
$$(x_t(\omega),y_t(\omega),z_t(\omega))=\omega(t).$$
Similarly, for $t\geq 0$ we define  Borel $\sigma$-algebra $\mathcal{B}_M^{t}=\sigma\{ (x(s),y(s),z(s)),s\geq t\}$.
Finally,  we define the canonical filtration  ${\mathcal{B}}_{M,t}^0:=\sigma\{ (x(s),y(s),z(s)),s\leq t\}$, $t\geq0$, and its right continuous version ${\mathcal{B}}_{M,t}:=\cap_{s>t}\mathcal{B}^0_{M,s}$, $t\geq 0$.
 For given probability measure $P$ we use $\E^P$ to denote the expectation under $P$.

\subsubsection{Probabilistically weak solutions}

For the same parameters $\alpha,q,k$ as before define
\begin{align}\label{eq:SSS}
\begin{aligned}
\Omega_{SPW}&:=\big\{(x,y,b)\in  C([0,\infty); H^{-3}\times (\mathcal{M}^+(\mathbb{T}^3,\R^{3\times 3}_{\rm{sym}}),w)\times U_1);
\\& \quad y\in W^{\alpha,q}_{\mathrm{loc}}([0,\infty);W^{-k,2}(\mathbb{T}^3,\R^{3\times 3}_{\rm{sym}})),b\in W^{\alpha/2,q}_{\mathrm{loc}}([0,\infty);U_1)\big\}.
\end{aligned}
\end{align}
Similarly to the above, the path space $\Omega_{SPW}$ is a Polish space. The parameters $\alpha,q$ are chosen in  a way to permit the definition of a certain Young integral, which will be needed below. Recall that for $g\in C^{\beta}$ and $f\in C^{\gamma}$ the Young integral $t\mapsto \int_{0}^{t}g_{r}\dif f_{r}$ is well-defined provided $\beta+\gamma>1$, cf. Lemma~\ref{lem:young}.  In view of the continuous embedding $W^{\alpha,q}\subset C^{\beta}$ valid in one dimension for $\beta\in (0,\alpha- 1/q)$,  the above Young integral is well-defined provided  $g\in W^{\alpha,q}$ and $f\in W^{\alpha/2,q}$ such that $3\alpha/2-2/q >1$.

Let $\mathscr{P}(\Omega_{SPW})$ denote the set of all probability measures on $(\Omega_{SPW},{\mathcal{B}}_{SPW})$ with ${\mathcal{B}_{SPW}}$   being the Borel $\sigma$-algebra coming from the topology of locally uniform convergence on $\Omega_{SPW}$. Let  $(x,y,b):\Omega_{SPW}\rightarrow H^{-3}\times \mathcal{M}^+(\mathbb{T}^3,\R^{3\times 3}_{\rm{sym}})\times U_{1}$ denote the canonical process on $\Omega_{SPW}$ given by
$$(x_t(\omega),y_t(\omega),b_t(\omega))=\omega(t).$$
For $t\geq 0$ we define   $\sigma$-algebra ${\mathcal{B}}_{SPW}^{t}=\sigma\{ \omega(s),s\geq t\}$.
Finally,  we define the canonical filtration  ${\mathcal{B}}_{SPW,t}^0:=\sigma\{ \omega(s),s\leq t\}$, $t\geq0$, and its right continuous version ${\mathcal{B}}_{SPW,t}:=\cap_{s>t}{\mathcal{B}}^0_{SPW,s}$, $t\geq 0$.

Similarly, we define
\begin{align*}
\Omega_{PW}&:=\big\{(x,y,z,b)\in  C([0,\infty); H^{-3}\times (\mathcal{M}^+(\mathbb{T}^3,\R^{3\times 3}_{\rm{sym}}),w)\times\R\times U_1);
\\&\quad y\in W^{\alpha,q}_{\mathrm{loc}}([0,\infty);W^{-k,2}(\mathbb{T}^3,\R^{3\times 3}_{\rm{sym}})),b\in W^{\alpha/2,q}_{\mathrm{loc}}([0,\infty);U_1)\big\},
\end{align*}
and the canonical process on $\Omega_{PW}$ and $(\mathcal{B}_{PW,t}^0)_{t\geq0},(\mathcal{B}_{PW,t})_{t\geq0}.$

\section{Class of dissipative solutions}

\subsection{Preliminary discussion}\label{s:p1}

It turns out that a good stable notion of solution to the Euler system \eqref{el} has to include more information than what is provided by the velocity field itself. Let us first discuss the main ideas on an informal level which also permits us to fix the notation.

In the definitions below we denote by $x$ the velocity field and rewrite the Euler system \eqref{el} as
\begin{equation}\label{eq:1a}
\dif x+\div\mathfrak{R}\,\dif t+\nabla p\,\dif t=G(x)\,\dif b,\qquad\div x =0,\qquad x(0)=x_{0},
\end{equation}
satisfied in a distributional sense. Here $b$ is a cylindrical Wiener process in $U$. This way we introduce a matrix-valued variable $\mathfrak{R}$, which is considered as  part of the solution. Furthermore, we require a compatibility condition, namely, that the so-called Reynolds stress  satisfies
\begin{equation}\label{eq:1b}
\mathfrak{N}:=\mathfrak{R}-x\otimes x\geq 0.
\end{equation}
Observe that if $\mathfrak{R}=x\otimes x$, the definition reduces to the usual notion of analytically weak solution. However, due to the lack of compactness for the Euler equations, weak solutions are not stable under approximations and this is the reason for weakening the notion of solution further by introducing $\mathfrak{R}$. Since $\mathfrak{R}$ is only $L^{\infty}$ with respect to the time variable, we  also work with its primitive function given by
\begin{equation}\label{eq:1c}
\partial_{t} y = \mathfrak{R},\quad y(0)=y_{0}.
\end{equation}

We aim for a class of solutions satisfying a weak--strong uniqueness principle, thus, we  include an energy inequality into our definition. In the deterministic setting, this corresponds to the notion of dissipative  measure--valued solution extensively studied in the  literature. However, we choose a different formulation which overcomes various difficulties present in the stochastic setting. In particular, we introduce a new variable $z$ satisfying
\begin{equation}\label{eq:1d}
\dif z=2\langle x,G(x)\dif b\rangle+ \|G(x)\|_{L_{2}(U,L^{2})}^{2}\dif t,\qquad z(0)=z_{0}.
\end{equation}
Note that if $x$ is an analytically weak solution possessing sufficient spatial regularity then the energy equality holds, that is, $z(t)=\|x(t)\|_{L^{2}}^{2}$ for all $t\geq 0$. This requirement  has to be relaxed and so  we  postulate instead the compatibility condition
\begin{equation}\label{eq:1e}
\int_{\T}\dif\tr\mathfrak{R}(t)=\|x(t)\|_{L^{2}}^{2}+\int_{\T}\dif\tr\mathfrak{N}(t)\leq z(t)\quad\mbox{for a.e.}\ t\geq 0,
\end{equation}
which only permits to show  that $z(t)\geq \|x(t)\|_{L^{2}}^{2}$ for all $t\geq 0$.

For the weak--strong uniqueness principle we additionally require $z_{0}=\|x_{0}\|_{L^{2}}^{2}$. However, a general initial value $z_{0}$ has to be allowed  for the Markov selection in order to obtain a notion of solution stable under shifts on trajectories.

Roughly speaking, a \emph{dissipative martingale solution} defined below is the probability law of $(x,y,z)$ satisfying \eqref{eq:1a}, \eqref{eq:1b}, \eqref{eq:1c}, \eqref{eq:1d}, \eqref{eq:1e} in a suitable sense. A \emph{dissipative probabilistically weak solution} is then the probability law of $(x,y,z,b)$ such that \eqref{eq:1a}, \eqref{eq:1b}, \eqref{eq:1c}, \eqref{eq:1d}, \eqref{eq:1e} hold.
In the case of $z_{0}=\|x_{0}\|_{L^{2}}^{2}$, it follows from \eqref{eq:1d} that $z$ is a function of the other variables. Therefore, we define a \emph{simplified dissipative martingale solution} as the law of $(x,y)$ under a dissipative martingale solution and a \emph{simplified dissipative probabilistically weak solution} as the law of  $(x,y,b)$ under a dissipative probabilistically weak solution.

\subsection{Dissipative martingale solutions}

Throughout the paper, the meaning of the variables $x, b, \mathfrak{R}, \mathfrak{N}, y, z$ remains the same as in Section~\ref{s:p1}.

\bd\label{martingale solution1}
Let  $(x_{0},y_0,z_0)\in \mX$. A probability measure $P\in \mathscr{P}(\Omega_{M})$ is  a  dissipative martingale solution to the Euler system \eqref{el}  with the initial value $(x_0,y_0,z_0) $ at time $s$ provided

\emph{(M1)} $P(x(t)=x_0, y(t)=y_0, z(t)=z_0,0\leq t\leq s)=1$, $P(\mathfrak{N}\in L^{\infty}_{\rm{loc}}([s,\infty);\mathcal{M}^{+}(\T;\mathbb{R}^{3\times 3}_{\rm{sym}})))=1$.

\emph{(M2)} For every $e_i\in C^\infty(\mathbb{T}^3)\cap L^2_\sigma$ and  $t\geq 0$ the process
$$M_{t,s}^{i}:=\langle x(t)-x(s),e_i\rangle-\int^t_s\int_{\mathbb{T}^3}\nabla e_i:\dif\mathfrak{R}(r) \dif r$$
is a continuous square integrable $({\mathcal{B}}_{M,t}^0)_{t\geq s}$-martingale under $P$ with the quadratic variation process
given by
$\int_s^t\|G(x(r))^*e_i\|_{U}^2\dif r$.

\emph{(M3)} P-a.s. for every $t\geq s$
\begin{align*}
z(t)=z(s)+2M^E_{t,s}+\int_s^t\|G(x(r))\|_{L_2(U,L^2)}^2\dif r,
\end{align*}
where $M^E_{t,s}=\sum_{i=1}^\infty\int_s^t\langle x(r), e_i\>\dif M_{r,s}^i$ is a continuous $({\mathcal{B}}_{M,t}^0)_{t\geq s}$-martingale and
\begin{align*}
P\left(\int_{\T}\dif\tr\mathfrak{R}(t)\leq z(t) \ \textrm{\emph{for a.e.}}\  t\geq s\right)=1.
\end{align*}
\ed

 The following result shows how to derive a priori estimates for dissipative solutions.

\begin{lemma}\label{lem:def} Suppose that \emph{(G1)} holds.
Let $P$ be a dissipative martingale solution to \eqref{el}. Then
P-a.s. for every $t\geq s$ and every $p\in[2,\infty)$
\begin{align}\label{eq:ito}
\begin{aligned}
z^{p}(t)&=z^{p}(s)+2p\int_s^tz^{p-1}(r)\dif M^{E}_{r}+p\int_s^tz^{p-1}(r)\|G(x(r))\|_{L_2(U,L^2)}^2\dif r\\
&\qquad +2p(p-1)\int_s^tz^{p-2}(r)\|G(x(r))^*x(r)\|_{U}^2\dif r,
\end{aligned}
\end{align}
and
\begin{equation}\label{eq:z1}
P(\|x(t)\|_{L^{2}}^{2}\leq z(t)\ \mbox{\emph{for all}}\  t\geq s)=1.
\end{equation}
Consequently, for every $N\in\N$ and $p\in[2,\infty)$ there exists a universal constant $C_{N,p}>0$ such that
$$\E^P\left[\sup_{t\in[0,N]}\|x(t)\|_{L^2}^{2p}\right]+\E^P\left[\esssup_{t\in[0,N]}\left(\int_{\mathbb{T}^3} \dif\mathrm{tr}\mathfrak{R}(t)\right)^p\right]\leq C_{N,p}(z^{p}(s)+1).$$
\end{lemma}

\begin{proof}
First, we observe that according to (M1), (M3) it holds $P$-a.s. for a.e. $t\geq s$ that $z(t)\geq 0$ and this can be extended to every $t\geq s$ by continuity of $t\mapsto z(t)$. Then \eqref{eq:ito} is an application of It\^o's formula.

Next, we realize that by (M1), (M3) it holds $P$-a.s. for a.e. $t\geq s$
$$
\int_{\T}\dif\tr\mathfrak{R}(t)\leq z(t),\qquad \|x(t)\|_{L^{2}}^{2}=\int_{\T}\dif\tr\mathfrak{R}(t)-\int_{\T}\dif\tr\mathfrak{N}(t)\leq  z(t).
$$
Since the function $t\mapsto \|x(t)\|_{L^{2}}^{2}$ is lower semicontinuous due to the continuity of $x$ in $H^{-3}$ and $t\mapsto z(t)$ is continuous, we deduce \eqref{eq:z1}.

To prove the last  claim, we first choose stopping time $\tau_R:=\inf\{t>0, z(t)\geq R\}$ and by (M3) we know
$\tau_R\to\infty$ as $R\to\infty$. In fact by (M3) and using the Burkholder--Davis--Gundy inequality, Young's inequality, the linear growth assumption (G1) on $G$ and Gronwall's lemma we obtain
$$
\E^{P}\left[\sup_{t\in[0,N]}z(t)\right]\lesssim z(s)+1.
$$
Then
we estimate the right hand side of \eqref{eq:ito} using the Burkholder--Davis--Gundy inequality, Young's inequality, the linear growth assumption (G1) on $G$ and Gronwall's lemma to deduce
$$
\E^{P}\left[\sup_{t\in[0,N\wedge \tau_R]}z^{p}(t)\right]\lesssim z^{p}(s)+1,
$$
where the implicit constant only depends on $N,p$ and the constant in (G1). Letting $R\to \infty$, the result follows.
\end{proof}

We note that a dissipative martingale solution in the sense of Definition~\ref{martingale solution1} may contain an initial jump of the energy in the sense that $z_{0}>\|x_{0}\|_{L^{2}}^{2}$. However, we are even able to construct solutions without the initial energy jump, and this will be seen in the construction by compactness in Section~\ref{s:ex} as well as in the construction by convex integration in Section~\ref{s:ci}. In addition, the weak--strong uniqueness principle requires the assumption $z_{0}=\|x_{0}\|_{L^{2}}^{2}$. The reason for relaxing this in our main definition of a solution is the Markov selection in Section~\ref{sec:mar}. More precisely, a notion of solution without an initial energy jump is not stable under shifts on trajectories, which is one of the main ingredients required by  the Markov selection.

We observe that the notion of dissipative solution simplifies in case of no initial energy jump.

\bd\label{smartingale solution1}
Let  $(x_{0},y_0)\in L^2_\sigma\times \mathcal{M}^{+}(\T;\mathbb{R}^{3\times 3}_{\rm{sym}})$. A probability measure $P\in \mathscr{P}(\Omega_{SM})$ is  a simplified dissipative martingale solution to the Euler system \eqref{el}  with the initial value $(x_0,y_0) $ at time $s$ provided

\emph{(M1)} $P(x(t)=x_0, y(t)=y_0, 0\leq t\leq s)=1$, $P(\mathfrak{N}\in L^{\infty}_{\rm{loc}}([s,\infty);\mathcal{M}^{+}(\T;\mathbb{R}^{3\times 3}_{\rm{sym}})))=1$.

\emph{(M2)} For every $e_i\in C^\infty(\mathbb{T}^3)\cap L^2_\sigma$ and  $t\geq 0$ the process
$$M_{t,s}^{i}:=\langle x(t)-x(s),e_i\rangle-\int^t_s\int_{\mathbb{T}^3}\nabla e_i:\dif\mathfrak{R}(r) \dif r$$
is a continuous square integrable $(\mathcal{B}_{SM,t}^0)_{t\geq s}$-martingale under $P$ with the quadratic variation process
given by
$\int_s^t\|G(x(r))^*e_i\|_{U}^2\dif r$.

\emph{(M3)} P-a.s. define for every $t\geq s$
\begin{align*}
z(t):=\|x_0\|_{L^2}^2+2M^E_{t,s}+\int_s^t\|G(x(r))\|_{L_2(U,L^2)}^2\dif r,
\end{align*}
where $M^E_{t,s}=\sum_{i=1}^\infty\int_s^t\langle x(r), e_i\>\dif M_{r,s}^i$ is a continuous $(\mathcal{B}_{SM,t}^0)_{t\geq s}$-martingale and
\begin{align*}
P\left(\int_{\T}\dif\tr\mathfrak{R}(t)\leq z(t) \ \textrm{\emph{for a.e.}}\  t\geq s\right)=1.
\end{align*}
\ed

For this definition we do not need $z$ in the path space and we can prove that if $z_{0}=\|x_{0}\|_{L^{2}}^{2}$ then these two definitions are equivalent.
\begin{corollary}\label{cor:1}
Let $P$ be a dissipative martingale solution with the initial value $(x_{0},y_{0},z_{0})$ at time $s$ such that $z_{0}=\|x_{0}\|_{L^{2}}^{2}$. Then the canonical process $z$ under $P$ is a function of $x,y$. In other words, $P$ is fully determined by the joint probability law of $x,y$ and can be identified with a probability measure on the reduced path space $\Omega_{SM}$. Hence $P$ is a simplified dissipative martingale solution with the initial value $(x_{0},y_{0})$ at time $s$.

Conversely, let $(x_{0},y_{0})\in L^2_\sigma\times \mathcal{M}^{+}(\T;\mathbb{R}^{3\times 3}_{\rm{sym}})$ be given and let $P\in \mathscr{P}(\Omega_{SM})$ be a simplified dissipative martingale solution and define $P$-a.s. for $t\geq s$
$$
z(t):=\|x_0\|_{L^2}^2+2M^E_{t,s}+\int_s^t\|G(x(r))\|_{L_2(U,L^2)}^2\dif r.
$$
Let $Q$ be the law of $(x,y,z)$ under $P$. Then $Q\in \mathscr{P}({\Omega}_M)$
gives raise to a dissipative martingale solution starting from the initial value $(x_{0},y_{0},\|x_0\|_{L^2}^2)$ at the time $s$.
\end{corollary}

\begin{proof}
It follows from (M2) that $M$ is a function of $x,y$ and consequently from (M3) we deduce that $z$ is determined by $x,y$ and the initial value $z_{0}$. This gives the first claim whereas the second claim is immediate.
\end{proof}

We also observe that if $P$ is a dissipative martingale solution to \eqref{el} with initial value $(x_{0},y_{0},z_{0})$ then the law of the process $(x,y+c_{1},z+c_{2})$ under $P$ is again a dissipative martingale solution to \eqref{el} with initial value $(x_{0},y_{0}+c_{1},z_{0}+c_{2})$ for every $c_{1}\in\R$, $c_{2}\geq 0$. Indeed, the initial value $y_{0}$ does not have any influence on the actual dynamics: it was introduced artificially by including $y$ into the path space rather than $\mathfrak{R}=\partial_{t}y$, which is not continuous in time. The role of the initial value $z_{0}$ is more delicate. It is related to the so-called energy sinks discussed more in detail in Section~\ref{sec:mar}.

In order to further verify that our definition of dissipative solution is reasonable, we first prove that a sufficiently regular dissipative solution is a solution in the classical sense.

\begin{proposition}\label{classical}  Suppose that \emph{(G1)} holds. Let $\mathfrak{t}$ be a $({\mathcal{B}}_{SM,t}^0)_{t\geq s}$-stopping time.
  If $P$ is a simplified dissipative martingale solution to \eqref{el} with the initial value $(x_{0},y_{0})$ at time $s$ such that
  $$
  x(\cdot\wedge\mathfrak{t}) \in C([s,\infty);C^{1}(\T))\hspace{1em} P\mbox{-a.s.},
  $$
then
  $$P(\mathfrak{R}=x\otimes x \ \mbox{\emph{for a.e.}}\  \mathfrak{t} \geq t\geq s)=1.$$
In other words, under $P$ the canonical process $x$ satisfies \eqref{el} before $\mathfrak{t}$ in the analytically strong sense.
\end{proposition}

\begin{proof}
Due to the sufficient spatial regularity  the canonical process  $x$ under $P$, we may apply It\^o's formula to obtain for an arbitrary  $(\mathcal{B}^{0}_{SM,t})_{t\geq s}$-stopping time $\tau\leq \mathfrak{t}$
\begin{align}
\| x (t\wedge \tau) \|_{L^2}^2&= \| x (s) \|_{L^2}^2 + 2  \int_s^{t\wedge \tau} \int_{\T}
  \nabla x : \dif \mathfrak{N} (r) \dif r  \no\\
  &\quad+2\sum_{i=1}^{\infty}\int_{s}^{t\wedge\tau}\langle x(r),e_{i}\rangle \dif M^{i}_{r,s}+  \int_s^{t\wedge\tau} \| G(x(r)) \|^2_{L_2(U,L^2_\sigma)} \dif r
. \label{eq:2}
\end{align}

Now, we subtract {\eqref{eq:2}} from $z(t)$ and use the positive semidefinitness of $\mathfrak{N}$ as well as  (M3) to deduce
\begin{equation}  \label{eq:31}
\begin{aligned}
\E^P \left[(z-\|x\|_{L^2}^2) (t\wedge\tau) \right] &=
    -2\E^P \left[ \int_s^{t\wedge\tau} \int_{\T} \nabla x : \dif \mathfrak{N} (r) \dif r \right] \\
    & \lesssim  \E^P \left[\int_s^{t\wedge\tau} \|\nabla x\|_{L^{\infty}}\int_{\T}  \dif\tr \mathfrak{N} (r) \dif r \right]\\
    & \leq \E^P \left[\int_s^{t\wedge\tau} \|\nabla x\|_{L^{\infty}}(z-\|x\|_{L^2}^2) (r) \dif r \right]\\
    & \leq \mathbf{E}^{P} \left[ \int_s^{ t} \|\nabla x\|_{L^{\infty}} (z-\|x\|_{L^2}^2) (r\wedge\tau) \dif r \right] .
   \end{aligned}
 \end{equation}
 Here in the last step we used \eqref{eq:z1}.
Let us now define the stopping times
\begin{align*}
 \tau_R = \inf \{ t \geqslant s ; \|\nabla x(t)\|_{L^{\infty}} \geq R \} \wedge \mathfrak{t}. 
\end{align*}
Then $\tau_R \rightarrow \mathfrak{t}$ $P$-a.s. and it follows from \eqref{eq:31}
that
\begin{align*}
\E^P \left[ (z-\|x\|_{L^2}^2) (t\wedge \tau_R) \right] &
    \leq R\int_s^{t} \E^P \left[ (z-\|x\|_{L^2}^2) (r\wedge\tau_{R}) \right] \dif r
    .
    \end{align*}
By Gronwall's inequality and sending $R\to\infty$ we obtain for every $\mathfrak{t}\geq t\geq s$
$$
P(z(t\wedge\mathfrak{t})=\|x(t\wedge\mathfrak{t})\|_{L^{2}}^{2})=1,
$$
hence the claim follows by the continuity of $t\mapsto z(t)$ as well as $t\mapsto\|x(t)\|_{L^{2}}^{2}$ under $P$ and (M3).
\end{proof}

\subsection{Dissipative probabilistically weak solutions}

We conclude this section with the definition  of dissipative probabilistically weak solution.

\bd\label{probabilistically weak solution}
Let  $(x_{0},y_0,z_0)\in \mX$. A probability measure $P\in \mathscr{P}(\Omega_{PW})$ is  a  dissipative probabilistically weak solution to the Euler system \eqref{el}  with the initial value $(x_0,y_0,z_0,b_{0}) $ at time $s$ provided

\emph{(M1)} $P(x(t)=x_0, y(t)=y_0, z(t)=z_0, b(t)=b_{0},0\leq t\leq s)=1$,\\
\phantom{ ~} \hspace{.9cm} $P(\mathfrak{N}\in L^{\infty}_{\rm{loc}}([s,\infty);\mathcal{M}^{+}(\T;\mathbb{R}^{3\times 3}_{\rm{sym}})))=1$.

\emph{(M2)} Under $P$, $b$ is a cylindrical $({\mathcal{B}}_{PW,t}^0)_{t\geq s}$-Wiener process in $U$ starting from $b_0$ at time $s$ and for every $e_i\in C^\infty(\mathbb{T}^3)\cap L^2_\sigma$ and  $t\geq s$
$$\langle x(t)-x(s),e_i\rangle-\int^t_s\int_{\mathbb{T}^3}\nabla e_i:\dif\mathfrak{R}(r) \dif r=\int_s^t\langle e_i,G(x(r))\dif b(r)\rangle.$$

\emph{(M3)} P-a.s. for every $t\geq s$
\begin{align*}
z(t)=z(s)+2\int_s^t\langle x(r), G(x(r))\dif b(r)\rangle+\int_s^t\|G(x(r))\|_{L_2(U,L^2)}^2\dif r
\end{align*}
and
\begin{align*}
P\left(\int_{\T}\dif\tr\mathfrak{R}(t)\leq z(t) \ \textrm{\emph{for a.e.}}\  t\geq s\right)=1.
\end{align*}
\ed

Similar as above we can also introduce the following simplified dissipative probabilistically weak solution if $z_0=\|x_0\|_{L^2}^2$.
\bd\label{sprobabilistically weak solution}
Let  $(x_{0},y_0)\in L^2_\sigma\times \mathcal{M}^{+}(\T;\mathbb{R}^{3\times 3}_{\rm{sym}})$. A probability measure $P\in \mathscr{P}(\Omega_{SPW})$ is  a simplified dissipative probabilistically weak solution to the Euler system \eqref{el}  with the initial value $(x_0,y_0,b_{0}) $ at time $s$ provided

\emph{(M1)} $P(x(t)=x_0, y(t)=y_0, b(t)=b_{0},0\leq t\leq s)=1$, $P(\mathfrak{N}\in L^{\infty}_{\rm{loc}}([s,\infty);\mathcal{M}^{+}(\T;\mathbb{R}^{3\times 3}_{\rm{sym}})))=1$.

\emph{(M2)} Under $P$, $b$ is a cylindrical $({\mathcal{B}}_{SPW,t}^0)_{t\geq s}$-Wiener process in $U$ starting from $b_0$ at time $s$ and for every $e_i\in C^\infty(\mathbb{T}^3)\cap L^2_\sigma$ and  $t\geq s$

$$\langle x(t)-x(s),e_i\rangle-\int^t_s\int_{\mathbb{T}^3}\nabla e_i:\dif\mathfrak{R}(r) \dif r=\int_s^t\langle e_i,G(x(r))\dif b(r)\rangle.$$

\emph{(M3)} P-a.s. define for every $t\geq s$
\begin{align*}
z(t)=\|x_0\|_{L^2}^2+2\int_s^t\langle x(r), G(x(r))\dif b(r)\rangle+\int_s^t\|G(x(r))\|_{L_2(U,L^2)}^2\dif r.
\end{align*}

Then

\begin{align*}
P\left(\int_{\T}\dif\tr\mathfrak{R}(t)\leq z(t) \ \textrm{\emph{for a.e.}}\  t\geq s\right)=1.
\end{align*}

\ed
\br
(i) From this definition it is easy to see  that the law of $(x,y,z)$ under a dissipative probabilistically weak solution $P$ gives a dissipative martingale solution in the sense of Definition~\ref{martingale solution1}.

(ii) Similarly to Corollary \ref{cor:1} these two definitions of dissipative probabilistically weak solutions are equivalent under the condition that $z_0=\|x_0\|_{L^2}^2$.
\er

\section{Weak--strong uniqueness, stability and existence}

\subsection{Weak--strong uniqueness}\label{s:un}

As the next step, we show that dissipative solutions satisfy a weak--strong uniqueness principle.

\bt\label{weak-strong uniqueness}  Suppose that \emph{(G1)} and \emph{(Glip)} hold.
  Let $P$ be a simplified dissipative probabilistically weak solution to \eqref{el} starting from the initial value $(x_{0},y_{0},b_{0})$ at time $s\geq0$. Assume that  on the  stochastic basis $(\Omega_{SPW},{\mathcal{B}_{SPW}},({\mathcal{B}}_{SPW,t}^0)_{t\geq 0},P)$ together with the $({\mathcal{B}}_{SPW,t}^0)_{t\geq s}$-Wiener process $b$, there exists an $({\mathcal{B}}_{SPW,t}^0)_{t\geq s}$-adapted process  $u$ which is an analytically strong solution  to \eqref{el} up to a $({\mathcal{B}}_{SPW,t}^0)_{t\geq s}$-stopping time $\mathfrak{t}$ such that $u(\cdot\wedge\mathfrak{t})\in C([s,\infty);C^1(\T))$ $P$-a.s. and
  $
  P(u(s)=x(s))=1.
  $
  Then
  $$P\big(x(t\wedge\mathfrak{t})=u(t\wedge\mathfrak{t})\  \mbox{\emph{for all}}\   t\geq s\big)=1.$$
\et

\begin{proof}
First of all, we note that since $u$ is regular enough, the usual a priori estimate for the Euler equations holds true. In particular, we may apply It\^o's formula to the function $u\mapsto \|u\|_{L^{2}}^{2}$ and estimate using  Burkholder--Davis--Gundy's  inequality, the linear growth assumption (G1) on $G$ and Gronwall's lemma to obtain for any $N\in\N$
\begin{align}\label{boundu}
  \E^P\left[\sup_{t\in [s,N]}\|u(t\wedge\mathfrak{t})\|_{L^2}^{2}\right]<\infty.
\end{align}

In order to establish the weak--strong uniqueness principle, we introduce a modification of the so-called relative energy between the two solutions $x$ and $u$, which is adapted to our definition of dissipative solution. Namely, for $t\in[ s,\mathfrak{t}]$ we let
\begin{align*}
E_{rel}(t)&:=\frac{1}{2}z(t)-\<x(t),u(t)\>+\frac{1}{2}\|u(t)\|_{L^2}^2
= \frac{1}{2}\|x(t)-u(t)\|^2_{L^2}+\frac{1}{2}\left(z(t)-\|x(t)\|_{L^2}^2\right).
\end{align*}
As a consequence of \eqref{eq:z1}, we obtain
\begin{align}\label{eq:E}
P(E_{rel}(t\wedge\mathfrak{t})\geq0 \textrm{ for all }t\geq s)=1.
\end{align}

Let $\tau\leq \mathfrak{t}$ be a $({\mathcal{B}}^{0}_{SPW,t})_{t\geq s}$-stopping time.
Using the regularity of $u$, we may apply It\^{o}'s formula to obtain
\begin{align*}
\langle x(t\wedge\tau),u(t\wedge\tau)\rangle
&=\langle x(s),u(s)\rangle+\sum_{i=1}^{\infty}\int_s^{t\wedge\tau}  \langle u(r),e_i\rangle \<e_i,G(x(r))\dif b_r\>
\\&\quad\quad+\int_s^{t\wedge\tau} \int_{\mathbb{T}^3}\nabla u(r):\dif\mathfrak{R}(r) \dif r-\int_s^{t\wedge\tau} \int_{\mathbb{T}^3}x(r)\cdot\div(u(r)\otimes u(r))\dif\xi \dif r
\\&\quad\quad+\sum_{i=1}^{\infty}\int_s^{t\wedge\tau}  \langle x(r),e_i\rangle \<e_i,G(u(r))\dif b_r\>+\int_s^{t\wedge\tau}  \langle G(x(r)),G(u(r))\rangle_{L_2(U,L^{2})}\dif r
\\&=\langle x(s),u(s)\rangle+\sum_{i=1}^{\infty}\int_s^{t\wedge\tau}   \langle u(r),e_i\rangle \<e_i,G(x(r))\dif b_r\>
\\&\quad\quad+\int_s^{t\wedge\tau} \int_{\mathbb{T}^3}\nabla u(r):\dif\mathfrak{N}(r) \dif r+\int_s^{t\wedge\tau} \int_{\mathbb{T}^3}\nabla u(r):(x\otimes x) \dif\xi \dif r\\
&\quad\quad-\int_s^{t\wedge\tau} \int_{\mathbb{T}^3}x(r)\cdot\div(u(r)\otimes u(r))\dif \xi \dif r
\\&\quad\quad+\sum_{i=1}^{\infty}\int_s^{t\wedge\tau} \langle x(r),e_i\rangle \<e_i,G(u(r))\dif b_r\>+\int_s^{t\wedge\tau} \langle G(x(r)),G(u(r))\rangle_{L_2(U,L^{2})}\dif r.
\end{align*}
Furthermore, we have
\begin{align*}
&\int_s^{t\wedge\tau}\int_{\mathbb{T}^3}\nabla u(r):\dif\mathfrak{N}(r) \dif r+\int_s^{t\wedge\tau}\int_{\mathbb{T}^3}\nabla u(r):(x\otimes x) \dif\xi dr-\int_s^{t\wedge\tau}\int_{\mathbb{T}^3}x(r)\cdot\div(u(r)\otimes u(r))\dif\xi \dif r
\\&\quad\quad\quad=\int_s^{t\wedge\tau}\int_{\mathbb{T}^3}D u(r):\dif\mathfrak{N}(r) \dif r+\int_s^{t\wedge\tau}\int_{\mathbb{T}^3}(x-u)\cdot Du(r)(x-u) \dif\xi \dif r.
\end{align*}
Here $Du=\frac{1}{2}(\nabla u+\nabla u^t)$. Taking expectation and using Lemma \ref{lem:def} and \eqref{boundu} we obtain
\begin{align*}
&\E^P\left[E_{rel}(t\wedge \tau)\right]= \frac{1}{2}\E^P\left[\|u(t\wedge \tau)\|^2_{L^2}\right]+\frac{1}{2}\E^P\left[z(t\wedge \tau)\right]
\\&\quad-\E^P\left[\langle x(s),u(s)\rangle\right]-\E^P\left[\int_s^{t\wedge\tau}\langle G(x(r)),G(u(r))\rangle_{L_2(U,L^{2})}\dif r\right]
\\&\quad-\E^P\left[\int_s^{t\wedge\tau}\int_{\mathbb{T}^3}D u(r):\dif\mathfrak{N}(r) \dif r\right]-\E^P\left[\int_s^{t\wedge\tau}\int_{\mathbb{T}^3}(x-u)\cdot Du(r)(x-u) \dif\xi \dif r\right].
\end{align*}
Hence  combining this with (M3) and the energy equality for $u$, namely,\begin{align*}
&\E^P[\|u(t\wedge\tau)\|^2_{L^2}]= \E^P[\|u(s)\|^2_{L^2}]
+ \E^P\left[\int_s^{t\wedge\tau}\|G(u(r))\|_{L_2(U,L^2)}^2\dif r\right],
\end{align*}
 implies that
\begin{align*}
&\E^P\left[E_{rel}(t\wedge\tau)\right]\leq \frac{1}{2}\E^P[\|x(s)-u(s)\|_{L^2}^2]+\E^P\left[\int_s^{t\wedge\tau}\| Du\|_{L^{\infty}}(z-\|x\|_{L^2}^2)(r)\dif r\right]
\\
&\quad+ \frac{1}{2}\E^P\left[\int_s^{t\wedge\tau}\| G(x(r))-G(u(r))\|^2_{L_2(U,L^{2})}\dif r\right]+\E^P\left[\int_s^{t\wedge\tau}\| Du\|_{L^{\infty}}\|x(r)-u(r)\|^2_{L^2}\dif r\right]
.
\end{align*}

Choose $\tau$ as the stopping time
$$\tau_R = \inf \{ t \geqslant s ; \|D u(t)\|_{L^{\infty}} \geq R \}\wedge \mathfrak{t} $$  and using similar argument as in the proof of Theorem~\ref{classical} together with \eqref{eq:E},
  we obtain
$$P\big(x(t\wedge \mathfrak{t})=u(t\wedge \mathfrak{t})\big)=1\quad  \mbox{for all}\   t\geq s,$$ which implies the result by the time continuity of $x$ and $u$.
\end{proof}

\subsection{Stability}\label{s:stab}

The following result provides a stability of the set of all probabilistically weak solutions with respect to the initial time and the initial condition. We denote by $\mathscr{C}_{PW}(s,x_0,y_0,z_0,b_{0})$ the set of all dissipative martingale solutions with the initial condition $(x_0,y_0,z_0,b_{0})$ and the initial time $s$.

\bt\label{convergence}  Suppose that \emph{(G1)}, \emph{(G2)} hold.
  Let $(x_n, y_n, z_n)\in\mathbb{X}$, $s_{n}\in [0,\infty)$, $n\in\N$, and assume that
  $$
  (s_n, x_n, y_n, z_n, b_n) \rightarrow (s_0, x_0, y_0, z_0, b_0) \ \mbox{in}\
  [0, \infty) \times L^2_{\sigma} \times \mathcal{M}^+ (\mathbb{T}^3,
  \mathbb{R}^{3 \times 3}_{\tmop{sym}}) \times [0, \infty) \times U_1
  $$ as $n
  \rightarrow \infty$ and let $P_n \in \mathscr C_{PW} (s_n, x_n, y_n, z_n, b_n)$. Then
  there exists a subsequence $n_k$ such that the sequence $(P_{n_k})_{k \in
  \mathbb{N}}$ converges weakly to some $P \in \mathscr C_{PW} (s_0, x_0, y_0, z_0, b_0)$.
\et

\begin{proof}
  \tmtextit{Step 1: Tightness.} In the first step, we show that $(P_n)_{n \in
  \mathbb{N}}$ is tight in $\Omega_{PW}$. Since for every $n \in
  \mathbb{N}$ the measure $P_n$ is a dissipative probabilistically weak
  solution to {\eqref{el}} starting from the initial condition $(x_n, y_n,
  z_n, b_n)$ at time $s_n$ in the sense of Definition \ref{probabilistically
  weak solution}, the process $b (\cdummy + s_n) - b_n$ is under $P_n$ a
  cylindrical Wiener process on $U$ starting at time $0$ from the initial
  value $0$. Using the fact that the law of a cylidrical Wiener process is
  unique and tight on $C^{\alpha/2+\varepsilon} _{\tmop{loc}}([0, \infty) ; U_1)$ for $\alpha/2+\varepsilon<1/2$, the same argument
  as in the proof of Theorem 5.1 in {\cite{HZZ19}} implies that the law of $b$
  under the family of measures $(P_n)_{n \in \mathbb{N}}$ is tight on
  $C([0, \infty) ; U_1)\cap W_{\tmop{loc}}^{\alpha/2,q}([0,\infty);U_1)$ for $\alpha,q$ as in the Section~\ref{s:path}.

  Next, Lemma~\ref{lem:def} yields all the necessary uniform estimates for the
  remaining variables. More precisely, in view of Lemma~\ref{lem:def} and
  (M2), (M3) we deduce for all $N \in \mathbb{N}$ and $\kappa \in (0, 1 / 2)$
  $$
   \sup_{n \in \mathbb{N}} \E^{P_n} \left[ \sup_{t \in [0, N]} \| x (t)
     \|_{L^2}^2 + \sup_{r \neq t \in [0, N]} \frac{\| x (t) - x (r) \|_{H^{-
     3}}}{| t - r |^{\kappa}} \right] < \infty, $$
$$ \sup_{n \in \mathbb{N}} \E^{P_n} \left[ \sup_{t \in [0, N]} | z (t) | +
     \sup_{r \neq t \in [0, N]} \frac{| z (t) - z (r) |}{| t - r |^{\kappa}}
     \right] < \infty . $$
  By the compact embedding (see Theorem~1.8.5 in {\cite{BFH18}})
  $$ L^{\infty} (0, N ; L^2 (\mathbb{T}^3)) \cap C^{\kappa} ([0, N] ; H^{- 3})
     \subset C([0, N] ; L_w^2 (\mathbb{T}^3)), $$
  this implies tightness of the law of $x$ and $z$, respectively, under the
  family of measures $(P_n)_{n \in \mathbb{N}}$ on $C ([0, \infty) ; L^2_w)$ and $C ([0,
  \infty))$, respectively.

  In order to prove tightness of $y$ under $(P_n)_{n \in \mathbb{N}}$ we
  recall that
  $$ y (t) = y_n + \int_{s_n}^t \mathfrak{R} (r) \mathd r \quad P_n \mbox{-a.s.} $$
  and that $\mathfrak R (r)$ is a positive semidefinite matrix-valued measure by (M1).
  Accordingly, it follows from Lemma~\ref{lem:def}
  $$\sup_{n \in \mathbb{N}} \E^{P_n} \left[ \sup_{t \in [0, N]} \| y (t) \|_{\mathcal{M}
     (\mathbb{T}^3, \mathbb{R}^{3 \times 3}_{\tmop{sym}})} + \sup_{r \neq t
     \in [0, N]} \frac{\| y (t) - y (r) \|_{\mathcal M (\mathbb{T}^3, \mathbb{R}^{3
     \times 3}_{\tmop{sym}})}}{| t - r |} \right] $$
$$ \lesssim \sup_{n \in \mathbb{N}} \E^{P_n} \left[ \sup_{t \in [0, N]}
     \int_{\mathbb{T}^3} \dif \tr \mathfrak{R} (t) \right] + \sup_{n \in
     \mathbb{N}} \| y_n \|_{\mathcal M (\mathbb{T}^3, \mathbb{R}^{3 \times
     3}_{\tmop{sym}})} < \infty . $$
  Therefore, we deduce the tightness of $y$ under $(P_n)_{n \in \mathbb{N}}$
  on $$C([0, \infty) ; (\mathcal M (\mathbb{T}^3, \mathbb{R}^{3 \times
  3}_{\tmop{sym}}), w))\cap W^{\alpha,q}_{\mathrm{loc}}([0,\infty);W^{-k,2}(\mathbb{T}^3,\R^{3\times 3}_{\rm{sym}})).$$
Without loss of generality, we may assume that $P_n$ converges weakly to some probability measure $P$ on $\Omega_{PW}$.

  As the next step, we apply Jakubowski--Skorokhod's representation theorem
  (cf. Theorem 2.7.1 in {\cite{BFH18}}). After passing to a subsequence, we deduce that on some probability space
  $(\Omega,\mathcal F, \tmmathbf{P})$ there are random variables $(\tilde{x}_n,
  \tilde{y}_n, \tilde{z}_n, \tilde{b}_n)$ as well as $(\tilde{x}, \tilde{y},
  \tilde{z}, \tilde{b})$ such that
  \begin{enumerate}
    \item[(i)] the law of $(\tilde{x}_n, \tilde{y}_n, \tilde{z}_n, \tilde{b}_n)$
    under $\tmmathbf{P}$ is given by $P_n$ for each $n \in \mathbb{N}$ and the law of $(\tilde{x}, \tilde{y}, \tilde{z}, \tilde{b})$
    under $\tmmathbf{P}$ is given by $P$.

    \item[(ii)] $(\tilde{x}_n, \tilde{y}_n, \tilde{z}_n, \tilde{b}_n) \rightarrow
    (\tilde{x}, \tilde{y}, \tilde{z}, \tilde{b})$ in $\Omega_{PW},$ and  $\tilde{x}_n \rightarrow
    \tilde{x}$ in
   $ C([0, \infty) ; L^2_{w}). $
  \end{enumerate}
  \tmtextit{Step 2: Identification of the limit.} Our goal is to show that $P
  \assign \tmop{Law} (\tilde{x}, \tilde{y}, \tilde{z}, \tilde{b})$ is a
  dissipative probabilistically weak solution starting from the initial
  condition $(x_0, y_0, z_0, b_0)$ at time $s$. First, we observe that as the
  initial conditions are deterministic, it follows immediately that
  $$ \tmmathbf{P} \left( \tilde{x} (t) = x_0, \tilde{y} (t) = y_0, \tilde{z}
     (t) = z_0, \tilde{b} (t) = b_0, \hspace{1em} t \in [0, s] \right) = 1. $$
  Since the weak formulation of {\eqref{el}} in (M2) as well as the energy
  equality in (M3) is satisfied by $(x, y, z, b)$ under each measure $P_n$, and
  $(\tilde{x}_n, \tilde{y}_n, \tilde{z}_n, \tilde{b}_n)$ has the same law under $\bf{P}$, it
  follows from Theorem~2.9.1 in {\cite{BFH18}} that (M2), (M3) are also
  satisfied by $(\tilde{x}_n, \tilde{y}_n, \tilde{z}_n, \tilde{b}_n)$ under
  $\tmmathbf{P}$. More precisely, $\tilde{b}_n$ is a cylindrical Wiener process on $U$ starting from $b_n$ at time $s_n$ with respect to 
  $\sigma((\tilde{x}_n, \tilde{y}_n, \tilde{z}_n, \tilde{b}_n)(s),s\leq t)$. Taking the limit  it is easy to see that $\tilde{b}$ is a cylindrical
  Wiener process on $U$ starting from $b_0$ at time $s$ with respect to 
$\sigma((\tilde{x}, \tilde{y}, \tilde{z}, \tilde{b})(s),s\leq t)$. Since $P$ is supported on $\Omega_{PW}$, $b$ is a cylindrical Wiener process on $U$ under $P$ with respect to $(\mathcal{B}^0_{PW,t})_{t\geq0}$.  By Lemma 2.6.6 in
  {\cite{BFH18}} and (G2), we are
  able to pass to the limit in the stochastic integrals in (M2) and (M3) in $L^2(0,T)$ in probability. Indeed, since $\tilde{x}_n \rightarrow
    \tilde{x}$ in
   $ C([0, \infty) ; L^2_{w}),$ we obtain $\sup_n\sup_{t\in[0,T]}\|\tilde{x}_n(t) \|_{L^2}<\infty$ for any $T>0$ $\mathbf{P}$-a.s.
The  convergence of the quadratic variation term
  $$ \int_{s_n}^t \| G (\tilde{x}_n (r)) \|_{L_2 (U, L^2)}^2 \mathd r
     \rightarrow \int_s^t \| G (\tilde{x} (r)) \|_{L_2 (U, L^2)}^2 \mathd r
     \quad \tmmathbf{P}\mbox{-a.s.} $$
  also follows from (G2) and (G1). For the stochastic integral  in (M3),  we have
\begin{align*}
&\|1_{[s_n,T]}G(\tilde{x}_n(r))^*\tilde{x}_n-1_{[s,T]}G(\tilde{x}(r))^*\tilde{x}\|_U
\\
&\qquad\qquad\leq 1_{[s_n,s)}\|G(\tilde{x}_n(r))^*\tilde{x}_n ||_U +1_{[s,T]}\|(G(\tilde{x}_n(r))-G(\tilde{x}(r)))\|_{L_2(U,H)}\|\tilde{x}_n\|_H\\
&\qquad\qquad\qquad+\|1_{[s,T]}(G(\tilde{x}(r))^*(\tilde{x}_n-\tilde{x}))\|_U,\end{align*}
and the last term goes to zero by the compactness of $G^*$ and weak convergence of $\tilde{x}_n$ to $\tilde{x}$.
Then by (G1) and (G2) and dominated convergence theorem we deduce $\mathbf{P}$-a.s.
$$\int_0^T\|1_{[s_n,T]}G(\tilde{x}_n(r))^*\tilde{x}_n-1_{[s,T]}G(\tilde{x}(r))^*\tilde{x}\|_U^2\dif r\rightarrow0.$$
Thus by Lemma 2.6.6 in {\cite{BFH18}} we can pass the limit of the stochastic integral  in (M3). For the stochastic term in (M2) the argument is similar and actually easier.

  Regarding the compatibility conditions on the stresses $\mathfrak{N}$ and $\mathfrak{R}$ in (M1)
  and (M3) and the convergence of the stress term in (M2), let us denote consistently
\begin{equation}\label{eq:stress1}
 \tilde{\mathfrak R}_n \assign \partial_t \tilde{y}_n, \quad \tilde{\mathfrak{N}}_n
     \assign \tilde{\mathfrak{R}}_n - \tilde{x}_n \otimes \tilde{x}_n, \quad
     \tilde{\mathfrak{R}} \assign \partial_t \tilde{y}, \quad \tilde{\mathfrak{N}} \assign
     \tilde{\mathfrak{R}} - \tilde{x} \otimes \tilde{x} .
     \end{equation}

In order to pass to the limit in the stress term in (M2), we note that by \eqref{eq:stress1} and using the convergence of $\tilde y_{n}\to\tilde y$ in $C([0,\infty);(\mathcal{M}^{+}(\T,\R^{3\times 3}_{\rm{sym}}),w))$ $\mathbf{P}$-a.s. we obtain
\begin{equation*}
\begin{aligned}
\int_{s}^{t} \int_{\T}\nabla e_{i}:\dif\tilde{\mathfrak{R}}_{n}(r)\dif r &= \langle \nabla e_{i}, \tilde y_{n}(t)-\tilde y_{n}(s) \rangle \to \langle \nabla e_{i}, \tilde y(t)-\tilde y(s)\rangle =\int_{s}^{t} \int_{\T}\nabla e_{i}:\dif\tilde{\mathfrak{R}}(r)\dif r,
\end{aligned}
\end{equation*}
where the convergence takes place in $C([0,\infty))$ $\mathbf{P}$-a.s. Thus (M2) follows.

  Since $\tilde{\mathfrak R}_n$ is a measurable function of $\tilde{y}_n$ and hence
  $\tilde{\mathfrak{N}}_n$ is a measurable function of $(\tilde{x}_n, \tilde{y}_n)$,
  we obtain from the equality of joint laws
  \begin{equation}
    \tmmathbf{P} \left( \int_{\mathbb{T}^3} \dif\tr \tilde{\mathfrak R}_n (t)
    \leqslant \tilde{z}_n (t) \  \mbox{for a.e.}\  t
    \geqslant s_n \right) = P_n \left( \int_{\mathbb{T}^3} \dif\tr\mathfrak  R (t)
    \leqslant z (t) \  \mbox{for a.e.} \  t \geqslant s_n
    \right) = 1, \label{eq:991}
  \end{equation}
  \begin{equation}
    \tmmathbf{P} \left(\tilde{\mathfrak N}_n \in L^{\infty}_{\rm{loc}} ([s_n, \infty) ; \mathcal{M}^+
    (\mathbb{T}^3, \mathbb{R}^{3 \times 3}_{\tmop{sym}}))\right) = P_n \left(\mathfrak N \in
    L^{\infty}_{\rm{loc}} ([s_n, \infty) ; \mathcal{M}^+ (\mathbb{T}^3, \mathbb{R}^{3 \times
    3}_{\tmop{sym}}))\right) = 1, \label{eq:992}
  \end{equation}
  and as in Lemma~\ref{lem:def} for every $N \in \mathbb{N}$
 \begin{equation}\label{eq:F1}
 \sup_{n \in \mathbb{N}} \E^{\tmmathbf{P}} \left[ \esssup_{t \in [0,
     N]}\left( \int_{\mathbb{T}^3} \dif\tr \tilde{\mathfrak R}_n (t) \right)^{2}\right] = \sup_{n
     \in \mathbb{N}} \E^{P_n} \left[ \esssup_{t \in [0, N]} \left(
     \int_{\mathbb{T}^3}\dif\tr \mathfrak  R (t) \right)^{2}\right] < \infty .
     \end{equation}

 Hence, by Banach--Alaoglu's theorem applied in the dual space (see \cite[Theorem 2.11]{MNRR96} and \cite[Theorem 8.20.3]{Ed65})
 $$
 L^{2}_{w}(\Omega;L^{\infty}_{w}(0,N;\mathcal{M}(\T;\R^{3\times 3}_{\rm{sym}})))\simeq \left( L^{2}(\Omega ; L^{1}(0,N;C(\T;\R^{3\times 3}_{\rm{sym}}))\right)^{*}
 $$ where the subscript $w$ stands for weak-star measurable mappings, we deduce that there exists $F\in L^{2}_{w}(\Omega;L^{\infty}_{w}(0,N;\mathcal{M}(\T;\R^{3\times 3}_{\rm{sym}})))$ such that
 \begin{equation}\label{eq:F2}
\tilde{ \mathfrak{R}}_{n}\to F \ \mbox{weak-star in}\ L^{2}_{w}(\Omega;L^{\infty}_{w}(0,N;\mathcal{M}(\T;\R^{3\times 3}_{\rm{sym}}))).
\end{equation}
On the other hand, as a consequence of the convergence of $\tilde{y}_n$ to
  $\tilde{y}$ in (ii) we deduce
  $$  \tilde{\mathfrak R}_n \rightarrow
     \tilde{\mathfrak R} \hspace{1em} \tmop{in}\
     \mathcal{D}' ((0, \infty)\times\T) \  \tmmathbf{P} \mbox{-a.s.} $$
   Thus, we get $F=\tilde{\mathfrak{R}}$. Even though $F=F_{N}$ is defined for times $t\in[0,N]$, taking $N\to\infty$ we may extend $F$ to $[0,\infty)$ using uniqueness of the limit.
Since all $\tilde{\mathfrak{R}}_{n}$ are positive semidefinite, the same remains valid for $F$  and  the corresponding norm of $F$ is bounded by the left hand side of \eqref{eq:F1} by weak-star lower semicontinuity.
Therefore we have in particular
$$ \tmmathbf{P} \left( \int_{\mathbb{T}^3} \dif\tr \tilde{\mathfrak R} \in
     L^{\infty}_{\rm{loc}} ([s, \infty)) \right) = 1. $$

  Therefore, as
  $$ \tilde{z}_n - \int_{\mathbb{T}^3} \dif\tr \tilde{\mathfrak R}_n $$
  is a non-negative distribution $\tmmathbf{P}$-a.s. by {\eqref{eq:991}}, the
  same remains valid for the limit
  $$ \tilde{z} - \int_{\mathbb{T}^3} \dif\tr \tilde{\mathfrak R} . $$
  Since this is an $L^{\infty}_{\rm{loc}} ([s, \infty))$-function $\tmmathbf{P}$-a.s., we
  deduce that
  \begin{equation*}
    \tmmathbf{P} \left( \int_{\mathbb{T}^3} \dif\tr \tilde{\mathfrak R} (t)
    \leqslant \tilde{z} (t) \  \mbox{for a.e.} \  t
    \geqslant s \right) = 1. 
  \end{equation*}
  Hence (M3) follows.
  Finally, it remains to verify the second condition in (M1). To this end, we
  take $\eta \in \mathbb{R}^3$ and write
  $$ \tilde{\mathfrak N} : (\eta \otimes \eta) = \tilde{\mathfrak R} : (\eta \otimes \eta)
     - | \tilde{x} \cdummy \eta |^2 = \lim_{n \rightarrow \infty}
     [\tilde{\mathfrak R}_n : (\eta \otimes \eta) - | \tilde{x}_n \cdummy \eta |^2]
     + \lim_{n \rightarrow \infty} [| \tilde{x}_n \cdummy \eta |^2 - |
     \tilde{x} \cdummy \eta |^2], $$
  where the limit is taken in the sense of distributions $\mathcal{D}' ([0,
  \infty) \times \mathbb{T}^3)$ $\tmmathbf{P}$-a.s. According to
  {\eqref{eq:992}}, the first limit on the right hand side is non-negative,
  whereas the second limit is a non-negative due to the weak lower
  semicontinuity of the convex function $x \mapsto | x \cdummy \eta |^2$.
  Thus, $\tilde{\mathfrak N}$ is a positive semidefinite matrix-valued measure. In
  addition, by weak lower semicontinuity and equality of laws, we obtain
  $$ \E^{\tmmathbf{P}} \left[\sup_{t \in [0, N]} \| \tilde{x} (t) \|_{L^2}^2\right]
     \leqslant \liminf_{n \rightarrow \infty} \E^{\tmmathbf{P}} \left[\sup_{t \in
     [0, N]} \| \tilde{x}_n (t) \|_{L^2}^2\right]
      = \liminf_{n \rightarrow \infty}
     \E^{P_n} \left[\sup_{t \in [0, N]} \| x (t) \|_{L^2}^2\right] < \infty . $$
  Hence, in view of  the definition of $\tilde{\mathfrak{N}}$ we conclude
  $$ \tmmathbf{P} \left(\tilde{\mathfrak N} \in L^{\infty} _{\rm{loc}}([s, \infty) ; \mathcal{M}^+
     (\mathbb{T}^3, \mathbb{R}^{3 \times 3}_{\tmop{sym}}))\right) = 1, $$
  which completes the proof.
\end{proof}

In the same way, we can prove stability for dissipative martingale solutions. To this end, we denote by $\mathscr{C}(s,x_0,y_0,z_0)$ the set of all dissipative martingale solutions with the initial condition $(x_0,y_0,z_0)$ and the initial time $s$.

\bt\label{convergence1} Suppose that \emph{(G1)}, \emph{(G2)} hold.
   Let $(x_n, y_n, z_n)\in\mathbb{X}$, $s_{n}\in [0,\infty)$, $n\in\N$, and assume that
  $$
  (s_n, x_n, y_n, z_n) \rightarrow (s_0, x_0, y_0, z_0) \ \mbox{in}\
  [0, \infty) \times L^2_{\sigma} \times \mathcal{M}^+ (\mathbb{T}^3,
  \mathbb{R}^{3 \times 3}_{\tmop{sym}}) \times [0, \infty)
  $$ as $n
  \rightarrow \infty$ and let $P_n \in \mathscr C (s_n, x_n, y_n, z_n)$. Then
  there exists a subsequence $n_k$ such that the sequence $(P_{n_k})_{k \in
  \mathbb{N}}$ converges weakly to some $P \in \mathscr C (s_0, x_0, y_0, z_0)$.
\et

\begin{proof}
The proof is a consequence  of Theorem~\ref{convergence}. More precisely, by the martingale representation theorem, see \cite{DPZ92}, the dissipative martingale solution $P_{n}$ gives raise to a dissipative probabilistically weak solution  $Q_{n}\in\mathscr C_{PW}(s_n, x_n, y_n, z_n,0)$, such that the projection of $Q_{n}$ on the first three components is $P_{n}$. By Theorem~\ref{convergence} there is  a subsequence $(Q_{n_{k}})_{k\in\N}$ converging weakly to some $Q\in \mathscr C_{PW}(s_0, x_0, y_0, z_0,0)$. Then it is easy to see that the projection of $Q$ on the first three components belongs to $\mathscr{C}(s_0, x_0, y_0, z_0)$.
\end{proof}

\subsection{Existence}
\label{s:ex}

Based on the proof of stability in Theorem~\ref{convergence}, we also obtain existence of dissipative probabilistically weak solutions and as a corollary existence of dissipative martingale solutions.

\bt\label{existence} Suppose that \emph{(G1)}, \emph{(G2)} hold.
For every $s\in[0,\infty)$, $(x_0,y_0,z_{0})\in\mathbb{X}$ and $b_{0}\in U_{1}$, there exists  $P\in\mathscr{P}(\Omega_{M})$ which is a dissipative probabilistically weak solution to the Euler system \eqref{el} starting at time $s$ from the initial condition $(x_0,y_0,z_{0},b_0)$.

\et

\begin{proof}
Consider the following Galerkin approximation through stochastic Navier--Stokes equations with vanishing viscosity
\begin{equation*}
\aligned
 \dif u_n-\frac{1}{n}\Delta u_n \dif t+\Pi_n\mathbb{P}\div(u_n\otimes u_n)\dif t&=\Pi_nG(u_n)\dif B,
\\\div u_n&=0,
\\u_n(t)&=\Pi_nx_0,\quad 0\leq t\leq s,
\endaligned
\end{equation*}
where $\Pi_n$ and $\mathbb{P}$, respectively, are the Galerkin and Leray projection operator, respectively. It is classical to show that a solution exists on some probability space $(\Omega,\mathcal{F},\bf{P})$ with a cylindrical Wiener process $B$  on $U$ starting from $b_0$ at time $s$. For notational simplicity and without loss of generality, we may assume that the probability space and the Wiener process do not depend on $n$.

Define for $t\geq s$
\begin{equation}
\label{eq:zn}
\begin{aligned}
z_n(t)&:=z_{0}+2\int_s^t\langle u_n,G(u_n(r))\dif B(r)\rangle+\int_s^t\|\Pi_nG(u_n(r))\|_{L_2(U,L^2)}^2\dif r,\\
\mathfrak{R}_{n}&:=u_{n}\otimes u_{n},\qquad\qquad y_{n}(t) := y_{0}+\int_{s}^{t} \mathfrak{R}_n(r)\dif r,
\end{aligned}
\end{equation}
and  denote by  $P_n$  the joint law of $(u_n,y_{n},z_n,B) $. By an application of It\^o's formula to the function $u\mapsto \|u\|_{L^{2}}^{2}$ we obtain in particular
$$
\mathbf{P}\left(\|u_{n}(t)\|_{L^{2}}^{2}\leq z_{n}(t)\ \mbox{for all}\ t\geq s\right)=\mathbf{P}\left(\int_{\T}\dif\tr\mathfrak{R}_{n}(t)\leq z_{n}(t)\ \mbox{for all}\ t\geq s\right)=1,
$$
and estimating the right hand side of the energy equality in \eqref{eq:zn} by Burkholder--Davis--Gundy's inequality and (G1), we obtain for every $N\in\N$
$$
\sup_{n\in\N}\E^{\bf{P}}\left[\sup_{t\in[0,N]}\|u_{n}(t)\|_{L^{2}}^{2p}+\left(\frac{1}{n}\int_{s}^{N}\|\nabla u_{n}(r)\|_{L^{2}}^{2}\dif r\right)^{p}\right]<\infty.
$$

Therefore, the processes $(u_{n},y_{n},z_n,B) $, $n\in\N$, satisfy exactly the same uniform bounds as in the proof of Theorem~\ref{convergence} and their tightness follows exactly in the same way. The identification of the limit is also the same, the only difference being the artificial viscosity term which vanishes in the asymptotic limit.
\end{proof}

\section{Non-uniqueness in law}\label{s:law}

This section is devoted to the proof of non-uniqueness in law in the case of an additive noise. In particular, we consider the stochastic Euler system
\begin{equation}\label{ela}
\aligned
 \dif u+\div(u\otimes u) \dif t+\nabla P \dif t&=G \dif B,
\\
\div u&=0,
\endaligned
\end{equation}
where the coefficient $G$ satisfies the hypotheses (G3), (G4). The proof follows in three main steps. First, in Section~\ref{s:ci} we apply the convex integration method based on Baire's category theorem in order to construct infinitely many adapted weak solutions to  \eqref{ela} satisfying an energy inequality. Based on a general construction developed in Section~\ref{s:ext}, we show in Section~\ref{s:appl} that these convex integration solutions give raise to simplified probabilistically weak solutions defined on the full time horizon $[0,\infty)$. With this in hand, we are able to complete the proof of the main result of this section which reads as follows.

\begin{theorem}\label{Main results2}
Suppose that  \emph{(G3)}, \emph{(G4)} hold. Then simplified dissipative martingale solutions to \eqref{ela} are not unique. Moreover, for any given $T>0$, non-uniqueness holds on $[0,T]$.
\end{theorem}

We note that the restriction to the additive noise case satisfying (G3) is required for the construction by convex integration performed in Section~\ref{s:ci}. The results of Section~\ref{s:ext} apply to a more general multiplicative noise whereas the assumption (G4) is required for their application to the convex integration solutions in Section~\ref{s:appl}.

\subsection{Construction by convex integration}\label{s:ci}

In this subsection, we use the convex integration method in order to find an initial condition which gives raise to infinitely many weak solutions satisfying energy inequality, all adapted to the canonical filtration generated by a given Wiener process $B$. In particular, we fix a probability space $(\Omega,\mathcal{F},\mathbf{P})$ with a cylindrical Wiener process $B$ on $U$ satisfying $B(0)=0$ and let $(\mathcal{F}_{t})_{t\geq 0}$ be its normal filtration. We recall that $(\mathcal{F}_{t})_{t\geq 0}$ is the canonical filtration augmented by $\mathbf{P}$-null sets and that it is right continuous. Therefore, the $\sigma$-algebra $\mathcal{F}_{0}$ is generated by the $\mathbf{P}$-null sets, i.e., $\mathcal{F}_{0}=\sigma\{A\in\mathcal{F};\mathbf{P}(A)=0\}$ and, as a consequence, every $\mathcal{F}_{0}$-measurable random variable is $\mathbf{P}$-a.s. constant.
Moreover,  we restrict ourselves to the additive noise case.
In particular, we fix parameters $$ 0<\delta<1/4, \quad p\in(1,\infty),\quad \beta\in (0,1),\quad\frac{1}{2}-2\delta +\frac{1}{p}<\beta<\frac{1}{2}-\delta,$$
and note that under (G3),  it holds  for any $T>0$
$$
 \|GB\|_{C_TH^{(3+\sigma)/2}}<\infty,\quad \|GB\|_{C_T^{1/2-\delta}H^1}<\infty\quad \mathbf{P}\mbox{-a.s.},$$
$$\left\|\int_0^\cdot \langle GB,G\dif B\rangle\right\|_{C^{1/2-2\delta}_T}\lesssim\left\|\int_0^\cdot \langle GB,G\dif B\rangle\right\|_{W^{\beta,p}_T}\lesssim\left\|\int_0^\cdot \langle GB,G\dif B\rangle\right\|_{C^{1/2-\delta}_T}<\infty\quad \mathbf{P}\mbox{-a.s.}$$
For a given  $L>1$  we  define the stopping times as
$$
T^1_L=\inf\left\{t\geq 0: \|GB(t)\|_{H^{(3+\sigma)/2}}\geq L\right\}\wedge \inf\left\{t\geq 0: \|GB\|_{C^{1/2-2\delta}_tH^1}\geq L\right\}\wedge L,$$
$$T^2_L= \inf\left\{t\geq 0: \left\|\int_0^\cdot \langle GB,G\dif B\rangle\right\|_{W^{\beta,p}_t}\geq L\right\}\wedge L,$$

\begin{equation}
\label{stopping time}
T_L=T^1_L\wedge T^2_L,
\end{equation}
and we let $B_{L}$ be the stopped Wiener process $B_{L}(\cdot)=B(\cdot\wedge T_{L})$.

This leads us to the truncated Euler system
\begin{equation}\label{eq:trunc}
\begin{aligned}
\dif u +\div (u\otimes u) \dif t +\nabla P \dif t &= G \dif B_{L},\\
\div u&=0,
\end{aligned}
\end{equation}
which we consider on the time interval $[0,T]$ for a fixed $T>0$. The truncated system coincides with the original system \eqref{ela} on the random time interval $[0,T_{L}]$, hence the solutions constructed in this section solve the stochastic Euler system \eqref{ela} up to the stopping time $T_{L}$. In addition, for a suitable (deterministic) initial condition $u_{0}$ and for every choice of an additional parameter $l\in[2,\infty]$, we construct at least one  solution $u$
satisfying the following energy equality for $\mathbf{P}$-a.s. and a.e. $t\in(0,T_{L}]$
\begin{align}\label{eq:energy1}
\frac{1}{2}\|u(t)\|_{L^2}^2=&\frac{1}{2}\|u_{0}\|_{L^2}^2+ M^E_{t,0}+\left(\frac{1}{2}-\frac{1}{l}\right)(t\wedge T_{L})\|G\|_{L_2(U,L^2)}^2,
\end{align}
where $M^E_{t,0}=\int_0^t\langle u, G \dif B_L\rangle$.  In view of the Definition~\ref{martingale solution1}, we therefore quantify the error in the energy equality through the defect
$\frac{1}{l}\|G\|_{L_2(U,L^2)}^2(t\wedge T_L).$
This permits us to conclude the existence of infinitely many solutions, i.e., at least one solution for every $l\in[2,\infty]$. Even though for a fixed $l$, the Baire category argument used below in fact yields the existence of infinitely many solutions to the truncated equation \eqref{eq:trunc} on $[0,T]$, we are not able to deduce that the non-uniqueness holds for the original equation, i.e. already on the random time interval $[0,T_{L}]$ with probability one. Therefore, we use different values of $l\in[2,\infty]$ to conclude the non-uniqueness.

 We rewrite the Euler system \eqref{eq:trunc} with additive noise by setting $v=u-GB_{L}$, which gives
\begin{equation}\label{eq:v}
\begin{aligned}
\partial_{t}v +\div((v+GB_{L})\otimes (v+GB_{L}))+\nabla P & =0,\\
\div v &= 0.
\end{aligned}
\end{equation}
This is the system we apply the convex integration to.
We define the energy functional
\begin{align}
\label{con:e}
e_l(v)(t):=\frac{1}{2}\|v(0)\|_{L^2}^2+ \int_0^t\langle v+GB_L, G \dif B_L\rangle+\left(\frac{1}{2}-\frac{1}{l}\right)(t\wedge T_{L})\|G\|_{L_2(U,L^2)}^2.
\end{align}
We note that if  a sequence $v_n$ satisfies $v_n(0)=v(0)$ and
$$
\esssup_{\omega\in\Omega}\|v_n-v\|_{C^{\alpha}_{T}H^{-1}}\to 0
$$
for some $\alpha>1/2+2\delta$, then due to Lemma~\ref{lem:young} we obtain
$$
\esssup_{\omega\in\Omega}\sup_{t\in[0,T]}\left|\int_0^t\langle v_n-v, G \dif B_L\rangle\right|\lesssim \esssup_{\omega\in\Omega} \|v_n-v\|_{C^\alpha_{T}H^{-1}}\esssup_{\omega\in\Omega}\|GB_{L}\|_{C^{1/2-2\delta}_{T}H^{1}}\rightarrow 0,$$
which implies that
\begin{equation}
\label{enconvergence}
\esssup_{\omega\in\Omega}\sup_{t\in[0,T]}|e_l(v_n)\rightarrow e_l(v)|\to 0.
\end{equation}

Let us now introduce some preliminary notations and definitions needed in the sequel. For the  filtration $(\mathcal{F}_{t})_{t\geq 0}$ we define its extension to negative times by $\mathcal{F}_{t}=\mathcal{F}_{0}$ whenever $t<0$. We denote by $\mathbb{R}^{3\times 3}_{0,\rm{sym}}$ the set of symmetric trace-less $3\times 3$ matrices and for $(w,H)\in \mathbb{R}^{3}\times\mathbb{R}^{3\times 3}_{0,\rm{sym}}$ we denote
$$
e(w,H)= \frac{3}{2}\lambda_{\rm{max}}\left(w\otimes w -H\right),
$$
where $\lambda_{\rm{max}}(Q)$ is the largest eigenvalue of a symmetric $3\times 3$ matrix $Q$. By \cite[Lemma 3]{DelSze3}, we have
\begin{equation}\label{boundwh}
\frac{1}{2}|w|^2\leq e(w,H),\quad \|H\|\leq \frac{4}{3}e(w,H),
\end{equation}
where $\|\cdot\|$ means the operator norm of the matrix.
Finally, we denote by $d$ a metric which metrizes the weak topology on bounded sets of $L^{2}$. We say that a Borel random variable  $F : \Omega \to  X$ ranging in a topological  $X$ has a compact range provided there is a (deterministic) compact set $K\subset X$ such that $F\in  K$ $\mathbf P$-a.s.

The following result is a modification of Lemma 5.11 from \cite{BFH20}. In particular, it permits us to construct an initial condition with a prescribed energy, which gives raise to a subsolution.

\bl\label{lem:osc}
 Let $[e,w,H]$ be an $(\mathcal{F}_t)_{t\geq 0}$-adapted stochastic process such that
 $$[e,w,H]\in C([0,T]\times \mathbb{T}^3; (0,\infty)\times\R^3\times \R^{3\times 3}_{0,\rm{sym}})\quad \mathbf{P}\mbox{-a.s.} $$
 with compact range and
 \begin{equation}\label{eq:222}
 e(w,H)<e-\delta,\quad \textrm{ for all } (t,x)\in [0,T]\times \mathbb{T}^3\quad \mathbf{P}\mbox{-a.s.}
 \end{equation}
for some deterministic constant $\delta>0$. Then for any $\varepsilon\in(0,T)$ there exists a sequence $[w_n,V_n]\in C^{\infty}_{c}((-T,T)\times \mathbb{T}^3; \R^3\times \R^{3\times 3}_{0,\rm{sym}})$  enjoying the following properties:

 i) the process $[w_n,V_n]$ is $(\mathcal{F}_t)_{t\in\mathbb{R}}$-adapted such that $[w_{n},V_{n}]\in C([-T,T]\times\mathbb{T}^3; \R^3\times \R^{3\times 3}_{0,\rm{sym}})$ $\mathbf{P}$-a.s. with compact range and $\supp( w_n,V_n)\subset [-\varepsilon,\varepsilon]\times \mathbb{T}^3$ $\mathbf{P}$-a.s.;

 ii) we have $\mathbf{P}$-a.s. $$\partial_t w_n+\div V_n=0, \quad \div w_n=0;$$

 iii) we have
 $$\esssup_{\omega\in\Omega}\sup_{t\in[0,T]} d(w_n(t),0)\rightarrow0 \textrm{ as } n\rightarrow\infty;$$

 iv) we have  $$e(w+w_n,H+V_n)<e-\delta_n,\quad \textrm{ for all } (t,x)\in [0,T]\times \mathbb{T}^3\quad \mathbf{P}\mbox{-a.s.}$$
 for some deterministic constant $\delta_n>0$;

 v) if $\bar{e}>0$ is such that $e\leq \bar{e}$, then the following holds $\mathbf{P}$-a.s.
 $$\liminf_{n \rightarrow \infty} \int_{\mathbb{T}^3}\frac12| w+w_n |^2(0,x) \dif x\geqslant \int_{\mathbb{T}^3}\frac12| w |^2(0,x)\dif x+\frac{c}{\bar{e}}
\int_{\mathbb{T}^3} \left( e - \frac{1}{2} | w |^2 \right)^{2}(0,x) \dif x$$
  for some universal constant $c>0$.

  \el

  \begin{proof}
Comparing to Lemma 5.11 in \cite{BFH20} we see that the new result is \emph{v)}. In order to prove it, we need to construct oscillations around the time $t=0$, which requires an extension of the functions $[e,w,H]$ to negative times. The chosen value for negative times is not important provided the extended process remains adapted to $(\mathcal{F}_{t})_{t\in\mathbb{R}}$. Therefore, we  set $[e,w,H](t)=[e,w,H](0)$ whenever $t<0$. Note that since we extended the filtration to negative times in the same way, $[e,w,H]$ is $(\mathcal{F}_{t})_{t\in\mathbb{R}}$-adapted and in particular $[e,w,H](t)$ is $\mathbf{P}$-a.s. deterministic for $t\leq 0$.

As the proof of Lemma 5.11 in \cite{BFH20}, we intend to approximate $[e,w,H]$ by piecewise constant functions $[e^{\rm{ap}},w^{\rm{ap}},H^{\rm{ap}}]$ such that for a certain (deterministic) $\gamma>0$  it holds
\begin{equation}\label{eq:split}
|[e^{\rm{ap}},w^{\rm{ap}},H^{\rm{ap}}](t,x)-[e,w,H](t,x)|<\gamma
\end{equation}
$\mathbf{P}$-a.s. for all $(t,x)\in [-T,T]\times\mathbb{T}^{3}$.
 Moreover, we choose $\gamma>0$ so that
\begin{equation}\label{eq:223}
\left|\left(e^{\rm{ap}} - \frac12 |w^{\rm{ap}}|^{2}\right)(t,x)-\left(e-\frac12|w|^{2}\right)(t,x)\right|<\frac{\delta}{2}.
\end{equation}
$\mathbf{P}$-a.s. for all $(t,x)\in [-T,T]\times\mathbb{T}^{3}$.
Choosing such a $\gamma>0$ deterministic is possible  due to  the assumption of  compact range in the space of continuous functions for the process $[e,w,H]$. Indeed, it implies that $[e,w,H]$ is uniformly continuous with respect to  $(t,x)$ (essentially) uniformly in $\omega$.
Therefore, for a given $\varepsilon>0$ we split the space-time domain $(-\varepsilon,\varepsilon]\times\mathbb{T}^{3}$ into finitely many disjoint boxes of the form $(t_{j},t_{j+1}]\times (a_{i},b_{i}]=:(t_{j},t_{j+1}]\times K_{i}$, where $\{t_{j};\,j=1,\dots,j_{\rm{max}}\}$ is a partition of $ [-\varepsilon,\varepsilon]$, $a_{i}=(a_{i,1},a_{i,2},a_{i,3}),\,b_{i}=(b_{i,1},b_{i,2},b_{i,3})\in \mathbb{T}^{3}$ for $i=1,\dots,i_{\rm{max}}$, and $(a_{i},b_{i}]=\Pi_{\ell=1}^{3}(a_{i,\ell},b_{i,\ell}]$. On each of these boxes, the piecewise constant approximation is defined as
$$
[e^{\rm{ap}},w^{\rm{ap}},H^{\rm{ap}}](t,x):=[e,w,H](t_{j},y_{i}) \ \mbox{for some}\ y_{i}\in K_{i}
$$
whenever $(t,x)\in(t_{j},t_{j+1}]\times K_{i}$. Note that this choice in particular preserves adaptedness. In addition, the boxes are chosen in a way that the time $t=0$ appears in the middle of one of the time intervals, say $(t_{j^{*}},t_{j^{*}+1}]$,  and that $\mathbf{P}$-a.s. \eqref{eq:split} holds true.

Using this approximation  in the proof of \cite[Lemma 5.11]{BFH20},  \emph{i)-iv)} from the statement of the present lemma follow and it only remains to prove $\emph{v)}$. To this end, we observe that on each of the above boxes we can construct oscillations by following the approach of   \cite[Lemma~5.6]{BFH20}. Moreover, if $w_{n}$, $n\in\N$, is the sequence of oscillations on the box $(t_{j^{*}},t_{j^{*}+1}]\times K_{i}$  then the arguments of Step 5 in the proof of \cite[Lemma 5.11]{BFH20} also imply
\begin{equation}\label{new:5.23}
  \liminf_{n \rightarrow \infty} \int_{K_{i}} | w_n |^2 (0, x) \dif x
  \geqslant \frac{c}{e(t_{j^{*}},y_{i})} \left( e(t_{j^{*}},y_{i}) - \frac{1}{2} | w(t_{j^{*}},y_{i}) |^2 \right)^2 \frac{
  |K_{i} |}{2} ,
\end{equation}
where $c>0$ is a universal constant.

The oscillations are in particular compactly supported within the interior of the corresponding box. Thus, they can be connected to  obtain the desired oscillatory perturbations $w_{n}$, $n\in\N$, of $w^{\rm{ap}}$ on the whole domain $(-\varepsilon,\varepsilon)\times\mathbb{T}^{3}$. As a consequence, it holds
\begin{equation}\label{new:5.42}
\begin{aligned}
\liminf_{n \rightarrow \infty} \int_{\mathbb{T}^3} | w_n |^2 (0, x) \dif x& =
   \liminf_{n \rightarrow \infty} \int_{\cup_{i=1}^{i_{\rm{max}}} K_{i}} | w_n |^2 (0, x) \dif x =
   \liminf_{n \rightarrow \infty} \sum_{i = 1}^{i_{\max}}
   \int_{K_{i}} | w_n |^2 (0, x) \dif x\\
  &\geqslant \sum_{i = 1}^{i_{\max}} \liminf_{n \rightarrow \infty}
  \int_{K_{i}} | w_n |^2 (0, x) \dif x
  \geqslant \sum_{i = 1}^{i_{\max}}
  \frac{c}{e} \left( e - \frac{1}{2} | w |^2 \right)^2 (t_{j^{*}}, y_{i})\frac{|
  K_{i} |}{2} \\
  &\geqslant  \frac{c}{\bar e}\int_{\mathbb{T}^3} \left( e^{\rm{ap}} - \frac{1}{2} | w^{\rm{ap}}
  |^2 \right)^2 (0, x) \dif x.
  \end{aligned}
\end{equation}
Thus, we obtain
\begin{equation*}
\begin{aligned}
\liminf_{n \rightarrow \infty} \int_{\mathbb{T}^3}\frac12|w+w_{n}|^{2}(0,x) \dif x &= \int_{\mathbb{T}^3}\frac12|w|^{2}(0,x) \dif x+\liminf_{n \rightarrow \infty} \int_{\mathbb{T}^3} \frac12| w_n |^2 (0, x) \dif x\\
&\geqslant \int_{\mathbb{T}^3}\frac12|w|^{2}(0,x) \dif x+ \frac{c}{2\bar e} \int_{\mathbb{T}^3} \left(e^{\rm{ap}} - \frac{1}{2} | w^{\rm{ap}}
  |^2 \right)^{2} (0, x) \dif x\\
  &\geqslant \int_{\mathbb{T}^3}\frac12|w|^{2}(0,x) \dif x+\frac{c}{2\bar e} \int_{\mathbb{T}^3} \left(e - \frac{1}{2} | w
  |^2 -\frac{\delta}{2}\right)^{2} (0, x) \dif x,
\end{aligned}
\end{equation*}
where in the first equality we used \emph{iii)} and the last inequality follows from \eqref{eq:223} together with the fact that by \eqref{eq:222} and \eqref{boundwh} we have
$$
\left(e-\frac12 |w|^{2}\right)(0,x)  >\delta
$$
$\mathbf{P}$-a.s. for all $x\in\T$.
The same argument permits to finally conclude
$$
\liminf_{n \rightarrow \infty} \int_{\mathbb{T}^3}\frac12|w+w_{n}|^{2}(0,x) \dif x \geqslant  \int_{\mathbb{T}^3}\frac12|w|^{2}(0,x) \dif x+\frac{c}{ 8\bar e} \int_{\mathbb{T}^3} \left(e - \frac{1}{2} | w
  |^2 \right)^{2} (0, x) \dif x,
$$
which completes the proof.
  \end{proof}

Let $e$ be a given $(\mathcal{F}_{t})_{t\geq0}$-adapted energy such that $e\in C([0,T]\times\mathbb{T}^{3};(0,\infty))$ $\mathbf{P}$-a.s. with a compact range. We define the collection of subsolutions corresponding to $e$ by
\begin{align*}
X_{0,e}&=\Big\{ v:\Omega\rightarrow C_{\rm{loc}}((0,T]\times \mathbb{T}^3;\mathbb{R}^3)\cap C([0,T],L^2_w);\, v \textrm{ is } (\mathcal{F}_t)_{t\geq0}\textrm{-adapted,}\\\
&\qquad\textrm{there exists }
H: \Omega\rightarrow C_{\rm{loc}}((0,T]\times \mathbb{T}^3;\mathbb{R}^{3\times3}_{0,\rm{sym}})\ (\mathcal{F}_t)_{t\geq0}\textrm{-adapted such that }\\
&\qquad
\partial_t v+\div H=0, \ \div v=0,\\
&\qquad \mbox{for every }\varepsilon\in(0,T)\ \mbox{there is deterministic } \delta_{\varepsilon}>0\ \mbox{such that}\\
&\qquad e(v+GB_{L},H)<e-\delta_\varepsilon\textrm{ for all }t\in [\varepsilon,T],\\
&\qquad(v,H)\textrm{ as a function in }C([\varepsilon,T]\times \mathbb{T}^3;\mathbb{R}^{3}\times \mathbb{R}^{3\times3}_{0,\rm{sym}}) \textrm{ is of compact range}\Big\}.
\end{align*}

The next result shows how to find a subsolution with a prescribed energy at time $t=0$.

\bl \label{lem:osc1}
Let $e$ be $(\mathcal{F}_{t})_{t\geq 0}$-adapted such that $e\in C([0,T];(0,\infty))$ $\mathbf P$-a.s.  with compact range and for some deterministic $\delta>0$
 \begin{equation}
 \label{con:e1}
 e(GB_{L},0)<e-\delta\quad\mbox{for all }t\in[0,T].
 \end{equation}
Then there exists $\bar{v}\in X_{0,e}$ such that
 $$\frac{1}{2}\int_{\mathbb{T}^{3}}| \bar{v}(0) |^2 \dif x= e(0)\quad \mathbf P\mbox{-a.s}.$$
 \el

 \begin{proof} The idea of the proof  follows from \cite[Section 5]{DelSze3}.
 We proceed iteratively and apply Lemma \ref{lem:osc}. In particular, letting $v_0=0$, $H_0=0$ we have
$v_{0}\in X_{0,e}$.  By the definition of the stopping time $T_{L}$ we know that $GB_L\in C([0,T]\times \T)$ $\mathbf{P}$-a.s. is of compact range. Then applying repeatedly Lemma \ref{lem:osc} to $[e,v_k+GB_L,H_k]$ we  construct a sequence $v_k=v_{k-1}+w_{k,n}$ for $n$ large enough, $k\in\N$, so that
$$
  v_k \in X_{0,e}, \quad
\esssup_{\omega\in\Omega}\sup_{t\in[0,T]}  d (v_k, v_{k - 1}) < \frac{1}{2^k}, \quad e(v_k+GB_L,H_k)<e-\delta_k,\, t\in[0,T],
$$
$$(v_k,H_k)\textrm{ as a function in }C([0,T]\times \mathbb{T}^3;\mathbb{R}^{3}\times \mathbb{R}^{3\times3}_{0,\rm{sym}}) \textrm{ is of compact range,}
  $$
  $$ \esssup_{\omega\in\Omega} \sup_{t\in[0,T]} \left| \int_{\T} (v_k -
  v_{k - 1}) \cdot v_m \dif x \right| < \frac{1}{2^k} \quad \textrm{for all } m = 0, \ldots, k - 1,
$$
\begin{equation}
\label{supp}
\textrm{supp}(v_k-v_{k-1})\subset[-1/2^k,1/2^k]\times\mathbb{T}^3,
\end{equation}
and $\mathbf{P}$-a.s.
\begin{equation}\label{alpha}
 \int_{\T} \frac{1}{2} | v_k|^2 (0,x) \dif x  \geqslant \int_{\T} \frac{1}{2} | v_{k -
  1} |^2 (0,x) \dif x + c \alpha_k^2,
\end{equation}
where
$$ \alpha_k = \int_{\T} \left( e - \frac{1}{2} | v_{k - 1} |^2 \right) (0,x)\,\dif x
   > 0 .$$
We recall  that since the $\sigma$-algebra $\mathcal{F}_{0}$ is generated by the $\mathbf{P}$-null sets,  the random variables $e$, $v_{k}$,  and consequently also $\alpha_{k}$, $k\in\mathbb{N}_{0}$,  are $\mathbf{P}$-a.s. deterministic at time $t=0$.

Therefore, the sequence
$$
  \int_{\T} \frac{1}{2} | v_k |^2 (0,x) \dif x
$$
is $\mathbf{P}$-a.s. non-decreasing and there exists $\bar{v}$ adapted such that
$
\esssup_{\omega\in\Omega}\sup_{t\in[0,T]}d(v_{k},\bar{v})\to 0
$ as $k\to\infty$.
By \eqref{supp}, we know that for any $\varepsilon>0$ there exists $k_0\in\N$ such that
$$\bar{v}(t)=v_k(t)=v_{k_0}(t)\quad t\in[\varepsilon,T],\,k\geq k_0,$$
which implies that $\bar{v}\in X_{0,e}$.
By the construction of $v_{k}$  and \eqref{boundwh} and the compact range of $e$  we obtain in particular
$$
\sup_{k\in\N_{0}}\esssup_{\omega\in\Omega}  \int_{\T} \frac{1}{2} | v_k (0,x)|^2 \dif x < \infty .
$$
Hence $ \int_{\T} \frac{1}{2} | v_k (0,x)|^2 \dif x $  has a $\mathbf{P}$-a.s. limit $Y$ and  by \eqref{alpha} we deduce
$\alpha_k^2 \rightarrow 0$ $\mathbf{P}$-a.s. Consequently, it follows from the definition of $\alpha_k$ that $\mathbf{P}$-a.s.
$$
  \int_{\T} \frac{1}{2} | v_k |^2 (0) \dif x \uparrow  e(0).
$$
Next, we have for all $n > m$
$$
  \int_{\T} \frac{1}{2} | v_n - v_m |^2 (0) \dif x = \int_{\T} \frac{1}{2} | v_n
  |^2 (0) \dif x - \int_{\T} \frac{1}{2} | v_m |^2 (0) \dif x - 2 \int_{\T}
  \frac{1}{2} ((v_n - v_m) \cdot v_m )(0) \dif x,
$$
where
$$
\int_{\T} \frac{1}{2} ((v_n - v_m) \cdot v_m )(0) \dif x = \sum_{k = 0}^{n
   - m - 1} \int \frac{1}{2} ((v_{m + k + 1} - v_{m + k}) \cdot v_m)
   (0) \dif x < \sum_{k = 0}^{n - m - 1} \frac{1}{2^{m + k}} \rightarrow
   0
   $$
  $\mathbf{P}$-a.s. as $m\to\infty$.
Accordingly, we obtain the $\mathbf{P}$-a.s. strong convergence $v_k (0) \rightarrow \bar{v}
(0)$ in $L^2$, which implies
$$\int_{\T} \frac{1}{2} | \bar{v} |^2 (0) \dif x=e (0) \quad \mathbf P\mbox{-a.s.} $$
and completes the proof.
\end{proof}

Now, we define the collection of subsolutions with variable energy by taking $e_{l}$ given in \eqref{con:e} and setting
\begin{align*}
X_{0}&=\Big\{ v:\Omega\rightarrow C_{\rm{loc}}((0,T]\times \mathbb{T}^3;\mathbb{R}^3)\cap C([0,T],L^2_w);\, v \textrm{ is }(\mathcal{F}_t)_{t\geq 0}\textrm{-adapted with $v(0)=\bar{v}(0)$, }
\\
&\qquad\textrm{there exists } H: \Omega\rightarrow C_{\rm{loc}}((0,T]\times \mathbb{T}^3;\mathbb{R}^{3\times3}_{0,\rm{sym}})\ (\mathcal{F}_t)_{t\geq0}\textrm{-adapted such that }\\
&\qquad \partial_t v+\div H=0,\  \div v=0,\\
&\qquad  \mbox{for every }\varepsilon\in(0,T)\ \mbox{there is deterministic }\delta_{\varepsilon}>0\ \mbox{such that}\\
&\qquad e(v+GB_{L},H)<e_{l}(v)-\delta_\varepsilon\textrm{ for all }t\in [\varepsilon,T],\\
&\qquad (v,H)\textrm{ as a function in } C([\varepsilon,T]\times \mathbb{T}^3) \textrm{ is of compact range}\Big\}.
\end{align*}

As the next step, we find a suitable energy $e$ for the application of  Lemma~\ref{lem:osc1}.

\begin{lemma}\label{lem:osc2}
Suppose that \emph{(G3)} holds. There exists a deterministic function $e\in C([0,T];(0,\infty))$ satisfying \eqref{con:e1} such that the associated $\bar v$ constructed in Lemma~\ref{lem:osc1} belongs to $X_0$.
\end{lemma}

\begin{proof}
The constructed energy $e$ has to satisfy \eqref{con:e1} as well as
\begin{equation}\label{eq:224}
e(t)\leq e_{l}(\bar{v})(t)=\frac12\|\bar v(0)\|_{L^{2}}^{2}+\int_{0}^{t}\langle \bar v+GB_{L},G\dif B_{L}\rangle+\left(\frac12-\frac{1}{l}\right)(t\wedge T_{L})\|G\|^{2}_{L_{2}(U,L^{2})},
\end{equation}
for all $t\in(0,T]$. Lemma~\ref{lem:osc1} then in particular implies that $\bar v$ has the desired energy at time $0$, i.e.,
$$
\frac12\|\bar v(0)\|_{L^{2}}^{2}=e(0)=e_{l}(\bar{v})(0).$$
In order to find a suitable energy $e$ so that this can be achieved, we shall therefore find a lower bound for $e_{l}(\bar v)$ based on the a priori information we get for $\bar v$.

To this end, we recall that
if $\bar{v}$ was constructed from Lemma~\ref{lem:osc1} with a given energy $e$ then in particular there is $\bar{H}$ such that
$$
\partial_t \bar{v} + {\div} \bar{H} = 0 \textrm{ in } \mathcal{D}'((0,T)\times \T), \qquad e (\bar{v} +G B_L, \bar{H}) < e -
   \delta_{\varepsilon} \quad \mbox{for }t\in[\varepsilon,T].
   $$
Thus, by \eqref{boundwh} it follows
$$ \frac{1}{2} | \bar{v} +G B_L |^2 \leqslant e \quad\mbox{on } (0, T], \qquad
   \frac{1}{2} | \bar H | \lesssim e \quad\mbox{on } (0, T],
   $$
and we deduce
$$ |\bar v|\lesssim \sqrt{e} +C(L),\qquad \sup_{s,t\in[0,T],s\neq t}\frac{\|  \bar{v}(t)-  \bar{v}(s)\|_{H^{- 1}}}{|t-s|} \lesssim \sup_{t\in[0,T]}\|\bar H \|_{L^2} \lesssim
   \sup_{t\in[0,T]}e,
   $$
   which by interpolation implies that for any $\beta\in(0,1)$
   $$ \|\bar v\|_{C_T^\beta H^{-1}}\lesssim \sup_{t\in[0,T]}e^{\frac{1+\beta}{2}} +C(L).
   $$
   All the implict constants are independent of $\bar{v}$.
Next, for $\beta>1/2+2\delta$ we estimate one stochastic integral by Lemma \ref{lem:young} using the definition of the stopping time $T_{L}$ in \eqref{stopping time} as
$$
\left|\int_{0}^{t}\langle\bar v, G\dif B_{L}\rangle\right|\leq C(L)\left(\sup_{t\in[0,T]}e^{\frac{1+\beta}{2}}+1\right)t^{1/2-2\delta}
$$
while for the other stochastic integral we get by the definition of the stopping time $T_{L}$
$$
\left|\int_{0}^{t}\langle GB_{L}, G\dif B_{L}\rangle\right|\leq C(L)t^{1/2-2\delta}.
$$
Finally, we observe that
$$
0\leq \left(\frac12-\frac{1}{l}\right)(t\wedge T_{L})\|G\|^{2}_{L_{2}(U,L^{2})} \lesssim t.
$$
Even though positive,  this term behaves linearly in $t$ hence it cannot compensate for the (possibly negative) martingale part for small times. As a consequence, \eqref{eq:224} leads us to the requirement for some $\beta\in (1/2+2\delta,1)$
$$
e(t)\leq e(0) -C(L)\left(\sup_{t\in[0,T]}e^{\frac{1+\beta}{2}}+1\right)t^{1/2-2\delta}.
$$
This in particular implies that  $e(t)\leq e(0)$ and therefore we deduce that for a given $L>1$ and $T>0$ there exists $e(0)$ sufficiently large so that for some $\beta\in (1/2+2\delta,1)$
$$
e(t):=e(0) -C(L)(e(0)^{\frac{1+\beta}{2}}+1)t^{1/2-2\delta}
$$
satisfies all the above requirements as well as \eqref{con:e1} since the $L^\infty$-norm of $GB_L$ is bounded before the stopping time by the Sobolev embedding theorem. The proof is complete.
\end{proof}

The following version of the oscillatory lemma permits us to construct solutions which satisfy the energy equality at a given stopping time $\tau$.

\bl\label{lem:oscs}
Let $\varepsilon\in( 0,T)$, $\alpha_0>0$ and let $\tau\leq T$ be an $(\mathcal{F}_{t})_{t\geq0}$-stopping time.
 Let $[e,w,H]$ be an $(\mathcal{F}_{t})_{t\geq0}$-adapted stochastic process such that
 $$[e,w,H]\in C([\varepsilon,T]\times \mathbb{T}^3; (0,\infty)\times\R^3\times \R^{3\times 3}_{0,\rm{sym}})\quad \mathbf{P}\mbox{-a.s.} $$
 with compact range and
 $$e(w,H)<e-\delta,\quad \textrm{ for all } (t,x)\in [\varepsilon,T]\times \mathbb{T}^3\quad \mathbf{P}\mbox{-a.s.}$$
 for some deterministic constant $\delta>0$. Let
 \begin{equation}\label{eq:226}
 I_{\varepsilon}[w]={\mathbf E}^{\mathbf{P}}\bigg[1_{\{\tau\geq 2\varepsilon\}}\int_{\T} \left(\frac{1}{2}|w|^2-e\right)(\tau)dx\bigg]<-\alpha_0.
 \end{equation}
Then  there exists a sequence $[w_n,V_n]\in C^{\infty}([\varepsilon,T]\times \mathbb{T}^3; \R^3\times \R^{3\times 3}_{0,\rm{sym}}),$ satisfying $(w_n,V_n)(\varepsilon,x)=0$, $x\in\T$, and enjoying the following properties:

 i) the process $[w_n,V_n]$ is $(\mathcal{F}_{t})_{t\geq0}$-adapted  and $[w_n,V_n]\in C([\varepsilon,T]\times \mathbb{T}^3; \R^3\times \R^{3\times 3}_{0,\rm{sym}})$ $\mathbf P$-a.s. with compact range;

 ii) we have $$\partial_t w_n+\div V_n=0, \quad \div w_n=0;$$

 iii) we have for any $\alpha\in(0,1)$
 $$\esssup_{\omega\in\Omega}\|w_n\|_{C^{\alpha}_{T}H^{-1}}\rightarrow0  \textrm{ as } n\rightarrow\infty,$$

 iv) we have  $$e(w+w_n,H+V_n)<e-\delta_n,\quad \textrm{ for all } (t,x)\in [\varepsilon,T]\times \mathbb{T}^3 \ \mathbf{P}\mbox{-a.s.}$$
 for some deterministic constant $\delta_n>0$.

 v) the following holds
 $$\liminf_{n \rightarrow \infty} I_{\varepsilon}[w+w_n] \geqslant I_{\varepsilon}[w]+ c\alpha_0^2,$$
 with $c$ depending only on $\|e\|_{L^\infty}$.
\el

\begin{proof}
The proof follows the lines of \cite[Section 4.5]{DelSze3} and therefore we only discuss the steps that need to be done differently. Let us define the shifted grid of size $h$ as in \cite[Section~4.5]{DelSze3} and use the same notation for $C_{\zeta,i}$, $\Omega^{h}_{\nu}$, $\tau^{h}_{\nu}$, $\nu=1,2$. In order to be consistent with our first oscillatory lemma, Lemma~\ref{lem:osc}, we use a different notation for the velocity and the corresponding oscillations: the lemma is applied to $w$ (instead of $v$ in \cite{DelSze3}) and the oscillations are denoted by $w_{n}$ (instead of $\tilde v_{N}$ in \cite{DelSze3}). We extend $[e,w,H]$ to $[T,T+\varepsilon]$ such that $[e,w,H](t)=[e,w,H](T)$ for $t\in [T,T+\varepsilon]$.

Similarly to  Lemma~\ref{lem:osc}, the compact range of the involved stochastic processes permits to choose the piecewise constant approximations uniformly in $\omega$. On the other hand, in the stochastic setting we are not free to assign an arbitrary value on each of the small cylinders $C_{\zeta,i}$, nor to choose the admissible segment arbitrarily: in order to preserve adaptedness, we shall always assign the value at the minimal time of $C_{\zeta,i}$ and choose the admissible segment  by a measurable selection as in Lemma~5.6 in \cite{BFH20}.

Let us now define
$$
E_{h}(t,x)=E_{h}(t_{i},x_{\zeta})=\frac12|w(t_{i},x_{\zeta})|^{2}-e(t_{i},x_{\zeta})\quad \mbox{for }(t,x)\in C_{\zeta,i},
$$
where $t_{i}$ is the minimal time  and $x_{\zeta}$ is an arbitrary spatial point in the cylinder $C_{\zeta,i}$.
Since $[e,w]$ is of compact range,
$$\lim_{h\rightarrow0}\int_{\Omega^h_{\nu}} E_h(t)\dif x=\frac{1}{2}\left(\frac{3}{4}\right)^3\int_{\T} \left(\frac{1}{2}|w|^2-e\right)(t)\dif x,$$
uniformly in $t\in[\varepsilon, T]$ $\mathbf{P}$-a.s. and this implies for $\nu\in\{1,2\}$
 \begin{align*}
&  \lim_{h\rightarrow0}\mathbf{E}^{\mathbf{P}}\left[1_{\{\tau\geq 2\varepsilon\}\cap \{\tau \in \tau_{\nu}^h\}}\int_{\Omega^h_{\nu}} E_h(\tau)\dif x\right]=\frac{1}{2}\left(\frac{3}{4}\right)^3\mathbf{E}^{\mathbf{P}}\left[1_{\{\tau\geq 2\varepsilon\}\cap \{\tau \in \tau_{\nu}^h\}}\int_{\T} \left(\frac{1}{2}|w|^2-e\right)(\tau)\dif x\right].
\end{align*}
Moreover, we have
$$
\sum_{\nu=1}^{2}\mathbf{E}^{\mathbf{P}}\left[1_{\{\tau\geq 2\varepsilon\}\cap \{\tau \in \tau_{\nu}^h\}}\int_{\T} \left(\frac{1}{2}|w|^2-e\right)(\tau)\dif x\right]\leq I_\varepsilon[w].$$
According to \eqref{eq:226}, there exists  $\nu\in\{1,2\}$ such that
$$
\mathbf{E}^{\mathbf{P}}\left[1_{\{\tau\geq 2\varepsilon\}\cap \{\tau \in \tau_{\nu}^h\}}\int_{\T} \left(\frac{1}{2}|w|^2-e\right)(\tau)\dif x\right]\leq-\alpha_0/2,$$
hence given $[e,w]$ we may choose $h$ sufficiently small so that
\begin{equation}\label{boundE}\mathbf{E}^{\mathbf{P}}\left[1_{\{\tau\geq 2\varepsilon\}\cap \{\tau \in \tau_{\nu}^h\}}\int_{\Omega^h_{\nu}}| E_h|(\tau)\dif x\right]\geq c\alpha_0\end{equation}
for a universal constant $c>0$.
Denoting by $w_{n}$ the oscillatory sequence, we observe  in particular that (60) in \cite{DelSze3} rewrites as
  \begin{equation}\label{eq:60}
    \lim_{n \rightarrow \infty} \int_{\Omega^h_{\nu}} \frac{1}{2} |
    w_{n}|^2 (t) \dif x \geqslant \frac{c}{M} \int_{\Omega^h_{\nu}} |
    E_h |^2 (t) \dif x
  \end{equation}
  uniformly in $t \in \tau_{\nu}^h \cap [2\varepsilon, T]$ $\mathbf P$-a.s. Here $M=\sup_{\Omega\times [\varepsilon,T]\times \mathbb{T}^3}e$.
Now,  we have
  \begin{align}\label{new3}
  \begin{aligned}
   \liminf_{n \rightarrow \infty} I_{\varepsilon} [w+w_{n}] &= \liminf_{n
     \rightarrow \infty} \mathbf{E}^{\mathbf{P}}\left[ 1_{\{\tau\geq 2\varepsilon\}}
     \int_{\mathbb{T}^3} \left(\frac{1}{2} |w+w_{n}|^2-e\right)(\tau)\dif x\right]
  \\ &=
     \liminf_{n \rightarrow \infty} \mathbf{E}^{\mathbf{P}}\left[1_{\{\tau\geq 2\varepsilon\}}\left(\int_{\T} \left( \frac{1}{2} | w|^2 -e\right)(\tau) \dif x +
     \int_{\mathbb{T}^3} \frac{1}{2} | w_{n} |^2(\tau) \dif x\right) \right]
     \\
     &\geq I_\varepsilon[w]+
\mathbf{E}^{\mathbf{P}}\left[1_{\{\tau\geq 2\varepsilon\}\cap \{\tau \in \tau_{1}^h\}}   \liminf_{n \rightarrow \infty}
     \int_{\Omega^h_{1}} \frac{1}{2} | w_{n}|^2(\tau) \dif x\right]\\
     &\qquad+  \mathbf{E}^{\mathbf{P}}\left[1_{\{\tau\geq 2\varepsilon\}\cap \{\tau \in \tau_{2}^h\}}   \liminf_{n \rightarrow \infty}
     \int_{\Omega^h_{2}} \frac{1}{2} | w_{n}|^2(\tau)\dif x \right]
     \\ &\geq I_\varepsilon[w] +\frac{c}{M}\left(\mathbf{E}^{\mathbf{P}}\left[1_{\{\tau\geq 2\varepsilon\}\cap \{\tau \in \tau_{1}^h\}}
     \int_{\Omega^h_{1}} | E_h |(\tau) \dif x \right]\right)^2\\
     &\qquad+ \frac{c}{M}\left(\mathbf{E}^{\mathbf{P}}\left[1_{\{\tau\geq 2\varepsilon\}\cap \{\tau \in \tau_{2}^h\}}
     \int_{\Omega^h_{2}} | E_h |(\tau) \dif x \right]\right)^2 .
     \end{aligned}
     \end{align}
     Here in the last inequality we used \eqref{eq:60}  and H\"{o}lder's inequality.

Thus the result in \emph{v)} follows from \eqref{new3} and \eqref{boundE}.

Comparing to Lemma~\ref{lem:osc}, it only remains to prove \emph{iii)}. By a modification of (5.29) and (5.31) in \cite{BFH20} we obtain
$$\esssup_{\omega\in\Omega}\|w_n\|_{C^1_TH^{-1}}\leq C, \quad \esssup_{\omega\in\Omega}\|w_n\|_{C_TH^{-1}}\rightarrow0,$$
which implies \emph{iii)} by interpolation.
\end{proof}

\begin{remark}
\label{modification}
\emph{i)} Although our proof of Lemma~\ref{lem:oscs} follows the approach of \cite{DelSze3}, we cannot obtain the corresponding result for the linear functional $$\mathbf E^{\mathbf{P}}\bigg[\inf_{t\in[\varepsilon,T]}\int_{\T} \left(\frac{1}{2}|w|^2-e\right)(t)\dif x\bigg]$$ since we cannot change the order of $\mathbf E^{\mathbf{P}}$ and $\inf$.

\emph{ii)} By a modification of \cite[Lemma 5.11]{BFH20}, we can also deduce that \cite[Lemma 5.11]{BFH20} holds  with \emph{iii)} in \cite[Lemma 5.11]{BFH20} strengthened  to the following statement: for any $\alpha\in(0,1)$
$$\esssup_{\omega\in\Omega}\|w_n\|_{C_{T}^{\alpha}H^{-1}}\rightarrow0  \textrm{ as } n\rightarrow\infty.$$
\end{remark}

Finally, we have all in hand to prove the main result of this section.

\bt \label{th:con}
Suppose that \emph{(G3)} holds. Let  $\tau$ be a strictly positive $(\mathcal{F}_{t})_{t\geq0}$-stopping time such that $\tau\leq T$. Let $e\in C([0,T];(0,\infty))$  be the energy constructed in Lemma~\ref{lem:osc2} and let $\bar v\in X_{0}$ be the corresponding velocity constructed in Lemma~\ref{lem:osc1}. There exist infinitely many $(\mathcal{F}_{t})_{t\geq0}$-adapted solutions $v_l$ to the system \eqref{eq:v} on $[0,T]$ with the initial condition $\bar{v}(0)$ satisfying
$\mathbf P$-a.s.
$$
e_l(v_{l})=\frac{1}{2}\|v_l+GB_{L}\|_{L^2}^2\quad \mbox{for a.e. }t\in (0,T),
$$
$$e_l(v_{l})(0)=\frac{1}{2}\|v_l(0)\|_{L^2}^2=\frac{1}{2}\|\bar{v}(0)\|_{L^2}^2,$$
and
\begin{equation}\label{eq:stop}
e_l(v_{l})(\tau)=\frac{1}{2}\|(v_l+GB_{L})(\tau)\|_{L^2}^2.
\end{equation}
In particular, the results hold for $T=L,\tau=T_L$.
 \et

\begin{proof}
\emph{Step 1: $L^{\infty}$-bound for subsolutions.} First,  we show that the subsolutions in $X_{0}$ are bounded essentially uniformly in all the variables $\omega,t,x$ and the bound is only determined by $\|\bar{v}(0)\|_{L^{2}}^{2}$ and the constants $L,T$. In other words, it is independent of the particular subsolution. To this end, we recall that
if $v \in X_0$ then there exists $H$ such that
$$ \partial_t v + \div H = 0, $$
and by \eqref{boundwh} for $t\in(0,T]$
$$ \frac{1}{2} \| v(t)\|_{L^\infty}^2 \leq C (L) + \| \bar{v} (0) \|_{L^{2}}^2 + \int_0^t \langle v +
   G B_L, G \dif B_L \rangle + \left( \frac{1}{2} - \frac{1}{l} \right) (t
   \wedge T_L) \| G \|_{L_{2}(U,L^{2})}^2,
   $$
$$ c \| H(t) \|_{L^{\infty}} \leq \| \bar{v} (0) \|_{L^{2}}^2 + \int_0^t \langle v + GB_L, G \dif B_L \rangle
   + \left( \frac{1}{2} - \frac{1}{l} \right) (t \wedge T_L) \| G
   \|_{L_{2}(U,L^{2})}^2.
   $$
Therefore, for $t\in[0,T]$ and $\beta\in(1/2+2\delta,1)$ and using Lemma~\ref{lem:young}
$$ \frac{1}{2} \| v\|_{L^\infty_{T,x}}^2 \lesssim C (L, T) + \| \bar{v} (0) \|_{L^{2}}^2 + \|
   v \|_{C^\beta_{T}H^{- 1}} C (L,T) ,
   $$
and
\begin{align*}
 \|  v \|_{C^1_{T}H^{- 1}} &\lesssim\| H
   \|_{L^{\infty}_{T}L^{2}} + \| \bar{v} (0) \|_{L^{2}} \lesssim C (L, T) +
   \| \bar{v} (0) \|^2_{L^{2}} +  \|
   v \|_{C^\beta_{T}H^{- 1}} C (L,T)
   \\&\lesssim C (L, T) +
   \| \bar{v} (0) \|^2_{L^{2}} +  \|
   v \|_{C^1_{T}H^{- 1}} ^\beta \|v\|^{1-\beta}_{L^\infty_{T,x}}C (L,T). \end{align*}
Hence substitute the first estimate into the second one and using Young's inequality we obtain
$$
 \|  v \|_{C_{T}^1H^{- 1}} \lesssim C (L, T) + \|
   \bar{v} (0) \|_{L^{2}},$$
and accordingly
$$
 \frac{1}{2} \| v \|_{L^\infty_{T,x}}^2 \lesssim C (L, T) + \| \bar{v} (0) \|_{L^{2}}^2 .
$$
This is the desired uniform $L^{\infty}$-bound on the subsolutions in $X_{0}$. This also implies that $e_l(v)\in C([0,T])$ is of compact range for $v\in X_0$.

\emph{Step 2: Definition of  functionals.}
Let $X$ be the completion of $X_0$ with respect to the metric
$$
D(v,w):=\mathbf E^{\mathbf{P}}\left[\sup_{t\in[0,T]}d(v(t),w(t))\right].
$$
Thus $X$ is a complete metric space.
For $\varepsilon\in (0,T/2)$, we introduce the functionals
$$
I_\varepsilon[v]:=\mathbf E^{\mathbf{P}}\bigg[\int_\varepsilon^T\int_{\T} \left(\frac{1}{2}|v+GB_{L}|^2-e_l(v)\right)\dif x\dif t\bigg],
$$
$$
I_{\tau,\varepsilon}[v]:= \mathbf E^{\mathbf{P}}\bigg[1_{\{\tau\geq 2\varepsilon\}}\int_{\T} \left(\frac{1}{2}|v+GB_{L}|^2-e_l(v)\right)(\tau,x)\dif x\bigg].
$$
Since
$$\mathbf E^{\mathbf{P}}\left[\sup_{t\in[0,T]}\left|\int_0^t \langle v_n-v,G\dif B_L\rangle\right|^{2}\right]\lesssim \E^{\mathbf{P}}\left[\int_0^T\|G^*(v_n-v)\|_U^2\dif s\right]\rightarrow0$$
by compactness of $G^*$ and the $L^\infty$-bound obtained in \emph{Step 1},
we obtain
$$
e_{l}(v_{n})\to e_{l}(v) \mbox{ in }L^{2}(\Omega;C([0,T]))
$$
whenever $v_{n}\to v$ in $X$.  Due to the continuity of the energy $e_{l}$ with respect to the topology on $X$, the functionals $I_\varepsilon, I_{\tau,\varepsilon}$ are lower semi-continuous on the space $X$.

\emph{Step 3: $I_{\varepsilon}$ vanishes at points of continuity.}
As the next step,  we prove that at each point  of continuity $v$ of $I_\varepsilon$ on $X$ it holds $I_\varepsilon[v]=0$.  The idea of proof comes from \cite[Section 3]{CFK}. We proceed by contradiction: fix $\varepsilon\in (0,T)$ and let $v\in X$ be a point of continuity of $I_{\varepsilon}$ and suppose that $I_{\varepsilon}(v)<0.$ By definition of the space $X$,  there is a sequence $v_m\in X_0$ such that $$D(v_m,v)\rightarrow0, \quad I_\varepsilon[v_m]\rightarrow I_\varepsilon[v], \quad I_\varepsilon[v_m]<-\alpha_0$$
for some $\alpha_0>0$.  Now, we can find an $(\mathcal{F}_{t})_{t\geq0}$-adapted function $e^m\in C([0,T])$ $\mathbf P$-a.s. and $\delta_m>0$ such that $(T-\varepsilon)\delta_m<\alpha_0/2$
$$
e^m< e_{l}(v_m),\ t\in(0,T],\quad e^m(0)=e_{l}(v_m)(0),\quad e^m=e_{l}(v_m)-\delta_m, \ t\in[\varepsilon,T],
$$
and
$$v_m\in X_{0,e^m}.$$
Then due to \emph{Step 1}, $e^m\in C([\varepsilon,T];(0,\infty))$ is of compact range. Moreover, by the definition of the stopping time we know that $GB_L\in C([0,T]\times \T;\R^{3})$ is  of compact range.
Define
$$
I_{\varepsilon,e^m}[v]=\mathbf E^{\mathbf{P}}\bigg[\int_\varepsilon^T\int_{\T} \left(\frac{1}{2}|v+GB_{L}|^2-e^m\right)\dif x\dif t\bigg].
$$
For $v_m$ we can choose $\delta_m$ small enough such that $I_{\varepsilon,e^m}[v_m]<-\alpha_0/2$. Then
 by using \cite[Lemma 5.11]{BFH20} and Remark~\ref{modification} for $(e_m,v_m+GB_L,H_m)$ on $[\varepsilon,T]$  we obtain an oscillatory sequence $w_{m,n}\in C^{\infty}_c((\varepsilon,T)\times \mathbb{T}^3;\R^{3})$. In particular, we can extend these functions to the whole $[0,T]$ by $w_{m,n}(t)=0$ for $t\in[0,\varepsilon]$. By Remark~\ref{modification} we have for $\alpha\in(0,1)$
\begin{align}\label{eq:227}
\esssup_{\omega\in\Omega}\|w_{m,n}\|_{C_{T}^{\alpha}H^{-1}}\rightarrow0 \textrm{ as } n\rightarrow\infty,
\end{align}
and
$$
\liminf_{n\rightarrow\infty}I_{\varepsilon,e^m}[v_m+w_{m,n}]\geq I_{\varepsilon,e^m}[v_m]+c\alpha_0^2=I_\varepsilon[v_m]+c\alpha_0^2+(T-\varepsilon)\delta_{m}.
$$

As the next step, we observe that $ v_m+w_{m,n}\in X_{0,e^m},$
which follows from the fact that $v_m+w_{m,n}$ is of compact range as a function in $C([\varepsilon_0,T]\times\T;\mathbb{R}^{3})$ for any $\varepsilon_0\in(0,T)$ since $w_{m,n}$ is of compact support in $(\varepsilon,T)$.
Moreover, we have
\begin{equation*}
\begin{split}
I_{\varepsilon,e^m}[v_m+w_{m,n}]-I_\varepsilon[v_m+w_{m,n}]&=\mathbf E\left[\int_\varepsilon^T\int_{\T}\left(e_{l}(v_m+w_{m,n})-e_{l}(v_m)+\delta_{m}\right) \dif x\dif t\right]\\
&\rightarrow (T-\varepsilon)\delta_{m} \quad\mbox{as }n\to\infty,
\end{split}
\end{equation*}
which implies that
$$
\liminf_{n\rightarrow\infty}I_{\varepsilon}[v_m+w_{m,n}]\geq I_{\varepsilon}[v_m]+c\alpha_0^2.$$

Now, we prove that $v_m+w_{m,n}\in X_0$. The statements for $t\in[0,\varepsilon)$ are obvious since $v_m+w_{m,n}=v_m$. For $t\in[\varepsilon,T]$ we use the convergence in \eqref{eq:227} to obtain \eqref{enconvergence}. In particular, we deduce
\begin{equation*}
e(v_m+w_{m,n}+GB_{L},H_{m}+V_{m,n})<e^{m}=e_{l}(v_m)-\delta_m\leq e_{l}(v_m+w_{m,n})-\delta_m/2,
\end{equation*}
for $n$ large enough. Here, $V_{m,n}$ satisfies $\partial_tw_{m,n}+\div V_{m,n}=0$, cf.  Remark~\ref{modification}. This way, we found a sequence $\tilde{v}_m=v_m+w_{m,n(m)}\in X_0$ such that
 $$
 D(\tilde{v}_m,v)\rightarrow0,\qquad\liminf_{m\rightarrow\infty}I_\varepsilon[\tilde{v}_m]>I_\varepsilon[v],$$ which contradicts the assumption that $v$ is a point of continuity of $I_\varepsilon$.

\emph{Step 4: $I_{\tau,\varepsilon}$ vanishes at points of continuity.}
Now, we prove that at each point $v$ of continuity of $I_{\tau,\varepsilon}$ on $X$ we have $I_{\tau,\varepsilon}[v]=0$. The proof is similar to \emph{Step 3}. For a contradiction, let us suppose that $v$ is a point of continuity and $I_{\tau,\varepsilon}(v)<0.$ Then there is a sequence $v_m\in X_0$  such that
$$D(v_m,v)\rightarrow0, \quad I_{\tau,\varepsilon}[v_m]\rightarrow I_{\tau,\varepsilon}[v],\quad  I_{\tau,\varepsilon}[v_m]<-\alpha_0$$
for some $\alpha_0>0$.
We can find an adapted function $e^m\in C([0,T];(0,\infty))$ $\mathbf P$-a.s. such that for $\delta_m<\alpha_0/2$
$$
e^m< e_{l}(v_m),\ t\in(0,T],\quad e^m(0)=e_{l}(v_m)(0),\quad e^m=e_{l}(v_m)-\delta_m, \ t\in[\varepsilon,T],$$
and
$v_m\in X_{0,e^m}.$
Define
$$
I_{\tau,\varepsilon,e^m}[v]=\mathbf E^{\mathbf{P}}\bigg[1_{\{\tau\geq2 \varepsilon\}}\int_{\T} \left(\frac{1}{2}|v+GB_{L}|^2-e^m\right)(\tau)\dif x\bigg].
$$
We apply Lemma \ref{lem:oscs} to $w=v_m+GB_{L}$ and obtain a sequence $w_{m,n}\in X_{0,e^m}$ such that for $\alpha\in(0,1)$
$$
v_m+w_{m,n}\in X_{0,e^m},\quad w_{m,n}(t)=0, \ t\in[0,\varepsilon),
$$
\begin{align*}
\esssup_{\omega\in\Omega}\|w_{m,n}\|_{C^{\alpha}_{T}H^{-1}}\rightarrow0  \textrm{ as } n\rightarrow\infty,
\end{align*}
and
\begin{equation*}
\liminf_{n\rightarrow\infty}I_{\tau,\varepsilon,e^m}[v_m+w_{m,n}]\geq I_{\tau,\varepsilon,e^m}[v_m]+c\alpha_0^2=I_{\tau,\varepsilon}[v_m]+c\alpha_0^2+\delta_m\mathbf P(\tau\geq2\varepsilon).
\end{equation*}
Moreover,
\begin{equation*}
\begin{split}
I_{\tau,\varepsilon,e^m}[v_m+w_{m,n}]-I_{\tau,\varepsilon}[v_m+w_{m,n}]&=\mathbf E^{\mathbf{P}}\left[1_{\{\tau\geq 2\varepsilon\}}\int_{\T}( e_{l}(v_m+w_{m,n})(\tau)-e_{l}(v_m)(\tau)+\delta_m )\dif x\right]\\
&\rightarrow \delta_m\mathbf P(\tau\geq2\varepsilon) \quad\mbox{as }n\to\infty,
\end{split}
\end{equation*}
which implies that
$$
\liminf_{n\rightarrow\infty}I_{\tau,\varepsilon}[v_m+w_{m,n}]\geq I_{\tau,\varepsilon}[v_m]+c\alpha_0^2.$$
As in \emph{Step 3}, we obtain  $v_m+w_{m,n}\in X_0$
and therefore, we have a sequence $\tilde{v}_m=v_m+w_{m,n(m)}$ such that
$$
D(\tilde{v}_m,v)\rightarrow0,\qquad\liminf_{n\rightarrow\infty}I_{\tau,\varepsilon}[\tilde{v}_m]>I_{\tau,\varepsilon}(v),$$ which contradicts the assumption that $v$ is a point of continuity of $I_{\tau,\varepsilon}$.

\emph{Step 5: Conclusion.}  According to the oscillatory lemma,  Remark~\ref{modification},  the set of subsolutions $X_{0}$ has an infinite cardinality and consequently the same is valid for $X$.
Moreover, due to lower semicontinuity of the functionals $I_\varepsilon$ and $I_{\tau,\varepsilon}$, it follows that  their points of continuity  form residual sets in $X$. Therefore, the set
$$
\mathcal{C}=\cap_{m\in\mathbb{N}}\left\{v\in X;\,I_{1/m}[v] \textrm{ is continuous}\right\}\cap \left\{v\in X;\,I_{\tau,1/m}[v] \textrm{ is continuous}\right\},
$$
is residual as it is  an intersection of a countable family of residual sets. Therefore, $\mathcal{C}$ has  infinite cardinality. In addition, as in \cite[Lemma~6.2]{BFH20} we obtain that if $I_{1/m}[v]=0$ for all $m\in \mathbb{N}$ then  $v$ solves the truncated system \eqref{eq:v} and $\mathbf P$-a.s. $e_l(v)(t)=\frac{1}{2}\|(v+B_{L})(t)\|_{L^2}^2$ for a.e. $t\in(0,T)$. If $I_{\tau,1/m}[v]=0$ for all $m\in \mathbb{N}$, we have $e_l(v)(\tau)=\frac{1}{2}\|v+B_{L}\|_{L^2}^2(\tau)$ $\mathbf P$-a.s.  Therefore, it follows from \emph{Step 3} and \emph{Step 4} that  there are infinitely many solutions of \eqref{eq:v} satisfying the conditions in the statement of the theorem.
\end{proof}

\begin{remark}
The fact that the energy equality holds at a given stopping time $\tau$, namely, that \eqref{eq:stop} holds in Theorem~\ref{th:con} is actually not necessary for the extension of the convex integration solutions in Section~\ref{s:appl} below. Besides, by a simple modification  the result of Theorem \ref{th:con} can be  strengthened further so that \eqref{eq:stop} holds true for an arbitrary sequence of stopping times.
\end{remark}

\subsection{Extension  of solutions}\label{s:ext}
In this section we present a general approach which permits to extend solutions defined up to a stopping time $\tau$ to the whole time interval $[0,\infty)$. The construction adapts the ideas of \cite[Section~5.2]{HZZ19} to the  Euler system. The arguments of this section apply to the setting of a general multiplicative noise coefficient $G$ satisfying (G1) and (G2). For simplicity, we restrict ourselves to the simplified probabilistically weak solutions as this is what is needed for the application to the convex integration solutions presented in Section~\ref{s:appl} below.

First, we  introduce the notion of dissipative probabilistically weak solution up to a stopping time. We note that unlike in our previous definitions of solution, the filtration here is the right continuous version $({\mathcal{B}}_{SPW,t})_{t\geq 0}$ of the canonical filtration $({\mathcal{B}^{0}_{SPW,t}})_{t\geq 0}$. This is required by the convex integration solutions from Section~\ref{s:ci}, where the corresponding stopping times can only be shown to satisfy the stopping time property   with respect to $({\mathcal{B}}_{SPW,t})_{t\geq 0}$, cf. Section~\ref{s:appl} below. To this end, for a $(\mathcal{B}_{SPW,t})_{t\geq0}$-stopping time $\tau$ we define
$$
\Omega_{SPW,\tau}:=\big\{\omega(\cdot\wedge\tau(\omega));\,\omega\in\Omega_{SPW}\big\}.
$$
Then $\Omega_{SPW,\tau}$ is a Borel subset of $\Omega_{SPW}$.

\bd\label{def:martsol}
Let  $(x_{0},y_0,b_{0})\in L^{2}_{\sigma}\times\mathcal{M}^{+}(\T;\R^{3\times 3}_{\rm{sym}})\times U_{1}$ and let $\tau\geq s$ be a $(\mathcal{B}_{SPW,t})_{t\geq 0}$-stopping time. A~probability measure $P\in \mathscr{P}(\Omega_{SPW,\tau})$ is  a simplified dissipative probabilistically weak solution to the Euler system \eqref{el} on $[s,\tau]$ with the initial value $(x_0,y_0,b_{0}) $ at time $s$ provided

\emph{(M1)} $P(x(t)=x_0, y(t)=y_0,  b(t)=b_{0},0\leq t\leq s)=1$, $P(\mathfrak{N}\in L^{\infty}_{\rm{loc}}([s,\tau];\mathcal{M}^{+}(\T;\mathbb{R}^{3\times 3}_{\rm{sym}})))=1$.

\emph{(M2)} Under $P$, for every $l\in U$, $\langle b(\cdot\wedge\tau),l\rangle_{U}$ is a continuous square integrable $(\mathcal{B}_{SPW,t})_{t\geq s}$-martingale starting from $b_{0}$ at time $s$ with quadratic variation process given by $(\cdot\wedge \tau-s)\|l\|_{U}^{2}$  and for every $e_i\in C^\infty(\mathbb{T}^3)\cap L^2_\sigma$ and  $t\geq s$

$$\langle x(t\wedge\tau)-x(s),e_i\rangle-\int^{t\wedge\tau}_s\int_{\mathbb{T}^3}\nabla e_i:\dif\mathfrak{R}(r) \dif r=\int_s^{t\wedge\tau}\langle e_i,G(x(r))\dif b(r)\rangle.$$

\emph{(M3)} P-a.s. define for every $s \leq t\leq \tau$
\begin{align*}
z(t):=\|x_{0}\|_{L^{2}}^{2}+2\int_s^t\langle x(r), G(x(r))\dif b(r)\rangle+\int_s^t\|G(x(r))\|_{L_2(U,L^2)}^2\dif r,
\end{align*}
then
\begin{align*}
P\left(\int_{\T}\dif\tr\mathfrak{R}(t)\leq z(t) \ \textrm{\emph{for a.e.}}\  s\leq t\leq \tau\right)=1.
\end{align*}
\ed

The following result is a generalization of \cite[Proposition~5.2]{HZZ19}. Here and in the sequel, for a $({\mathcal{B}}_{SPW,t})_{t\geq 0}$-stopping time $\tau$ we denote by ${\mathcal{B}}_{SPW,\tau}$ the associated $\sigma$-algebra.

\begin{proposition}\label{prop:1}
Let $\tau$ be a bounded $({\mathcal{B}}_{SPW,t})_{t\geq0}$-stopping time.
Then for every $\omega\in \Omega:=\Omega_{SPW}\cap \Omega_{SPW,0}$ with $\Omega_{SPW,0}=\{(x,y,b);x\in L_{\rm{loc}}^\infty([0,\infty);L^2)\}$ there exists $Q_{\omega}\in\mathscr{P}({\Omega_{SPW}})$ such that
\begin{equation}
\label{qomega}
Q_\omega\big(\omega'\in{\Omega}_{SPW}; \omega'(t)=\omega(t) \textrm{\emph{ for }} 0\leq t\leq \tau(\omega)\big)=1,
\end{equation}
and
\begin{equation}\label{qomega2}
Q_\omega(A)=R_{\tau(\omega),(x,y,b)(\tau(\omega),\omega)}(A)\qquad\text{for all}\  A\in\mathcal{B}_{SPW}^{\tau(\omega)},
\end{equation}
where $R_{\tau(\omega),(x,y,b)(\tau(\omega),\omega)}\in\mathscr{P}({\Omega}_{SPW})$ is a simplified dissipative probabilistically weak solution to the Euler system \eqref{el} starting at time $\tau(\omega)$ from the initial condition $(x,y,b)(\tau(\omega),\omega)$. Furthermore, for every $B\in\mathcal{B}_{SPW}$ the mapping $\omega\mapsto Q_{\omega}(B)$ is $\mathcal{B}_{SPW,\tau}$-measurable.
\end{proposition}

\begin{proof}
According to the stability with respect to the initial time and the initial condition in  Theorem \ref{convergence}, for every $(s,x_0,y_0,b_0)\in [0,\infty)\times L_{\sigma}^2\times \mathcal{M}^+(\mathbb{T}^3,\mathbb{R}^{3\times 3}_{
\rm{sym}})\times U_1$ the set $\mathscr{C}(s,x_{0},y_0,b_0)$ of all associated simplified dissipative probabilistically weak solutions  is compact with respect to the weak convergence of probability measures. Let $\mathrm{Comp}(\mathscr{P}(\Omega_{SPW}))$ denote the space of all compact subsets of $\mathscr{P}(\Omega_{SPW})$ equipped with the Hausdorff metric. Using the stability from Theorem \ref{convergence} again together with \cite[Lemma 12.1.8]{SV79} we obtain that the map
$$
[0,\infty)\times L^{2}_{\sigma}\times \mathcal{M}^+(\mathbb{T}^3,\mathbb{R}^{3\times 3}_{
\rm{sym}})\times  U_1\to\mathrm{Comp}(\mathscr{P}(\Omega_{SPW})),\qquad (s,x_{0},y_0,b_0)\mapsto \mathscr{C}(s,x_{0},y_0,b_0),$$
is Borel measurable. Accordingly, \cite[Theorem 12.1.10]{SV79} gives the existence of a measurable selection. More precisely, there exists a Borel measurable map
$$
[0,\infty)\times L^{2}_{\sigma}\times \mathcal{M}^+(\mathbb{T}^3,\mathbb{R}^{3\times 3}_{
\rm{sym}})\times U_{1}\to\mathscr{P}(\Omega_{SPW}),\qquad(s,x_{0},y_0,b_{0})\mapsto R_{s,x_{0},y_0,b_{0}},
$$
such that $R_{s,x_{0},y_0,b_0}\in \mathscr{C}(s,x_{0},y_0,b_0)$.

As the next step, we recall that the canonical process $\omega$ on $\Omega_{SPW}$ is continuous in
$$
H^{-3}\times \mathcal{M}^+(\mathbb{T}^3,\mathbb{R}^{3\times 3}_{
\rm{sym}})\times U_1,
$$
hence $(x,y,b):[0,\infty)\times\Omega_{SPW}\rightarrow H^{-3}\times \mathcal{M}^+(\mathbb{T}^3,\mathbb{R}^{3\times 3}_{
\rm{sym}})\times U_1$ is progressively measurable with respect to the canonical filtration $({\mathcal{B}}^{0}_{SPW,t})_{t\geq 0}$ and consequently it is also progressively measurable  with respect to the right continuous filtration $({\mathcal{B}}_{SPW,t})_{t\geq 0}$. Define $\tilde{x}:[0,\infty)\times\Omega_{SPW}\rightarrow L^2_\sigma$ by $\tilde{x}(t,\omega)=x(t,\omega)$ if $x(t,\omega)\in L^2_\sigma$ and $\tilde{x}(t,\omega)=0$  if $x(t,\omega)\in H^{-3}\setminus L^2_\sigma$. Furthermore, $L^2_\sigma\subset H^{-3}$ continuously and densely, by Kuratowski's measurable theorem we know $L^2_\sigma\in \mathcal{B}(H^{-3})$ and $\mathcal{B}(L^2_\sigma)=\mathcal{B}(H^{-3})\cap L^2_\sigma$. For $A\in \mathcal{B}(L^2_\sigma)$ with $0\in A$ then $\tilde{x}|_{[0,t]}^{-1}(A)=x|_{[0,t]}^{-1}(A)\cup x|_{[0,t]}^{-1}(H^{-3}\setminus L^2_\sigma)$. For $A\in \mathcal{B}(L^2_\sigma)$ with $0\notin A$ then $\tilde{x}|_{[0,t]}^{-1}(A)=x|_{[0,t]}^{-1}(A)$. This implies that
$$
(\tilde{x},y,b):[0,\infty)\times \Omega_{SPW}\rightarrow L^2_\sigma\times \mathcal{M}^+(\mathbb{T}^3,\mathbb{R}^{3\times 3}_{
\rm{sym}})\times U_{1}
$$
is progressively measurable  with respect to the right continuous filtration $({\mathcal{B}}_{SPW,t})_{t\geq 0}$.

 In addition, $\tau$ is a stopping time with respect to the same filtration $({\mathcal{B}}_{SPW,t})_{t\geq 0}$. Therefore, it follows from \cite[Lemma~1.2.4]{SV79} that both $\tau$ and $(\tilde{x},y,b)(\tau(\cdot),\cdot)$ are ${\mathcal{B}}_{SPW,\tau}$-measurable. Combining this fact with the measurability of the selection $(s,x_{0},y_0,b_0)\mapsto R_{s,x_{0},y_0,b_0}$ constructed above, we deduce that
\begin{equation}\label{eq:R}
\Omega_{SPW}\to \mathscr{P}(\Omega_{SPW}),\qquad\omega\mapsto R_{\tau(\omega),(\tilde{x},y,b)(\tau(\omega),\omega)}
\end{equation}
is ${\mathcal{B}}_{SPW,\tau}$-measurable as a composition of ${\mathcal{B}}_{SPW,\tau}$-measurable mappings. Recall that for every $\omega\in\Omega_{SPW}$ this mapping gives a simplified probabilistically weak solution starting at the deterministic time $\tau(\omega)$ from the deterministic initial condition $(\tilde{x},y,b)(\tau(\omega),\omega)$. In other words,
$$
R_{\tau(\omega),(\tilde{x},y,b)(\tau(\omega),\omega)}\big(\omega'\in\Omega_{SPW}; (\tilde{x},y,b)(\tau(\omega),\omega')=(\tilde{x},y,b)(\tau(\omega),\omega)\big)=1.
$$

Now, we apply \cite[Lemma 6.1.1]{SV79} and deduce that for every $\omega\in \Omega_{SPW}$ there is a unique probability measure
\begin{equation*}
Q_\omega=
  \delta_\omega\otimes_{\tau(\omega)}R_{\tau(\omega),(\tilde{x},y,b)(\tau(\omega),\omega)}\in\mathscr{P}(\Omega_{SPW}),
\end{equation*}
such that for $\omega\in \Omega$ \eqref{qomega} and \eqref{qomega2} hold.

This permits to concatenate, at the deterministic time $\tau(\omega)$, the Dirac mass $\delta_{\omega}$ with the probabilistically weak solution $R_{\tau(\omega),(\tilde{x},y,b)(\tau(\omega),\omega)}$.

In order to show that the mapping $\omega\mapsto Q_{\omega}(B)$  is ${\mathcal{B}}_{SPW,\tau}$-measurable for every $B\in {\mathcal{B}_{SPW}}$, it is enough to consider sets of the form $A=\{(x,y,b)(t_{1})\in \Gamma_{1},\dots, (x,y,b)(t_{n})\in\Gamma_{n}\}$ where $n\in\mathbb{N}$, $0\leq t_{1}<\cdots< t_{n}$, and $\Gamma_{1},\dots,\Gamma_{n}\in\mathcal{B}(H^{-3}\times \mathcal{M}^+(\mathbb{T}^3,\mathbb{R}^{3\times 3}_{
\rm{sym}})\times U_1)$. Then by the definition of $Q_{\omega}$, we have
\begin{align*}
\begin{aligned}
Q_{\omega}(A)&={\bf 1}_{[0,t_{1})}(\tau(\omega))R_{\tau(\omega),(\tilde{x},y,b)(\tau(\omega),\omega)}(A)\\
&\quad+\sum_{k=1}^{n-1}{\bf 1}_{[t_{k},t_{k+1})}(\tau(\omega)){\bf 1}_{\Gamma_{1}}((x,y,b)(t_{1},\omega))\cdots {\bf 1}_{\Gamma_{k}}((x,y,b)(t_{k},\omega))\\
&\qquad\quad\times R_{\tau(\omega),(\tilde{x},y,b)(\tau(\omega),\omega)}\big((x,y,b)(t_{k+1})\in \Gamma_{k+1},\dots, (x,y,b)(t_{n})\in\Gamma_{n}\big)\\
&\quad+{\bf 1}_{[t_{n},\infty)}(\tau(\omega)){\bf 1}_{\Gamma_{1}}((x,y,b)(t_{1},\omega))\cdots {\bf 1}_{\Gamma_{n}}((x,y,b)(t_{n},\omega)).
\end{aligned}
\end{align*}
Here the right hand side is ${\mathcal{B}}_{SPW,\tau}$-measurable as a consequence of the ${\mathcal{B}}_{SPW,\tau}$-measurability of \eqref{eq:R} and $\tau$. The proof is complete.
\end{proof}
\

We proceed with a counterpart of \cite[Proposition~5.3]{HZZ19}.
\begin{proposition}\label{prop:2}
Let $\tau$ be a bounded $({\mathcal{B}}_{SPW,t})_{t\geq0}$-stopping time and let  $P$ be a simplified dissipative probabilistically weak solution to the Euler system \eqref{el} on $[0,\tau]$ starting at the time $0$ from the initial condition $(x_{0},y_0,b_{0})$. Suppose that there exists a Borel  set $\mathcal{N}\subset{\Omega}_{SPW,\tau}$ such that $P(\mathcal{N})=0$ and for every $\omega\in \mathcal{N}^{c}$ it holds
\begin{equation}\label{Q1}
\aligned
&Q_\omega\big(\omega'\in\Omega_{SPW}; \tau(\omega')=
\tau(\omega)\big)=1.
\endaligned\end{equation}
Then the  probability measure $ P\otimes_{\tau}R\in \mathscr{P}(\Omega_{SPW})$ defined by
\begin{equation*}
P\otimes_{\tau}R(\cdot):=\int_{\Omega_{SPW}}Q_{\omega} (\cdot)\,P(\dif\omega)
\end{equation*}
satisfies $P\otimes_{\tau}R= P$ on ${\Omega}_{SPW,\tau}$ and
 is a simplified dissipative probabilistically weak solution to the Euler system \eqref{el} on $[0,\infty)$ with initial condition $(x_{0},y_0,b_{0})$.
\end{proposition}

\begin{proof}
First, by using similar argument as in the proof of Lemma \ref{lem:def} and the fact that $\Omega$ is a Borel subset of ${\Omega}_{SPW}$, we know that $P(\Omega)=1$.
We observe that due to \eqref{Q1} and \eqref{qomega} and $P({\Omega}_{SPW,0})=1$, it holds $P\otimes_{\tau}R(A)=P(A)$ for every Borel set $A\subset{\Omega}_{SPW,\tau}$. It remains to verify that $P\otimes_{\tau}R$ satisfies (M1), (M2) and {(M3)} in Definition \ref{sprobabilistically weak solution} with $s=0$.
The first condition in (M1)  follows easily since by construction
$
P\otimes_{\tau}R(x(0)=x_{0},y(0)=y_{0},b(0)=b_{0})=P(x(0)=x_{0},y(0)=y_{0},b(0)=b_{0})=1.
$
Similarly, by definition we have
\begin{align*}
&P\otimes_{\tau}R\left(\mathfrak{N}\in L_{\rm{loc}}^\infty([0,\infty);\mathcal{M}^+(\mathbb{T}^3,\R^{3\times 3}_{\rm{sym}}))\right)
\\
&=\int_{\Omega_{SPW}}1_{\{\mathfrak{N}\in L_{\rm{loc}}^\infty([0,\tau(\omega)];\mathcal{M}^+(\mathbb{T}^3,\R^{3\times 3}_{\rm{sym}}))\}}Q_\omega\left(\mathfrak{N}\in L_{\rm{loc}}^\infty([\tau(\omega),\infty);\mathcal{M}^+(\mathbb{T}^3,\R^{3\times 3}_{\rm{sym}}))\right)\dif P(\omega)=1.
\end{align*}

Now we shall verify (M2). First, we recall that since $P$ is a probabilistically weak solution on $[0,\tau]$, for $l\in U$ the process
 $\langle b(\cdot\wedge\tau),l\rangle_U$ is a   continuous square integrable $({\mathcal{B}}_{SPW,t})_{t\geq 0}$-martingale under $P$ with  the quadratic variation process
given by
$\|l\|_U^2(\cdot\wedge\tau).$
On the other hand, since for every $\omega\in \Omega$, the probability measure $R_{\tau(\omega),(x,y,b)(\tau(\omega),\omega)}$ is a weak solution  starting at the time $\tau(\omega)$ from the  initial condition $(x,y,b)(\tau(\omega),\omega)$,  the process $\langle b(\cdot)-b(\cdot\wedge \tau(\omega)),l\rangle_U$ is a continuous square integrable $({\mathcal{B}}_{SPW,t})_{t\geq \tau(\omega)}$-martingale under $R_{\tau(\omega),(x,y,b)(\tau(\omega),\omega)}$ with  the quadratic variation process
given by
$(t-\tau(\omega))\|l\|_U^2$, $t\geq\tau(\omega)$. Also $P(\Omega)=1$. Then by the same arguments as in the proof of \cite[Proposition~5.3]{HZZ19} we deduce
that under $P\otimes_\tau R$, $\langle b,l\rangle_{U}$ is a  continuous square integrable $({\mathcal{B}}_{SPW,t})_{t\geq 0}$-martingale  with  the quadratic variation process
given by
$t\|l\|_U^2$, $t\geq0$, which implies that $b$ is cylindrical $({\mathcal{B}}_{SPW,t})_{t\geq 0}$-Wiener process on $U $. Since $b$ is adapted to $({\mathcal{B}}^0_{SPW,t})_{t\geq 0}$, $b$ is cylindrical $({\mathcal{B}}^0_{SPW,t})_{t\geq 0}$-Wiener process on $U $.

Furthermore, we have under $P$ for every $e_i\in C^\infty(\mathbb{T}^3)\cap L^2_\sigma$, and for $t\geq 0$ 
$$
\langle x(t\wedge \tau)-x(0),e_i\rangle-\int^{t\wedge \tau}_0\langle \mathfrak{R},\nabla e_i\rangle \dif r=\int_0^{t\wedge\tau} \langle e_i, G(x(r))  \dif b(r)\rangle.
$$
For $\omega\in \Omega$, under $R_{\tau(\omega),(x,y,b)(\tau(\omega),\omega)}$ it holds for $t\geq \tau(\omega)$
$$
\langle x(t)-x(\tau(\omega)),e_i\rangle-\int^{t}_{\tau(\omega)}\langle\mathfrak{R},\nabla e_i\rangle \dif  r=\int_{\tau(\omega)}^{t} \langle e_i, G(x(r))  \dif b(r)\rangle.
$$
Hence, we have
\begin{align*}
&P\otimes_\tau R\bigg\{\langle x(t)-x(0),e_i\rangle-\int^{t}_0\langle\mathfrak{R},\nabla e_i\rangle \dif r=\int^{t}_0 \langle e_i, G(x(r))  \dif b(r)\rangle, e_i\in C^\infty(\mathbb{T}^3)\cap L^2_\sigma,t\geq0\bigg\}
\\
&=
\int_{\Omega} \dif P(\omega)Q_\omega\bigg\{\langle x(t)-x(t\wedge \tau(\omega)),e_i\rangle-\int^{t}_{t\wedge\tau(\omega)}\langle\mathfrak{R},\nabla e_i\rangle \dif r=\int^{t}_{t\wedge\tau(\omega)}\langle e_i, G(x(r))  \dif b(r)\rangle,\\
&
\langle x(t\wedge \tau(\omega))-x(0),e_i\rangle-\int^{t\wedge\tau(\omega)}_0\langle\mathfrak{R},\nabla e_i\rangle \dif r=\int^{t\wedge\tau(\omega)}_0\langle e_i, G(x(r))  db(r)\rangle, e_i\in C^\infty(\mathbb{T}^3)\cap L^2_\sigma,t\geq0\bigg\}
\\&=\int \dif P(\omega)Q_\omega\bigg\{\langle x(t\wedge \tau(\omega))-x(0),e_i\rangle-\int^{t\wedge \tau(\omega)}_0\langle\mathfrak{R},\nabla e_i\rangle \dif r\\
&\quad\quad\quad=\int^{t\wedge\tau(\omega)}_0 \langle e_i, G(x(r))  \dif b(r)\rangle, e_i\in C^\infty(\mathbb{T}^3)\cap L^2_\sigma,t\geq0\bigg\}.
\end{align*}
By using (\ref{qomega}) we have
\begin{align*}
&\int \dif P(\omega)Q_\omega\bigg\{\langle x(t\wedge \tau(\omega))-x(0),e_i\rangle-\int^{t\wedge \tau(\omega)}_0\langle\mathfrak{R},\nabla e_i\rangle \dif r\\
&\quad\quad\quad=\int^{t\wedge\tau(\omega)}_0 \langle e_i, G(x(r))  \dif b(r)\rangle, e_i\in C^\infty(\mathbb{T}^3)\cap L^2_\sigma,t\geq0\bigg\}
\\&=P\bigg\{\langle x(t\wedge \tau)-x(0),e_i\rangle-\int^{t\wedge \tau}_0\langle\mathfrak{R},\nabla e_i\rangle \dif r\\
&\quad\quad\quad=\int^{t\wedge\tau}_0 \langle e_i, G(x(r))  \dif b(r)\rangle, e_i\in C^\infty(\mathbb{T}^3)\cap L^2_\sigma,t\geq0\bigg\}=1.
\end{align*}
Then  (M2) follows.

Finally, we prove (M3). Define  $P$-a.s.
$$z(t):=\|x(0)\|_{L^2}^2+2M^E_{t,0}+\int_0^{t} \|G(x(r))\|_{L_2(U,L^2_\sigma)}^2\dif r$$ with $M^E_{t,0}=\int_0^t\langle x, G(x(r))\dif b(r)\rangle.$
It holds  that $\|x(t)\|_{L^2}^2\leq z(t)$ for a.e. $t\in[0,\tau]$  $P$-a.s.
which by lower semi-continuity implies
\begin{equation}\label{tau}
\|x( \tau)\|_{L^2}^2\leq z(\tau)=\|x(0)\|_{L^2}^2+2M^E_{\tau,0}+\int_0^{\tau}\|G(x(r))\|_{L_2(U,L^2_\sigma)}^2\dif r.
\end{equation}
Thus we have
\begin{align*}
&\int \dif P(\omega)Q_\omega\bigg\{\int_{\mathbb{T}^3}\dif\tr\mathfrak{R}(t)  \leq z(t) \  \mbox{a.e.}\  t\geq 0 \bigg\}
\\
&=\int \dif P(\omega)Q_\omega\bigg\{\int_{\mathbb{T}^3}\dif
\tr\mathfrak{R}(t)  \leq z(t) \ \mbox{a.e.}\  t\in[0,\tau(\omega)],
 \,\int_{\mathbb{T}^3}\dif\tr\mathfrak{R}(t)  \leq z(t) \ \mbox{a.e.}\  t\geq\tau(\omega)\bigg\}
 \\
 &=\int \dif P(\omega) 1_{\{\int_{\mathbb{T}^3}\dif\tr\mathfrak{R}(t)  \leq z(t) \ \mbox{a.e.}\  t\in[0,\tau]\}}
 Q_\omega\bigg\{\int_{\mathbb{T}^3}\dif\tr\mathfrak{R}(t)\leq  z(t) \ \mbox{a.e.}\  t\geq\tau(\omega)\bigg\}\\
 &=\int \dif P(\omega) 1_{\{\int_{\mathbb{T}^3}\dif\tr\mathfrak{R}(t)  \leq z(t) \ \mbox{a.e.}\  t\in[0,\tau]\}}
 Q_\omega\bigg\{\int_{\mathbb{T}^3}\dif\tr\mathfrak{R}(t)\leq z(t)-z(\tau(\omega)) + z(\tau(\omega))   \ \mbox{a.e.}\  t\geq\tau(\omega)\bigg\}\\
&\geq \int \dif P(\omega) 1_{\{\int_{\mathbb{T}^3}\dif\tr\mathfrak{R}(t)  \leq z(t) \ \mbox{a.e.}\  t\in[0,\tau]\}}
 \\&\qquad\qquad Q_\omega\bigg\{\int_{\mathbb{T}^3}\dif\tr\mathfrak{R}(t)\leq z(t)-z(\tau(\omega)) + \|x(\tau(\omega),\omega)\|_{L^{2}}^{2}   \ \mbox{a.e.}\  t\geq\tau(\omega)\bigg\}=1,
\end{align*}
where in the last second step we used \eqref{qomega} and \eqref{tau} to deduce for $P$-a.s. $\omega$ $$Q_\omega(\omega': \|x(\tau(\omega),\omega)\|_{L^2}^2\leq z(\tau(\omega),\omega)=z(\tau(\omega),\omega'))=1$$  and in the last step we used \eqref{Q1} and \eqref{tau} together with the fact that the corresponding process $z$ for a solution starting from the initial time $\tau(\omega)$ from the initial condition $(x,y,b)(\tau(\omega),\omega)$ as in \eqref{qomega2} is
$$
z(t)-z(\tau(\omega))+\|x(\tau(\omega),\omega)\|_{L^{2}}^{2}.
$$
\end{proof}

\subsection{Application to the convex integration solutions}\label{s:appl}

As in Section~\ref{s:ci}, we restrict ourselves to the additive noise case satisfying the assumption (G3). Our goal is then to  apply the construction from Section~\ref{s:ext} to the convex integration solutions obtained in Theorem~\ref{th:con}, in order to conclude the non-uniqueness in law. Moreover, we assume (G4) on $G$.

In this subsection we fix the parameters $l\in[2,\infty]$, $L>1$. As the first step, it is necessary to define a process on the path space $\Omega_{SPW}$, that is, a function of the variables $(x,y,b)$ defined without the usage of any probability measure,  such that  under the law of the convex integration solution $u=u_{l}=v_{l}+GB_{L}$  from Section~\ref{s:ci} it becomes the stochastic integral $M^{E}$ from the energy equality in \eqref{eq:energy1}.
In particular, since it holds
$$
M^{E}_{t,0}=\int_{0}^{t}\langle u,G\dif B_{L}\rangle =\int_{0}^{t}\langle v,G\dif B_{L}\rangle +\int_{0}^{t}\langle GB_{L},G\dif B_{L}\rangle,
$$
where $v$ solves the transformed system \eqref{eq:v}, we see that the only part requiring probability is the last term, namely the iterated integral of $GB_{L}$. Indeed, the first term on the right hand side can be defined without any probability as a Young integral due to the fact that $v$ necessarily has better time regularity than $u$.
This motivates the following definition: for every $\omega=(x,y,b)\in \Omega_{SPW}$ we let
\begin{equation}\label{eq:M1}
\begin{aligned}
\bar{M}_{t,0}^{(x,y,b)}&:=\frac{1}{2}\| x(t)\|_{L^2}^2-\frac{1}{2}\|x({0})\|_{L^2}^2-\left(\frac{1}{2}-\frac{1}{l}\right)t\|G\|_{L_{2}(U,L^{2})}^2\\
&\qquad-\int_0^t \langle x({0})-\mathbb{P} \div (y (s)-y(0)) ,G\dif b(s)\rangle,
\end{aligned}
\end{equation}
where in view of \eqref{eq:v} and (G4) and discussions in Section 2.3.2 the last integral is well-defined   Young integral, cf. Lemma~\ref{lem:young}. In other words, we will show below that under the law of the corresponding convex integration solution the process $\bar{M}$ is a.s. the  iterated integral of $GB_{L}$.

 For $n\in\mathbb{N},L>0$ and for  $\delta\in(0,1/12)$ to be determined below   we define

 \begin{equation*}
 \aligned\tau_L^{n,1}&=\inf\left\{t\geq 0, \|Gb(t)\|_{H^{\frac{3+\sigma}{2}}}>L-\frac1n\right\}\wedge \inf\left\{t\geq 0,\|Gb\|_{C_t^{\frac{1}{2}-2\delta}H^{1}}>L-\frac1n\right\}\wedge L,
 \endaligned
 \end{equation*}
 \begin{equation*}
 \aligned\tau_L^{n,2}&= \inf\left\{t\geq 0,\|\bar{M}\|_{W_t^{\beta,p}}>L-\frac1n\right\}\wedge L,
\qquad
\tau_L^n=\tau_L^{n,1}\wedge\tau_L^{n,2}.
 \endaligned
 \end{equation*}
We observe that the sequence $(\tau^{n}_{L})_{n\in\mathbb{N}}$ is nondecreasing and define
\begin{equation}\label{eq:tauL}
\tau_L:=\lim_{n\rightarrow\infty}\tau_L^n.
\end{equation}
Note that without an additional regularity of the trajectory $\omega$, it holds true that $\tau^{n}_{L}(\omega)=0$.
By  \cite[Lemma~3.5]{HZZ19} we obtain that $\tau_L^{n,1}$ is $({\mathcal{B}}_{SPW,t})_{t\geq0}$-stopping time.
Here $\omega\mapsto \|\bar M\|_{W_t^{\beta,p}}$ is $(\mathcal{B}^0_{SPW,t})_{t\geq 0}$-adapted since the fact that $(x,y,b)$ is progressively measurable with respect to $({\mathcal{B}}^0_{SPW,t})_{t\geq0}$ implies that $\bar M$ is progressively measurable with respect to $({\mathcal{B}}^0_{SPW,t})_{t\geq0}$. Since $t\mapsto \|\bar{M}\|_{W_t^{\beta,p}}$ is increasing, it holds
for
$$
\tau=\inf\left\{t\geq 0,\|\bar{M}\|_{W_t^{\beta,p}}>L\right\}
$$
that
\begin{align*}
\{\tau\geq t\}=\cap_{m=1}^\infty \left\{\|\bar{M}\|_{W_{t-\frac1m}^{\beta,p}}\leq L\right\}\in \mathcal{B}_{SPW,t}^0.
\end{align*}
Indeed, it is obvious that the right hand side is contained in the left hand side. In addition, the set $\{\tau>t\}$ is also contained in the right hand side. For $\{\tau=t\}$ we also have
for every $m\geq 1$
$$
\|\bar{M}\|_{W_{t-\frac1m}^{\beta,p}}\leq L.$$
Therefore,  $\tau_L^{n,2}$ is $(\mathcal{B}_{SPW,t})_{t\geq0}$-stopping time.
Thus also  $\tau_L$ is a $(\mathcal{B}_{SPW,t})_{t\geq 0}$-stopping time as an increasing limit of stopping times with respect to a right continuous filtration.

Now, we fix   a $GG^{*}$-Wiener process $B$ defined on a probability space $(\Omega, \mathcal{F},\mathbf{P})$ and we denote by  $(\mathcal{F}_{t})_{t\geq0}$ its normal filtration, i.e. the canonical filtration of $B$ augmented by all the $\mathbf{P}$-negligible sets. We recall that this filtration is right continuous. On this stochastic basis, for fixed  parameters $l\in[2,\infty]$, $L>1$, apply Theorem~\ref{th:con} and denote by $u=u_{l}=GB_{L}+v_{l}$ the corresponding solution to the Euler system \eqref{el} on $[0,T_{L}]$, where the stopping time $T_{L}$ was defined in \eqref{stopping time}. We recall that $u$ is adapted with respect to $(\mathcal{F}_{t})_{t\geq0}$ which is  essential to show the martingale property in (M2) in Proposition~\ref{prop:ext} below. We denote by $P$ the law of $(u,\int_0^\cdot u\otimes u \,\dif s, B)$ and prove the following result.

\begin{proposition}\label{prop:ext}
The probability measure $P$ is a simplified dissipative probabilistically  weak solution to the Euler system \eqref{el} on $[0,\tau_{L}]$ in the sense of Definition \ref{def:martsol}, where $\tau_{L}$ was defined in \eqref{eq:tauL}.
\end{proposition}

\begin{proof}
Recall that the stopping time $T_{L}$ was defined in \eqref{stopping time} in terms of the process $B$. Theorem~\ref{th:con} yields the existence of a solution $u=v_{l}+GB_{L}$ to the Euler system \eqref{el} on $[0,T_{L}]$  such that
$$
\left(u,\int_0^\cdot u\otimes u\,\dif s, B\right)(\cdot\wedge T_{L})\in \Omega_{SPW}\quad \mathbf{P}\mbox{-a.s.}
$$
 In the following,  we  write $(u,B)$  for notational simplicity to denote $(u,\int_0^\cdot u\otimes u\,\dif s, B)$.

 We will now prove that
 \begin{equation}\label{eq}
 \tau_L\left(u, B\right)=T_L \quad \mathbf{P}\text{-a.s}.
 \end{equation}
To this end, we observe that due to the definition of $\bar{M}$ in \eqref{eq:M1} together with the fact that $u$ solves the Euler system \eqref{el} on $[0,T_{L}]$ and $\frac{1}{2}\|u\|^2_{L^2}=e_l(u-GB_L)$ for a.e. $t\in [0,T_L]$, we have $\mathbf{P}$-a.s.
  \begin{equation}\label{eq1}
  b^{(u, B)}(t)=B(t) \quad\textrm{ for } t\in[0,T_L],
  \end{equation}
and
   \begin{equation}
\label{eq2}
  \bar{M}^{(u,B)}(t)=\int_0^t\langle G B(s),G\dif B(s)\rangle \quad\textrm{ for a.e. } t\in[0,T_L].
  \end{equation}
Since $GB\in CH^{\frac{3+\sigma}{2}}\cap C^{\frac{1}{2}-\delta}_{\mathrm{loc}}H^{1}$ $\mathbf{P}$-a.s. and $\int_0^\cdot\langle G B(s),G\dif B(s)\rangle \in C^{\frac{1}{2}-\delta}_{\mathrm{loc}}$ $\mathbf{P}$-a.s., the trajectories of the processes
$$
t\mapsto \|GB(t)\|_{H^{\frac{3+\sigma}{2}}}\quad\text{and}\quad t\mapsto \|GB\|_{C^{\frac12-2\delta}_{t}H^{1}}\quad\text{and}\quad t\mapsto \left\|\int_0^\cdot\langle G B(s),G\dif B(s)\rangle\right\|_{W_{t}^{\beta,p}} $$
are  $\mathbf{P}$-a.s. continuous.
 It follows from the definition of $T_L$ that one of the following four statements holds $\mathbf{P}$-a.s.:
$$
\text{either }\  T_L=L\ \text{ or }\  \|GB(T_L)\|_{H^{\frac{3+\sigma}{2}}}\geq L \ \text{ or }\  \|GB\|_{C_{T_L}^{\frac{1}{2}-2\delta}H^{1}}\geq L,$$
$$
\text{ or }\  \left\|\int_0^{\cdot}\langle G B(s),G\dif B(s)\rangle\right\|_{W_{T_L}^{\beta,p}}\geq L.
$$
Therefore, as a consequence of \eqref{eq1} and \eqref{eq2}, we deduce that $\tau_L(u,B)\leq T_L$ $\mathbf{P}$-a.s. Suppose now  that $\tau_L(u, B)<T_L$ holds true on a set of positive probability $\mathbf{P}$. Then it holds on this set that
$$
\|GB(\tau_L)\|_{H^{\frac{3+\sigma}{2}}}=\|Gb(\tau_L)\|_{H^{\frac{3+\sigma}{2}}}\geq L\ \text{ or }\  \|Gb\|_{C_{\tau_L}^{\frac{1}{2}-2\delta}H^{1}}=\|GB\|_{C_{\tau_L}^{\frac{1}{2}-2\delta}H^{1}}\geq L,
$$
$$
 \text{ or }\ \|\bar{M}\|_{{W_{\tau_L}^{\beta,p}}}= \left\|\int_0^{\cdot}\langle G B(s),G\dif B(s)\right\|_{{W_{\tau_L}^{\beta,p}}}\geq L.
$$
which however contradicts the definition of $T_L$. Hence we have proved  (\ref{eq}).

Recall that $\tau_{L}$ is a $(\mathcal{B}_{SPW,t})_{t\geq 0}$-stopping time. We intend to show that  $P$ is a simplified dissipative probabilistically  weak solution to the Euler system \eqref{el} on $[0,\tau_{L}]$ in the sense of Definition \ref{def:martsol}. First, we observe that it can be seen from the construction in Theorem \ref{th:con} that the initial value $u(0)=\bar{v}(0)+GB(0)=\bar{v}(0)$ is indeed deterministic. Moreover, the convex integration solutions are analytically weak, that is, $\mathfrak{N}\equiv 0$ $P$-a.s.
Hence the condition (M1) follows.
Since $(u,B)$ satisfies the  Euler equations before $T_L$, the equation in (M2) follows from \eqref{eq1} and \eqref{eq}. By using adaptedness of $u$ and a similar argument as in \cite[Proposition 3.7]{HZZ19} we obtain $\langle b(\cdot\wedge\tau_{L}),l\rangle_{U}$ is a continuous square integrable $(\mathcal{B}_{SPW,t})_{t\geq s}$-martingale starting from $b_{0}$ at time $s$ with quadratic variation process given by $(\cdot\wedge \tau_{L}-s)\|l\|_{U}^{2}$.
In order to verify (M3), we recall  $\int_{\mathbb{T}^3}\dif\tr\mathfrak{R}(t)=\|x(t)\|_{L^2}^2$ for $t\in[0,\tau_L]$ $P$-a.s. and that by Theorem \ref{th:con} we have for
$$
Z(t):=\|u(0)\|_{L^2}^2+2\int_0^t\langle u, G\dif B\rangle+t\|G\|_{L_2(U,L^{2})}^2
$$
that
\begin{align*}
\mathbf{P}\left(\|u(t)\|_{L^2}^2\leq Z(t)\ \mbox{for a.e.}\  t\in[0,T_L]\right)=1.
\end{align*}
 Thus (M3) follows from \eqref{eq1} and \eqref{eq}.
\end{proof}

At this point, we are already able to deduce that simplified dissipative probabilistically weak solutions on $[0,\tau_{L}]$ in the sense of Definition \ref{def:martsol} are not unique. However, we aim at a stronger result, namely that globally defined simplified dissipative probabilistically  weak solutions on $[0,\infty)$ in the sense of Definition \ref{sprobabilistically weak solution} are not unique. Moreover, we are able to prove non-uniqueness on an arbitrary time interval $[0,T]$, $T>0$.

\begin{proposition}\label{prp:ext2}
The probability measure $P\otimes_{\tau_{L}}R$ is a simplified dissipative probabilistically  weak solution to the Euler system \eqref{el} on $[0,\infty)$ in the sense of Definition \ref{sprobabilistically weak solution}.
\end{proposition}

\begin{proof}
In light of Proposition \ref{prop:1} and Proposition \ref{prop:2}, it only remains to establish \eqref{Q1}.
Due to \eqref{eq} \eqref{eq1} and \eqref{eq2}, we know that
\begin{align*}
\begin{aligned}
&P\left(\omega:Gb^\omega(\cdot\wedge \tau_L(\omega))\in CH^{\frac{3+\sigma}{2}}\cap C^{\frac{1}{2}-\delta}_{\mathrm{loc}}H^{1}\right)\\
&\qquad=\mathbf{P}\left(Gb^{(u,B)}(\cdot\wedge \tau_L(u,B))\in CH^{\frac{3+\sigma}{2}}\cap C_{\mathrm{loc}}^{\frac{1}{2}-\delta}H^{1}\right)\\
&\qquad=\mathbf{P}\left(GB(\cdot\wedge T_L)\in CH^{\frac{3+\sigma}{2}}\cap C^{\frac{1}{2}-\delta}_{\mathrm{loc}}H^{1}\right)=1,
\end{aligned}
\end{align*}
\begin{align*}
\begin{aligned}
P\left(\omega:\bar{M}^\omega(\cdot\wedge \tau_L(\omega))\in W^{\beta,p}_{\tau_L(\omega)}\right)&=\mathbf{P}\left(\bar{M}^{(u,B)}(\cdot)\in W^{\beta,p}_{\tau_L(u,B)}\right)\\
&=\mathbf{P}\left(\int_0^{\cdot}\langle GB,G\dif B\rangle\in  W^{\beta,p}_{T_L}\right)=1.
\end{aligned}
\end{align*}
In other words,   there exists a $P$-measurable set $\mathcal{N}\subset \Omega_{SPW}$ such that $P(\mathcal{N})=0$ and for $\omega\in \mathcal{N}^c$
\begin{equation}\label{continuity}
Gb^\omega_{\cdot\wedge \tau_L(\omega)}\in CH^{\frac{3+\sigma}{2}}\cap C^{\frac{1}{2}-\delta}_{\mathrm{loc}}H^{1},\quad \bar{M}^\omega_{\cdot\wedge \tau_L(\omega)}\in  W^{\beta,p}_{\tau_L(\omega)}.
\end{equation}

Using  \eqref{qomega} and \eqref{qomega2} it holds that for all $\omega\in\Omega,$ with $\Omega$ defined in the statement of Proposition~\ref{prop:1},
\begin{equation*}\aligned
& Q_\omega\left(\omega'\in\Omega_{SPW} ; Gb^{\omega'}_{\cdot}\in CH^{\frac{3+\sigma}{2}}\cap C^{1/2-\delta}_{\mathrm{loc}}H^{1}, \right)
\\&=Q_\omega\left(\omega'\in\Omega_{SPW}; Gb^{\omega'}_{\cdot\wedge \tau_L(\omega)}\in CH^{\frac{3+\sigma}{2}}\cap C^{1/2-\delta}_{\mathrm{loc}}H^{1}, Gb^{\omega'}_{\cdot}-Gb_{\cdot\wedge\tau_L(\omega)}^{\omega'}\in CH^{\frac{3+\sigma}{2}}\cap C^{1/2-\delta}_{\mathrm{loc}}H^{1}\right)
\\&=\delta_\omega\left(\omega'\in\Omega_{SPW}; Gb^{\omega'}_{\cdot\wedge \tau_L(\omega)}\in CH^{\frac{3+\sigma}{2}}\cap C_{\mathrm{loc}}^{1/2-\delta}H^{1}\right) \\
&\qquad\times  R_{\tau_L(\omega),(x,y,b)(\tau_L(\omega),\omega)}\left(\omega'\in\Omega_{SPW}; Gb^{\omega'}_{\cdot}-Gb_{\cdot\wedge\tau_L(\omega)}^{\omega'}\in CH^{\frac{3+\sigma}{2}}\cap C^{1/2-\delta}_{\mathrm{loc}}H^{1}\right).
\endaligned
\end{equation*}
Here the first factor on the right hand side equals to $1$ for all  $\omega\in \mathcal{N}^c$ due to \eqref{continuity}.
Since $R_{\tau_L(\omega),(x,y,z)(\tau_L(\omega),\omega)}$ is a simplified dissipative probabilistically weak solution starting at the deterministic time $\tau_{L}(\omega)$ from the deterministic initial condition $(x,y,b)(\tau_{L}(\omega),\omega)$, the process $\omega'\mapsto b_{\cdot}^{\omega'}-b^{\omega'}_{\cdot\wedge \tau_L(\omega)}$ is a cylindrical $(\mathcal{B}_{SPW,t})_{t\geq0}$-Wiener process on $U$ starting from $\tau_L(\omega)$   under the measure $R_{\tau_L(\omega),(x,y,b)(\tau_L(\omega),\omega)}$.
Thus we deduce that also the second factor equals to $1$.

To summarize, we have proved that  for all  $\omega\in \mathcal{N}^c\cap \Omega$, with $\Omega$ defined in the statement of Proposition~\ref{prop:1},
$$
Q_\omega\left(\omega'\in\Omega_{SPW}; Gb^{\omega'}_{\cdot}\in CH^{\frac{3+\sigma}{2}}\cap C^{1/2-\delta}_{\mathrm{loc}}H^{1}\right)=1,
$$
and similarly we obtain
$$
Q_\omega\left(\omega'\in\Omega_{SPW}; \bar{M}^{\omega'}_{\cdot\wedge \tau_{L}(\omega)}\in W^{\beta,p}_{\tau_L(\omega)}\right)=1.
$$
As a consequence, for all
$\omega\in \mathcal{N}^c\cap\Omega$ there exists a measurable set $N_\omega$ such that $Q_\omega(N_\omega)=0$ and for all $\omega'\in N_\omega^c$ the trajectory
$t\mapsto Gb^{\omega'}(t)$ belongs to $CH^{\frac{3+\sigma}{2}}\cap C^{1/2-\delta}_{\mathrm{loc}}H^{1}$ and $t\mapsto \bar{M}^{\omega'}(t)$ belongs to $ W^{\beta,p}_{\tau_L(\omega)}$. Therefore, by \eqref{eq:tauL} and continuity we obtain that
 $$\{\omega'\in N_\omega^c:\tau_L(\omega')=\tau_L(\omega)\}=\{\omega'\in N_\omega^c:\bar{\tau}_L(\omega')=\tau_L(\omega)\},$$
  where
 \begin{equation*}
 \bar{\tau}^1_L(\omega')=\inf\left\{t\geq 0, \|Gb^{\omega'}(t)\|_{H^{\frac{3+\sigma}{2}}}\geq L\right\}\wedge\inf\left\{t\geq 0,\|Gb^{\omega'}\|_{C_t^{1/2-2\delta}H^{1}}\geq L\right\}\wedge L,
 \end{equation*}
 \begin{equation*}
 \bar{\tau}^2_L(\omega')=\inf\left\{t\geq 0, \|\bar{M}^{\omega'}\|_{W^{\beta,p}_{t\wedge\tau_L(\omega)}}\geq L\right\}\wedge L,\qquad \bar{\tau}_L=\bar{\tau}^1_L\wedge \bar{\tau}^2_L.
 \end{equation*}
Also for $t< L$
\begin{equation}
\label{mea}
\aligned
&\left\{\omega'\in N_\omega^c,\bar{\tau}_L(\omega')\leq t\right\}=\left\{\omega'\in N_\omega^c, \sup_{s\in\mathbb{Q},s\leq t}
\|Gb^{\omega'}(s)\|_{H^{\frac{3+\sigma}{2}}}\geq L\right\}
\\&\qquad\cup\left\{\omega'\in N_\omega^c,\sup_{s_1\neq s_2\in \mathbb{Q}\cap [0,t]}\frac{\|Gb^{\omega'}(s_1)-Gb^{\omega'}(s_2)\|_{H^{1}}}{|s_1-s_2|^{\frac{1}{2}-2\delta}}\geq L\right\}
\\&\qquad\cup\left\{\omega'\in N_\omega^c,
\|\bar{M}^{\omega'}\|_{W^{\beta,p}_{t\wedge\tau_L(\omega)}}\geq L\right\}
=:N^c_\omega\cap A_t.
\endaligned
\end{equation}
Finally, we deduce that for all $\omega\in\mathcal{N}^c\cap \Omega$
\begin{equation*}\label{Q}
\aligned
&Q_\omega\big(\omega'\in\Omega_{SPW}; \tau_L(\omega')=\tau_L(\omega)\big)=Q_\omega\big(\omega'\in N_\omega^c; \tau_L(\omega')=\tau_L(\omega)\big)
\\&\quad=Q_\omega\big(\omega'\in N_\omega^c; \omega'(s)=\omega(s), 0\leq s\leq \tau_L(\omega), \bar{\tau}_L(\omega')=\tau_L(\omega)\big)=1,
\endaligned
\end{equation*}
where we used \eqref{qomega} and the fact that (\ref{mea}) implies
$$\{\omega'\in N_\omega^c; \bar{\tau}_L(\omega')=\tau_L(\omega)\}=N_\omega^c\cap (A_{\tau_L(\omega)}\backslash (\cup_{n=1}^\infty A_{\tau_L(\omega)-\frac1n}))\in N_\omega^c\cap \mathcal{B}_{SPW,\tau_L(\omega)}^0,$$
 and  $Q_\omega(A_{\tau_L(\omega)}\backslash (\cup_{n=1}^\infty A_{\tau_L(\omega)-\frac1n}))=1$.
 This verifies the condition \eqref{Q1} in Proposition \ref{prop:2} and as a consequence $P\otimes_{\tau_{L}}R$ is a simplified dissipative probabilistically weak  solution to the Euler system \eqref{el} on $[0,\infty)$ in the sense of Definition \ref{sprobabilistically weak solution}.
\end{proof}

Finally, we have all in hand to  prove  the main result of this section.

\begin{proof}[Proof of Theorem  \ref{Main results2}]
By Proposition~\ref{prp:ext2} we obtain the existence of infinitely many simplified dissipative probabilistically weak solutions $P=P_{l}$ on $[0,\infty)$ starting from the initial value $(\bar{v}(0),0,0)$ and parametrized by $l\in[2,\infty]$.
In addition,
 it holds $P_{l}$-a.s.
$$\|x(t\wedge \tau_L)\|_{L^2}^2=\|\bar{v}(0)\|_{L^2}^2+2\int_0^{t\wedge \tau_L}\langle x,G\dif B\rangle+\left(\frac12-\frac{1}{l}\right)(t\wedge\tau_L)\|G\|_{L_2(U,L^2)}^2\quad \mbox{for a.e.}\  t\in[0,T].$$
Thus, we obtain that the law of $(x,y,b)$ on $[0,T]$ is not unique. In addition, by (M2) we know that under each $P_{l}$ the law of $b$ is determined by $(x,y)$. Therefore, we deduce that the law $(x,y)$ is not unique which completes the proof of non-uniqueness in law for simplified dissipative martingale solutions.
\end{proof}

\section{Existence and non-uniqueness of strong Markov solutions}\label{sec:mar}

We aim at showing existence of a strong Markov dissipative solution to \eqref{el}. The approach relies on the abstract Markov selection introduced by Krylov~\cite{K73} and presented by Stroock, Varadhan~\cite{SV79}. Applications to SPDEs can be found in \cite{FR08, GRZ09, BFH18markov} where existence of almost sure Markov solutions was proved for several SPDEs including stochastic Navier--Stokes equation in the incompressible as well as compressible setting. In these works, the existence of solutions satisfying the usual Markov property, was left open.
Motivated by the recent construction of solution semiflows to deterministic isentropic and complete Euler system in \cite{BreFeiHof19B,BreFeiHof19}, we put forward a construction which not only applies to the stochastic Euler equations \eqref{el} but it even permits to obtain strong Markov solutions.

The principal idea is to include an additional datum into the selection procedure. In \cite{BreFeiHof19B,BreFeiHof19}, this was done through the energy which in our notation corresponds to
$
E(t):= \int_{\T}\dif\tr\mathfrak{R}(t).
$
In the deterministic setting, this function is non-increasing hence of finite variation. Thus, it admits left- and right-limits at all times and is continuous except for an at most countable set of times.
Nevertheless, as in the stochastic setting $E$ is not decreasing due to the martingale part and even if it admits left- and right-limits  these exceptional times become random, it does not seem  possible to get existence of Markov selections by including solely the variable $E$. As suggested in Section~\ref{s:p1}, we rather include the variable $z$ which plays the same role but in addition has time continuous trajectories.

As uniqueness of simplified dissipative martingale solutions was disproved in Section~\ref{s:law}, it is necessary to find additional selection criteria in order to select  physically relevant solutions. A well-accepted criterion in the case of the deterministic Euler equations is  the maximization of energy dissipation or, equivalently, minimization of the total energy of the system proposed by Dafermos \cite{Daf4}. In the same spirit, we aim at selecting only solutions which minimize the average total energy. To be more precise, let  $P$ and $Q$ be two dissipative martingale solutions starting from the same initial data $(x_0,y_0,z_0)$. We introduce the relation
\begin{align*}
P\prec Q\Leftrightarrow \E^{P}\left[\int_{\T}\dif\tr\mathfrak{R}(t)\right]\leq \E^{Q}\left[\int_{\T}\dif\tr\mathfrak{R}(t)\right]\textrm{ for a.e. } t\in (0,\infty),
\end{align*}
which leads to the following admissibility condition.

\bd[Admissible dissipative martingale solution]
We say that a dissipative martingale solution $P$ starting from $(x_0,y_0,z_0)$ is admissible if it is minimal with respect to the relation $\prec$, i.e., if there is another dissipative martingale solution $Q$ starting from $(x_0,y_0,z_0)$ such that $Q\prec P$, then
$$
\E^{Q}\left[\int_{\T}\dif\tr\mathfrak{R}(t)\right]=\E^P\left[\int_{\T}\dif\tr\mathfrak{R}(t)\right]\textrm{\emph{ for a.e. }}t\in (0,\infty).
$$

\ed

We remark that this definition is consistent with the admissibility condition in the deterministic compressible setting introduced in \cite{BreFeiHof19B,BreFeiHof19}. In these works, it allowed to establish stability of deterministically stationary states, i.e., time independent solutions: if the system reaches
such an equilibrium then it remains there for all times.

\begin{remark}
It is interesting to see what admissibility means for the convex integration solutions from Section~\ref{s:ci}.
Let $T_{L}$ be the stopping time defined in \eqref{stopping time}. Then the solutions $u_{l}=v_{l}+GB_{L}$, $l\in[2,\infty]$, constructed through Theorem~\ref{th:con} satisfy
$$
\E^{\mathbf{P}}\left[\|u_{l}(t)\|_{L^{2}}^{2}\right]=\|u_{0}\|_{L^{2}}^{2}+\left(\frac12 -\frac{1}{l}\right)\E^{\mathbf{P}}\left[(t\wedge T_{L})\right]\|G\|_{L_{2}(U,L^{2})}^{2}.
$$
We observe that the right hand side is minimal when $l=2$.
%
However, we note that apart from the convex integration solutions parametrized by $l\in[2,\infty]$, the same convex integration construction gives raise to other convex integration solution where the defect in the energy is prescribed differently. Such solutions does not have to be comparable by the admissibility criterion. Furthermore, it actually seems that solutions obtained by convex integration as in Section~\ref{s:ci} cannot be admissible.
\end{remark}

\subsection{Selection of strong Markov processes}

Once a suitable notion of solution has been found, the abstract framework of Markov selections can be applied. For  readers' convenience, we recall  some of the key results in this respect presented by Stroock and Varadhan \cite{SV79}.
 First, based on \cite[Lemma 1.3.3, Theorem 1.3.4]{SV79} we obtain a disintegration result, that is, the existence of a regular conditional probability distribution.
 We  use $a(t,\omega)$ to denote $(x(t,\omega),y(t,\omega),z(t,\omega))$ for simplicity.

\bt\label{rcpd}
Given $P\in \mathscr{P}(\Omega_{M})$ and $\tau$ a finite $(\mathcal{B}_{M,t}^0)_{t\geq0}$-stopping time, there exists a regular conditional probability distribution (abbreviated as r.c.p.d.) $P(\cdot|\mathcal{B}_{M,\tau}^0)(\omega)$, $\omega\in\Omega_{M}$, of $P$ with respect to $\mathcal{B}_{M,\tau}^0$ such that
\begin{enumerate}
  \item For every $\omega\in\Omega_{M}$, $P(\cdot|\mathcal{B}_{M,\tau}^0)(\omega)$ is a probability measure on $(\Omega_{M},{\mathcal{B}_M})$.
  \item For every $A\in{\mathcal{B}_M}$, the mapping $\omega\mapsto P(A|\mathcal{B}_{M,\tau}^0)(\omega)$ is $\mathcal{B}^{0}_{M,\tau}$-measurable.
  \item There exists a $P$-null set $N\in\mathcal{B}_{M,\tau}^0$ such that for any $\omega\notin N$
     $$P\left(\{\tilde{\omega};\,\tau(\tilde\omega)=\tau(\omega),a(s,\tilde{\omega})=a(s,\omega),0\leq s\leq \tau(\tilde\omega)\}|\mathcal{B}_{M,\tau}^0\right)(\omega)=1.$$
  \item For any Borel set $A\in \mathcal{B}_{M,\tau}^0$ and any Borel set $B\subset\Omega_M $
  $$P\left(a|_{[0,\tau]}\in A, a|_{[\tau,\infty)}\in B\right)=\int_{\tilde{\omega}\in A}P\left(B|\mathcal{B}_{M,\tau}^0\right)(\tilde\omega)\dif P(\tilde\omega).$$
\end{enumerate}
\et

According to   \cite[Theorem~6.1.2]{SV79} we obtain the following reconstruction result.

\bt\label{th:m2} Let $\tau$ be a finite $(\mathcal{B}_{M,t}^0)_{t\geq0}$-stopping time.
Let $\omega\mapsto Q_\omega$ be a mapping from $\Omega_M$ to $\mathscr{P}(\Omega_M)$ such that for any $A\in {\mathcal{B}_M}$, $\omega\mapsto Q_\omega(A)$ is $\mathcal{B}_{M,\tau}^0$-measurable and for any $\omega\in \Omega_M$
$$Q_\omega\big(\tilde{\omega}\in \Omega_M: a(\tau(\omega),\tilde{\omega})=a(\tau(\omega),\omega)\big)=1.$$
Then for any $P\in \mathscr{P}(\Omega_M)$, there exists a unique $P\otimes_{\tau}Q\in \mathscr{P}(\Omega_M)$ such that
$$(P\otimes_{\tau} Q)(A)=P(A)\  \mbox{for all}\  A\in \mathcal{B}_{M,\tau}^0,$$
and for $P\otimes_{\tau} Q$-almost all $\omega\in \Omega_M$
$$\delta_\omega\otimes_{\tau(\omega)} Q_\omega=(P\otimes_{\tau}Q)\left(\cdot|\mathcal{B}_{M,\tau}^0\right)(\omega).$$

\et

Recall that $\mathbb{X}=\{(x_0,y_0,z_0)\in L^2_\sigma\times \mathcal{M}^+(\mT^3,\mR^{3\times 3}_{\rm{sym}})\times \mR;\,\|x_0\|^2_{L^2}\leq z_0\}$.   We say $P\in \mathscr{P}_{\mathbb{X}}(\Omega_M)\subset \mathscr{P}(\Omega_M)$ is concentrated on the paths with values in $\mX$ if there exists $A\in \mathcal{B}$ with $P(A)=1$ such that $A\subset \{\omega\in\Omega_M;\,\omega(t)\in \mX \ \mbox{for all}\  t\geq0\}$.
It is clear that $\mathcal{B}(\sP_{\mX}(\Omega_M))=\mathcal{B}(\sP(\Omega_M))\cap \sP_{\mX}(\Omega_M)$. We also use $\mathrm{Comp}(\sP_{\mX}(\Omega_M))$ to denote the space of all compact subsets of $\sP_{\mX}(\Omega_M)$ and by $\sC(a)$ we denote the set of dissipative  martingale solutions starting from $a\in \mX$ at time $s=0$.
The shift operator $\Phi_t:\Omega_M\to \Omega_M^t$ is defined by
$$\Phi_t(\omega)(s):=\omega(s-t), \quad s\geq t.$$

Now, we have all in hand to recall the definition of a strong Markov process.

\bd\label{def:ad}
A family $(P_a)_{a\in \mX}$ of probability measures in $\mathscr{P}_{\mX}(\Omega_M)$, is called a strong Markov family provided
\begin{enumerate}
\item  for every $A\in\mathcal{B}$, the mapping $a\mapsto P_a(A)$ is $\mathcal{B}(\mX)/\mathcal{B}([0,1])$-measurable,
\item for every  finite $(\mathcal{B}_{M,t}^0)_{t\geq0}$-stopping time $\tau$, every $a\in\mathbb{X}$ and for $P_a$-a.s. $\omega\in \Omega_M$
$$P_a\left(\cdot|\mathcal{B}_{M,\tau}^0\right)(\omega)=P_{a(\tau(\omega),\omega)}\circ \Phi_{\tau(\omega)}^{-1}.$$
\end{enumerate}
\ed

A strong Markov family can be obtained from a so-called pre-Markov family through a selection procedure.

\bd\label{def:pre}
Let the mapping $\mX\to \mathrm{Comp}(\sP_{\mX}(\Omega_{M}))$, $  a\mapsto \sC(a) ,$ be Borel measurable. We say that $(\sC(a))_{a\in \mX}$ forms a pre-Markov family if for each $a\in \mX$, $P\in \sC(a)$ and every finite $(\mathcal{B}^{0}_{M,t})_{\geq0}$-stopping time $\tau$ there holds
\begin{enumerate}
  \item (Disintegration) there is a $P$-null set $N\in \mathcal{B}_{M,\tau}^0$ such that for $\omega\notin N$,
  $$a(\tau(\omega),\omega)\in \mX, \quad P\left(\Phi_{\tau(\omega)}(\cdot)|\mathcal{B}_{M,\tau}^0\right)(\omega)\in \sC(a(\tau(\omega),\omega)),$$
  \item (Reconstruction) if a mapping $\Omega_{M}\to\sP_{\mX}(\Omega_{M})$, $ \omega\mapsto Q_\omega,$ satisfies the assumptions of Theorem~\ref{th:m2} and there is a $P$-null set $N\in\mathcal{B}_{M,\tau}^0$ such that for all $\omega\notin N$
      $$a(\tau(\omega),\omega)\in \mX,\quad Q_\omega\circ \Phi_{\tau(\omega)}\in \sC(a(\tau(\omega),\omega)),$$
      then $P\otimes_\tau Q\in \sC(a)$.
\end{enumerate}
\ed

Finally, we  recall the following abstract Markov selection theorem, cf.  \cite[Theorem~12.2.3]{SV79}. The generalization to our setting of Polish spaces can be found in \cite{GRZ09}. One difference is that in the definition of the path space $\Omega_{M}$ we need to include $y\in W^{\alpha,q}_{\rm{loc}}([0,\infty);W^{-k,p}(\T,\R^{3\times 3}_{\rm{sym}}))$. It can be checked that the proof \cite[Theorem~12.2.3]{SV79} applies also to this setting.  The only missing point is then the maximization of a given functional, however, this can be achieved easily by choosing  this functional in the selection procedure as the first functional to be maximized.

\bt\label{selection}
Let $(\sC(a))_{a\in \mX}$ be a pre-Markov family. Suppose that for each $a\in\mX$, $\sC(a)$ is non-empty and convex. Then there exists a measurable
selection $\mX\to\sP_{\mX}(\Omega_{M})$, $ a\mapsto P_a ,$ such that $P_{a}\in\sC(a)$ for every $a\in\mathbb{X}$ and $(P_a)_{a\in \mX}$ is a strong Markov family. In addition, if $F:\mathbb{X}\to \R$ is a bounded continuous function and $\lambda>0$ then  the selection can be chosen for every $a\in\mathbb{X}$ to maximize
\begin{equation}\label{eq:func}
\E^{P}\left[\int_{0}^{\infty}e^{-\lambda s}F(x(s),y(s),z(s))\dif s\right]
\end{equation}
among all dissipative martingale solutions $P$ with the initial condition $a$.
\et

\subsection{Application to the stochastic Euler equations}\label{s:6.2}

Our goal is to verify the assumptions of Theorem~\ref{selection} and to choose a suitable function $F$ so that  the Markov selection only contains admissible dissipative martingale solutions. To be more precise, we aim at proving the following.

\bt\label{th:ma1} Suppose that \emph{(G1)}, \emph{(G2)} hold.
The family $(\sC(a))_{a\in \mX}$ admits a measurable strong Markov selection. In other words, there exists a strong Markov family $(P_a)_{a\in \mX}$ solving the stochastic Euler equations \eqref{el}.
Moreover,
 for every $a\in \mX$, the dissipative martingale solution $P_a$ is admissible in the sense of Definition \ref{def:ad}.
\et

\begin{remark}
We note that by choosing a different function $F$ from the one needed in the proof of Theorem~\ref{th:ma1}, we could find a strong Markov selection which is not necessarily admissible. This fact will be explored in Section~\ref{s:nonM} where combined with the non-uniqueness in law from Section~\ref{s:law} it permits to prove non-uniqueness of strong Markov selections.
\end{remark}

In the remainder of this subsection, we prove Theorem~\ref{th:ma1}. First of all, we recall that
Theorem~\ref{convergence1} yields stability of dissipative martingale solutions and as a consequence of Theorem~\ref{existence} we obtain their existence. Thus, we deduce that for every $a\in\mathbb{X}$ the set $\sC(a)$ is a non-empty subset of $\mathrm{Comp}(\mathscr{P}_{\mathbb{X}}(\Omega_{M}))$ and that the mapping $\mathbb{X}\to \mathrm{Comp}(\mathscr{P}_{\mathbb{X}}(\Omega_{M}))$, $a\mapsto \sC(a)$, is Borel measurable. We refer to \cite{FR08} for more details on this step. The convexity of $\sC(a)$ is immediate as the conditions (M1), (M2), (M3) only contain integration with respect to probability measures.

As the next step, we verify  that $(\sC(a))_{a\in \mX}$ has the disintegration as well as the reconstruction property from Definition \ref{def:pre}.

\bl
The family $(\sC(a))_{a\in \mX}$ satisfies the disintegration property in Definition \ref{def:pre}.
\el

\begin{proof}
Fix $a\in \mX$, $P\in \sC(a)$ and a finite $(\mathcal{B}^{0}_{M,t})_{t\geq0}$-stopping time $\tau$. Let $P(\cdot|\mathcal{B}_{M,\tau}^0)(\omega)$ be a r.c.p.d. of $P$ with respect to $\mathcal{B}_{M,\tau}^0$. We want to show that there is a $P$-null set $N\in \mathcal{B}_{M,\tau}^0$ such that for all $\omega\notin N$
$$a(\tau(\omega),\omega)\in \mX,\quad P\left(\Phi_{\tau(\omega)}(\cdot)|\mathcal{B}_{M,\tau}^0\right)(\omega)\in \sC(a(\tau(\omega),\omega)).$$
Due to (M3) and continuity of $x$ we know
$$P\big(\|x(\tau)\|_{L^2}^2\leq z(\tau)\big)=1.$$
Let $N_0=\{\|x(\tau)\|_{L^2}^2\leq z(\tau)\}^c$. Then $P(N_0)=1$ and $N_0\in \mathcal{B}_{M,\tau}^0$.

Next, we shall verify that $P(\Phi_{\tau(\omega)}(\cdot)|\mathcal{B}_{M,\tau}^0)(\omega)$ satisfies the conditions (M1), (M2), (M3) in Definition~\ref{martingale solution1} with the initial condition $a(\tau(\omega),\omega)$ and the initial time $0$ or alternatively that $P(\cdot|\mathcal{B}_{M,\tau}^0)(\omega)$ satisfies (M1), (M2), (M3) in Definition~\ref{martingale solution1} with the initial condition $a(\tau(\omega),\omega)$ and the initial time $\tau(\omega)$.

(M1): Due to (3) from Theorem \ref{rcpd}, it follows that outside of a $P$-null set in $\mathcal{B}^{0}_{M,\tau}$, it holds
$$
P\left(\{ \tilde\omega;\,a(\tau(\omega),\tilde\omega)=a(\tau(\omega),\omega) \}  |\mathcal{B}_{M,\tau}^0\right)(\omega)=1.
$$
In other words, $P(\cdot|\mathcal{B}_{M,\tau}^0)(\omega)$ has the correct initial value at the initial time $\tau(\omega)$.
By using (4) in Theorem~\ref{rcpd} we deduce
\begin{align*}
1&=P\left(\mathfrak{N}\in L^\infty_{\rm{loc}}([0,\infty);\mathcal{M}^+(\mT^3,\mR^{3\times 3}_{\rm{sym}}))\right)\\
&=P\left(\mathfrak{N}\in L^\infty_{\rm{loc}}([0,\tau];\mathcal{M}^+(\mT^3,\mR^{3\times 3}_{\rm{sym}})),\mathfrak{N}\in L^\infty_{\rm{loc}}([\tau,\infty);\mathcal{M}^+(\mT^3,\mR^{3\times 3}_{\rm{sym}}))\right)\\
&=\int_{\omega\in\{\mathfrak{N}\in L^\infty_{\rm{loc}}([0,\tau];\mathcal{M}^+(\mT^3,\mR^{3\times 3}_{\rm{sym}}))\}} P\left(\mathfrak{N}\in L^\infty_{\rm{loc}}([\tau,\infty);\mathcal{M}^+(\mT^3,\mR^{3\times 3}_{\rm{sym}})|\mathcal{B}^{0}_{M,\tau}\right)({\omega})\dif P({\omega}),
\end{align*}
which implies that
$$
P\left(\mathfrak{N}\in L^\infty_{\rm{loc}}([\tau,\infty);\mathcal{M}^+(\mT^3,\mR^{3\times 3}_{\rm{sym}})) |\mathcal{B}^{0}_{M,\tau}\right)=1\quad P\textrm{-a.s.} \quad{\omega}\in \Omega_M.$$
Moreover, the corresponding $P$-null set belongs to $\mathcal{B}^{0}_{M,\tau}$.
We denote the union of the above two  $P$-null sets by $N_1$.

(M2): Using \cite[Theorem 1.2.10]{SV79} (cf. \cite[Lemma B.3]{GRZ09}), there exists a $P$-null set $N_2\in \mathcal{B}^{0}_{M,\tau}$
such that for all ${\omega}\notin N_2$, $P( \Phi_{\tau({\omega})}(\cdot)|\mathcal{B}^{0}_{M,\tau})(\omega)$ satisfies (M2).

(M3): Similarly, there exists a $P$-null set $N_3\in \mathcal{B}^{0}_{M,\tau}$
such that for all ${\omega}\notin N_3$, the process $M^E$ is a continuous $(\mathcal{B}_{M,t}^{0})_{t\geq0}$-martingale under $P(\Phi_{\tau({\omega})}(\cdot)|\mathcal{B}_{M,\tau}^0)(\omega)$,
 and
\begin{align*}
P\left(\int_{\T}\dif\tr\mathfrak{R}(t)\leq z(t) \textrm{ for a.e. } t\in [\tau({\omega}),\infty)|\mathcal{B}_{M,\tau}^0\right)({\omega})=1.
\end{align*}

Recall that under $P$ we have for every $t\geq 0$
\begin{align*}
 z(t)&= z(0)+2\int_0^t\dif M^E_{r}+\int_0^t \|G(x(r))\|_{L_2(U,L^2)}^2\dif r.
 \end{align*}
Now, we split this equation into two parts, i.e.,
\begin{align}
\label{eq:r1}
z(t)&= z(0)+2\int_0^t\dif M^E_{r}+\int_0^t \|G(x(r))\|_{L_2(U,L^2)}^2\dif r,\quad 0\leq t\leq \tau({\omega}),
\end{align}
\begin{align}\label{eq:r2}
z(t)&= z(\tau({\omega}))+2\int_{\tau({\omega})}^t\dif M^E_{r}+\int_{\tau({\omega})}^t \|G(x(r))\|_{L_2(U,L^2)}^2\dif r,\quad \tau({\omega})\leq t<\infty,
\end{align}
and  consider the sets
\begin{align*}
R_{\tau({\omega})}&=\left\{\tilde{\omega}\in \Omega_M;\,\tilde{ \omega}|_{[0,\tau({\omega})]} \text{ satisfies \eqref{eq:r1}}\right\},\\
R^{\tau({\omega})}&=\left\{\tilde{\omega}\in \Omega_M;\, \tilde{\omega}|_{[\tau({\omega}),\infty)} \text{ satisfies \eqref{eq:r2}}\right\}.
\end{align*}
We obtain for $P$
\begin{align*}
1=P\left(R_{\tau}\cap R^{\tau}\right)=\int_{R_\tau}P\left(R^{\tau({\omega})}|\mathcal{B}_{M,\tau}^0\right)(\omega)\dif P(\omega),
\end{align*}
where we used (3) from Theorem \ref{rcpd}.
Consequently, there is a $P$-null set $N_{4}\in\mathcal{B}^{0}_{M,\tau}$ such that for every $\omega\notin N_{4}$ it holds $P(R^{\tau({\omega})}|\mathcal{B}_{M,\tau}^0)(\omega)=1$.
Hence (M3) follows.

We complete the proof by choosing the null set $N=\cup_{i=0}^4N_i$.
\end{proof}

\bl
The family $(\sC(a))_{a\in \mX}$ satisfies the reconstruction property in Definition \ref{def:pre}.
\el

\begin{proof}
Fix $a\in \mX$, $P\in \sC(a)$, a finite $(\mathcal{B}^{0}_{M,t})_{t\geq0}$-stopping time $\tau$ and $Q_\omega$ as in Definition \ref{def:pre} (2). We shall show that  $P\otimes_{\tau}Q\in\sC(a)$ hence we need to verify the conditions (M1), (M2), (M3) from Definition~\ref{martingale solution1}.
The proof follows the lines of  Proposition \ref{prop:2} using  the fact that
$$
\delta_\omega\otimes Q_\omega\big(\omega': \tau(\omega')=\tau(\omega)\big)=1,
$$
since for a $(\mathcal{B}^{0}_{M,t})_{t\geq0}$-stopping time $\tau$, $\{\omega': \tau(\omega')=\tau(\omega)\}\in \mathcal{B}^{0}_{M,\tau(\omega)}$.
\end{proof}

Since we intend  to select those solutions that are admissible, we shall find a suitable functional of the form \eqref{eq:func} which achieves this. By integration by parts formula it holds
\begin{align}\label{eq:C1}
\E^P\left[\int_0^\infty e^{-\lambda s}\left(-\int_{\mT^3} \dif\tr y(s)\right)\dif s\right]=-\frac{1}{\lambda}\int_{\mT^3} \dif\tr y(0)+\frac{1}{\lambda}\E^P\left[\int_0^\infty e^{-\lambda s}\left(- \int_{\T}\dif\tr\mathfrak{R}(s)\right)\dif s\right].
\end{align}
Since the initial datum $y_{0}$ is fixed, the minimization of
$
\E^P\left[\int_{\T}\dif\tr\mathfrak{R}(t)\right]
$
required for the admissibility condition is equivalent to maximization of the  functional on the left hand side of \eqref{eq:C1}. Indeed, if $Q\prec P$ and $\E^P\left[\int_0^\infty e^{-\lambda s}\left( \int_{\T}\dif\tr\mathfrak{R}(s)\right)\dif s\right]$ is the minimal one, we obtain $\E^P \left[ \int_{\T}\dif\tr\mathfrak{R}(s)\right]=\E^Q\left[\int_{\T}\dif\tr\mathfrak{R}(s)\right]$ for a.e. $s\in(0,\infty)$.

Even though the  function
$$F:(x,y,z)(s)\mapsto \left(-\int_{\mT^3} \dif\tr y(s)\right)$$
is not bounded on $\mathbb{X}$, it follows from Lemma~\ref{lem:def} using the definition of $y$ together with (M3) that
$$
\E^{P}\left[\sup_{s\in[0,t]}\int_{\mT^3} \dif\tr y(s)\right]\leq  y_{0}+ (z_{0}+Ct)e^{Ct},
$$
where $C>0$ is the constant determined in the linear growth assumption (G1). In other words, there is a universal choice of $\lambda>0$  in \eqref{eq:C1}, independent of $P$ as well as its initial value, such that the functional in \eqref{eq:C1} is well-defined and can be employed in the proof of Theorem~\ref{selection}.

Finally, we have all in hand to apply Theorem~\ref{selection} which completes the proof of Theorem~\ref{th:ma1}.

\subsection{Non-uniqueness of strong Markov solutions}
\label{s:nonM}

Finally, combining the existence of a strong Markov solution from Theorem~\ref{th:ma1} with the non-uniqueness in law from Section~\ref{s:law} we obtain non-uniqueness of strong Markov solutions.  This holds under the assumptions of Section~\ref{s:law}, in particular, for an additive noise.

\bt\label{th:ma2}
Suppose that \emph{(G3)}, \emph{(G4)} hold. Then strong Markov families associated to the stochastic Euler system \eqref{ela} are not unique. More precisely, apart from the admissible strong Markov family $(P_a)_{a\in\mX}$ constructed in Theorem \ref{th:ma1}, there exists a possibly non-admissible strong Markov family.
\et

\begin{proof}
The proof follows the lines of   \cite[Theorem 12.2.4]{SV79}. In particular, in view of the non-uniqueness in law from Theorem~\ref{Main results2} there exists an initial value $x_{0}\in L^{2}_{\sigma}$ such that for $z_{0}:=\|x_{0}\|_{L^{2}}^{2}$ and for every $y_{0}$ there is infinitely many dissipative martingale solutions with the initial condition $(x_{0},y_{0},z_{0})$ at the time $0$. In particular, there are  dissipative martingale solutions $P$, $Q$ and a function $F:\mathbb{X}\to \R$ such that
$$
\E^{P}\left[\int_{0}^{\infty}e^{-\lambda s}F(x(s),y(s),z(s))\dif s\right]>\E^{Q}\left[\int_{0}^{\infty}e^{-\lambda s}F(x(s),y(s),z(s))\dif s\right].
$$
Now, applying Theorem~\ref{selection} once with $F$ and once with $-F$ we obtain selections $(P^{+}_a)_{a\in\mX}$ and $(P^{-}_a)_{a\in\mX}$, respectively. In particular, for $a=(x_{0},y_{0},z_{0})$ it holds
$$
\E^{P^{+}_{a}}\left[\int_{0}^{\infty}e^{-\lambda s}F(x(s),y(s),z(s))\dif s\right]\geq \E^{P}\left[\int_{0}^{\infty}e^{-\lambda s}F(x(s),y(s),z(s))\dif s\right],
$$
and
$$
\E^{Q}\left[\int_{0}^{\infty}e^{-\lambda s}F(x(s),y(s),z(s))\dif s\right] \geq \E^{P^{-}_{a}}\left[\int_{0}^{\infty}e^{-\lambda s}F(x(s),y(s),z(s))\dif s\right].
$$
In other words, the two strong Markov selections are different. The admissibility cannot be guaranteed since for this it is necessary to choose the first functional as in \eqref{eq:C1}.
\end{proof}

\begin{remark}\label{semiflow}
The approach of Theorem~\ref{th:ma2} translated to the deterministic setting implies non-uniqueness of the semiflow associated to the incompressible Euler equations. To be more precise, existence of a semiflow selection can be proved by the same  procedure as in \cite{BreFeiHof19B,BreFeiHof19}, whereas non-uniqueness of solutions satisfying the energy equality was proved in  \cite{DelSze3}. Thus, by a straightforward modification of the proof of Theorem~\ref{th:ma2} the claim follows. A challenging question which remains open in both deterministic and stochastic setting is whether the selection can be chosen admissible, i.e., whether it satisfies the principle of maximal energy dissipation.
\end{remark}

 \appendix
  \renewcommand{\appendixname}{Appendix~\Alph{section}}
  \renewcommand{\theequation}{A.\arabic{equation}}

\section{Young integration}

We present an auxiliary lemma used in Section~\ref{s:ci}.

\begin{lemma}\label{lem:young}
Let $X$ be a Banach space and  $X^{*}$ be its topological dual. Let $\alpha,\beta\in (0,1)$ be such that $\alpha+\beta>1$. Assume that $f\in C^{\alpha}([0,T];X)$ and $g\in C^{\beta}([0,T];X^{*})$. Then the Young integral
$$
t\mapsto \int_{0}^{t}\langle g_{r},\dif f_{r}\rangle
$$
is well-defined and satisfies
$$
\left\|\int_{0}^{\cdot}\langle g_{r},\dif f_{r}\rangle \right\|_{C^{\alpha}_T}\lesssim \|g\|_{C^{\beta}_{T}X^{*}}\|f\|_{C^{\alpha}X}.
$$
\end{lemma}

\begin{proof}
For $0\leq s\leq\theta\leq  t\leq T$ we use the notation $ f_{s t} = f_t - f_s$  and for a two-index map $\delta h_{s \theta t} = h_{s t} - h_{s \theta} - h_{\theta t}$. We consider the following local approximation of the integral
\begin{equation}\label{eq:yy}
\int_s^t \langle g_{r}, \dif f_{r}\rangle = \langle g_s,  f_{s t}\rangle + g^{\natural}_{s t},
\end{equation}
where $g^{\natural}_{s t}$ is expected to be a sufficiently regular remainder. Then
 we obtain for $\Xi_{st}=\langle g_s,  f_{s t}\rangle$
$$
 | \delta\Xi_{s \theta t} | = | \langle  g_{s \theta},   f_{\theta
   t}\rangle | \leqslant \|g\|_{C^{\beta}_{T}X^{*}}\|f\|_{C^{\alpha}X}
   |t - s |^{\alpha + \beta}.
   $$
Hence the sewing lemma \cite[Lemma~2.2]{DGHT19} yields
$$
 | g^{\natural}_{s t} | \lesssim \|g\|_{C^{\beta}_{T}X^{*}}\|f\|_{C^{\alpha}X}| t - s |^{\alpha + \beta}.
   $$
Thus, \eqref{eq:yy} implies
$$
\sup_{s,t\in[0,T],s\neq t}\frac{\left|\int_{s}^{t}\langle g_{r},\dif f_{r}\rangle \right|}{|t-s|^{\alpha}}\leq \sup_{s,t\in[0,T],s\neq t}\frac{|\langle g_{s},f_{st}\rangle|}{|t-s|^{\alpha}}+ \sup_{s,t\in[0,T],s\neq t}\frac{|g^{\natural}_{s t} |}{|t-s|^{\alpha}}\lesssim \|g\|_{C^{\beta}_{T}X^{*}}\|f\|_{C^{\alpha}X}
$$
and the claim follows.
\end{proof}

\end{document}